\documentclass[reqno,12pt]{amsart}

\usepackage[utf8]{inputenc}
\usepackage{microtype}
\usepackage{amsfonts}
\usepackage{amsmath}
\usepackage{graphicx}
\usepackage{float}
\usepackage[dvipsnames]{xcolor}
\usepackage{hyperref}
\usepackage{tikz-cd}

\usepackage{amssymb}
\usepackage{enumitem}
\usepackage{mathrsfs}

\usepackage{tikz}
\usetikzlibrary{topaths}
\usetikzlibrary{calc}

\usepackage{a4wide}

\usepackage{empheq}

\usepackage[noabbrev,nameinlink,capitalise]{cleveref}
        \crefname{subsection}{Subsection}{Subsections}

\newcommand{\listlabel}[1]{%
  \item[\quad(#1)]\hypertarget{#1}{}\global\expandafter\def\csname #1\endcsname{\hyperlink{#1}{(#1)}}}

\newtheorem{theorem}{Theorem}[section]
\newtheorem*{theorem*}{Theorem}
\newtheorem{lemma}[theorem]{Lemma}
\newtheorem{proposition}[theorem]{Proposition}
\newtheorem*{proposition*}{Proposition}
\newtheorem{corollary}[theorem]{Corollary}
\newtheorem*{corollary*}{Corollary}

\newtheorem{cit}[theorem]{Citation}
\newtheorem*{conjecture*}{Conjecture}

\newtheorem*{claim*}{Claim}

\newtheorem{thmA}{Theorem}

\theoremstyle{definition}
\newtheorem{definition}[theorem]{Definition}
\newtheorem{remark}[theorem]{Remark}
\newtheorem{example}[theorem]{Example}

\newtheorem{notation}[theorem]{Notation}
\newtheorem{question}[theorem]{Question}
\newtheorem*{question*}{Question}

\newcommand{\N}{\mathbb{N}}
\newcommand{\Z}{\mathbb{Z}}
\newcommand{\Q}{\mathbb{Q}}
\newcommand{\R}{\mathbb{R}}

\DeclareMathOperator{\Ab}{Ab}

\DeclareMathOperator{\Aut}{Aut}

\DeclareMathOperator{\Homeo}{Homeo}
\DeclareMathOperator{\PSL}{PSL}

\DeclareMathOperator{\F}{F}

\DeclareMathOperator{\Supp}{Supp}
\DeclareMathOperator{\Sub}{Sub}
\DeclareMathOperator{\Stab}{Stab}
\DeclareMathOperator{\Fix}{Fix}
\DeclareMathOperator{\id}{id}

\DeclareMathOperator{\Hyp}{\mathcal{H}}
\DeclareMathOperator{\Cay}{Cay}
\DeclareMathOperator{\Conf}{Conf}

\DeclareMathOperator{\LL}{LL}

\newcommand{\X}{\mathcal{X}}

\numberwithin{equation}{section}

\begin{document}

\title[Hyperbolic actions of Thompson's group $F$ and generalizations]{Hyperbolic actions of Thompson's group $F$ \\ and generalizations}
\date{\today}


\author[S.~Balasubramanya]{Sahana Balasubramanya}
\address{Department of Mathematical Sciences, Lafayette College, Easton, PA, USA}
\email{hassanba@lafayette.edu}

\author[F.~Fournier-Facio]{Francesco Fournier-Facio}
\address{Department of Pure Mathematics and Mathematical Statistics, University of Cambridge, UK}
\email{ff373@cam.ac.uk}

\author[M.~C.~B.~Zaremsky]{Matthew C.~B.~Zaremsky}
\address{Department of Mathematics and Statistics, University at Albany (SUNY), Albany, NY, USA}
\email{mzaremsky@albany.edu}

\begin{abstract}
We study the poset of hyperbolic structures on Thompson's group $F$ and its generalizations $F_n$ for $n \geq 2$.
The global structure of this poset is as simple as one would expect, with the maximal non-elementary elements being two quasi-parabolic actions corresponding to well-known ascending HNN-extension expressions of $F_n$. However, the local structure turns out to be incredibly rich, in stark contrast with the situation for the $T$ and $V$ counterparts.
We show that the subposet of quasi-parabolic hyperbolic structures consists of two isomorphic posets, each of which contains uncountably many subposets of \emph{lamplike} structures, which can be described combinatorially in terms of certain hyperbolic structures on related lamplighter groups.
Moreover, each of these subposets, as well as intersections and complements thereof, is very large, in that it contains a copy of the power set of the natural numbers.
We also prove that these uncountably many uncountable subposets are not the entire picture, indeed there exists a copy of the power set of the natural numbers consisting entirely of non-lamplike structures.
We also prove that this entire vast array of hyperbolic structures on $F_n$ collapses as soon as one takes a natural semidirect product with $\Z/2\Z$. These results are all proved via a detailed analysis of confining subsets, and along the way we establish a number of fundamental results in the theory of confining subsets of groups.
\end{abstract}

\maketitle
\thispagestyle{empty}

\tableofcontents

\section*{Introduction}

A recent, fruitful line of inquiry in geometric group theory involves classifying all possible cobounded hyperbolic actions of a given group. An action of a group by isometries on a metric space is called \emph{hyperbolic} if the metric space is hyperbolic, in the sense of Gromov. In order to avoid pathologies, it is standard to only consider actions that are \emph{cobounded}, meaning that the space can be covered by $G$-translates of some fixed bounded subset. There is a natural notion of order on such actions, which induces an equivalence relation: the resulting quotient set inherits the structure of a poset, called the \emph{poset of hyperbolic structures}, denoted by $\Hyp(G)$.

This poset was introduced in \cite{ABO}, where its properties for general groups were studied, and it was shown to be typically unwieldy for groups exhibiting negative curvature, specifically acylindrically hyperbolic groups. On the other hand, for certain solvable groups, the poset of hyperbolic structures is more reasonable, and can be computed explicitly \cite{caprace15, sahana:wreath, AR:BS, confining, ABR2}. In another direction, the poset is \emph{trivial} for higher-rank lattices \cite{haettel} and many groups of dynamical origin acting on the circle or the Cantor set \cite{anthony, NL}.

\medskip

Our focus in this paper is another group of dynamical origin, which instead acts on the interval, and shares some of the algebraic properties of the solvable groups whose poset of hyperbolic structure has been previously studied: \emph{Thompson's group $F$}.
This is the group of all piecewise linear orientation preserving homeomorphisms of the interval $[0,1]$ with all slopes powers of 2 and breakpoints in $\Z[1/2]$. This seemingly ad hoc definition belies the rich structure and bizarre properties of $F$; for instance it was the first known torsion-free group to admit a classifying space with every skeleton finite, but not admit any finite classifying space \cite{brown84}. It is also an example of a concrete, straightforward-to-define group for which properties like amenability, $C^*$-simplicity, and automaticity remain open. Background material on $F$ can be found for example in \cite{belk04,cannon96}.

There is an obvious generalization of $F$ to an infinite family of groups $F_n$ ($n\ge 2$), with $F=F_2$, namely, $F_n$ is the group of all piecewise linear orientation preserving homeomorphisms of the interval $[0,1]$ with all slopes powers of $n$ and breakpoints in $\Z[1/n]$. These groups were first officially investigated by Brown in \cite{brown87}, and are the natural ``$F$-like'' version of the Higman--Thompson groups $V_n$ introduced by Higman in \cite{higman74}.

\medskip

Our goal in this paper is to understand the poset of hyperbolic structures on $F_n$. We analyze the structure of this poset from both a ``global'' and a ``local'' perspective, first breaking the poset into large subposets, revealing a global structure, and then inspecting these smaller pieces individually to obtain the local structure. Our first main result shows that the global structure is as simple as one might expect; see \cref{ssec:general_hyp_actions} for the relevant definitions.

\begin{thmA}[Global structure of $\Hyp(F_n)$]
\label{intro:thm:global}
The poset $\Hyp(F_n)$ has the following structure.
\begin{enumerate}
   \item The poset of general type structures $\Hyp_{gt}(F_n)$ is empty.
   \item The poset of quasi-parabolic structures $\Hyp_{qp}(F_n)$ consists of two incomparable isomorphic posets $\X_0$ and $\X_1$.
   \item The poset of lineal structures $\Hyp_{\ell}(F_n)$ coincides with the poset of oriented lineal structures $\Hyp_{\ell}^+(F_n)$, which is an uncountable antichain, with two distinguished elements $[\pm\chi_0]$ and $[\pm\chi_1]$ dominated by $\X_0$ and $\X_1$, respectively.
   \item The poset of elliptic structures $\Hyp_e(F_n)$ is reduced to the smallest element, which is dominated by every lineal structure.
\end{enumerate}
\end{thmA}

See \cref{intro:fig:global} for a picture.

\begin{figure}[htb]
\centering
\begin{tikzpicture}[line width=0.8pt]

\draw (0,2) -- (0,0) -- (4,-2) -- (8,0) -- (8,2);
\draw (2,0) -- (4,-2) -- (4,0)   (4,-2) -- (6,0);

\filldraw[white] (4,-2) circle (8pt);
\draw (4,-2) circle (8pt);
\node at (4,-2) {$*$};
\node at (-2.5,-2) {$\Hyp_e(F_n)$:};

\filldraw[white] (0,0) circle (16pt);
\draw (0,0) circle (16pt);
\node at (0,0) {$[\pm\chi_0]$};

\filldraw[white] (8,0) circle (16pt);
\draw (8,0) circle (16pt);
\node at (8,0) {$[\pm\chi_1]$};
\node at (-2.5,0) {$\Hyp_\ell(F_n)$:};

\filldraw[white] (0,2) circle (18pt);
\draw (0,2) circle (18pt);
\node at (0,2) {$\mathcal{X}_0$};

\filldraw[white] (8,2) circle (18pt);
\draw (8,2) circle (18pt);
\node at (8,2) {$\mathcal{X}_1$};
\node at (-2.5,2) {$\Hyp_{qp}(F_n)$:};

\filldraw[white] (1,1) -- (7,1) -- (7,-0.3) -- (1,-0.3);
\node at (2,0) {$\cdots$};
\node at (4,0) {$\cdots$};
\node at (6,0) {$\cdots$};

\end{tikzpicture}
\caption{The global structure of the poset $\Hyp(F_n)$, as described in \cref{intro:thm:global}.}
\label{intro:fig:global}
\end{figure}
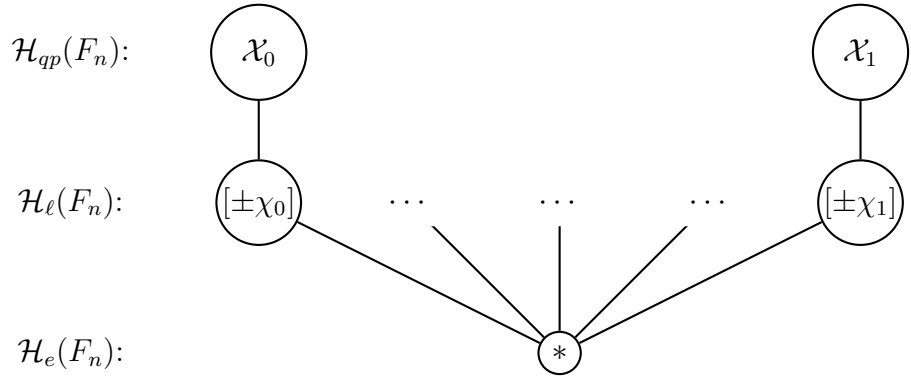

Note that the abelianization of $F_n$ has rank $n$, and in fact the main content of this theorem is the proof that most lineal structures, corresponding to characters of $F_n$, are maximal elements in $\Hyp(F_n)$. In particular there are uncountably many maximal elements.

\medskip

It is well-known that the posets $\X_0$ and $\X_1$ are non-empty. Indeed, $F_n$ can be written as an ascending HNN-extension whose base is the subgroup of elements supported on an interval $[t, 1], t > 0$, and the stable letter is an element $a \in F_n$ with slope $n$ at $0$ such that $t.a^{-k} \to 0$ as $k \to \infty$. The action of $F_n$ on the corresponding Bass--Serre tree is quasi-parabolic, and represents an element of $\X_0$. There is an analogous action belonging to $\X_1$, defined the same way but with an interval of the form $[0, t]$ and an element with slope $n$ at $1$, instead. These two structures have a distinguished place in $\Hyp(F_n)$:

\begin{thmA}
\label{intro:thm:largest}

The two actions of $F_n$ on the Bass--Serre tree for the ascending HNN-extension expressions described above represent the two non-elementary maximal hyperbolic structures on $F_n$.
\end{thmA}

Let us remark that $F_n$ is in a sense very far from being negatively curved. For one, it has no non-abelian free subgroups \cite{brin85}, and it also contains infinite direct sums of $\Z$. Moreover, it has very few quotients, indeed all proper quotients factor through the abelianization \cite{brown87}. From the point of view of its large scale geometry, it has linear divergence \cite{divergence, divergencen}.

However, surprisingly, the actions in \cref{intro:thm:largest} are far from being the only quasi-parabolic structures on $F_n$. Our next goal is to describe the poset $\X$, which is the common isomorphism type of $\X_0$ and $\X_1$, and in particular give an idea of just how large it is. For the next statement, we recall the following terminology involving posets. In a poset $\mathcal{P}$ with partial order $\preccurlyeq$, if $x\preccurlyeq y$ we say that $x$ is \emph{dominated} by $y$. A subset $S$ of $\mathcal{P}$ is called \emph{upper} if every element of $\mathcal{P}$ dominating an element of $S$ lies in $S$. An element of $\mathcal{P}$ is the \emph{largest} element if it dominates every element, and the \emph{smallest} element if it is dominated by every element. If an element of $\mathcal{P}$ is not dominated by any other element it is \emph{maximal}, and if it does not dominate any other element it is \emph{minimal}.

\begin{thmA}[Local structure of $\Hyp(F_n)$]
\label{intro:thm:local}

The poset $\X$ has the following properties, with more precise statements given in the course of the paper.
\begin{enumerate}
	\item There exists a largest element.
	\item $\X$ contains an uncountable family of subposets $\LL_t$, indexed by $t$ from a real interval, called subposets of \emph{lamplike} hyperbolic structures.
	\item Each $\LL_t$ is upper (hence contains the largest element of $\X$), and has a smallest element, which is minimal in $\X$. These minimal elements are pairwise distinct, thus incomparable, for different values of $t$. 
	\item Each $\LL_t$ contains a copy of $\mathcal{P}(\N)$. In fact, for all $t\ne t'$, there are copies of $\mathcal{P}(\N)$ contained in each of $\LL_t \setminus \LL_{t'}$, $\LL_{t'}\setminus \LL_t$, and $\LL_t\cap \LL_{t'}$.
\end{enumerate}
\end{thmA}

See \cref{intro:fig:local} for a picture, and see Definition~\ref{def:lamplike:moving} for the official definition of ``lamplike''. Intuitively, lamplike hyperbolic structures on $F_n$ are inherited from certain hyperbolic structures on lamplighter groups, hence the name.

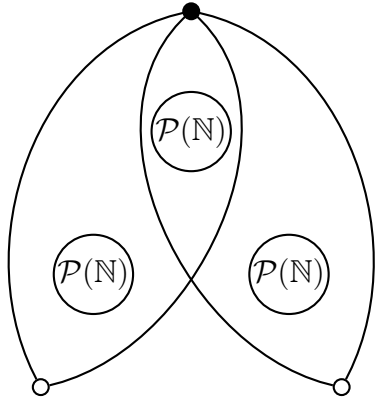
\begin{figure}[htb]
\centering
\begin{tikzpicture}[line width=0.8pt]

\draw (0,0) to[out=190,in=120] (-2,-5);
\draw (0,0) to[out=-40,in=10] (-2,-5);

\draw (0,0) to[out=-10,in=60] (2,-5);
\draw (0,0) to[out=-140,in=170] (2,-5);

\filldraw (0,0) circle (3pt);
\filldraw[white] (-2,-5) circle (3pt);
\draw (-2,-5) circle (3pt);
\filldraw[white] (2,-5) circle (3pt);
\draw (2,-5) circle (3pt);

\draw (-1.3,-3.5) circle (15pt);
\node at (-1.3,-3.5) {$\mathcal{P}(\N)$};
\draw (1.3,-3.5) circle (15pt);
\node at (1.3,-3.5) {$\mathcal{P}(\N)$};
\draw (0,-1.6) circle (15pt);
\node at (0,-1.6) {$\mathcal{P}(\N)$};

\end{tikzpicture}
\caption{The local structure of the poset $\X$, as described in \cref{intro:thm:local}. Two subposets of the form $\LL_t$ are exhibited. The largest element of $\X$, drawn as a black dot, is the largest element of each $\LL_t$. The $\LL_t$ have incomparable smallest elements, drawn as white dots, which are minimal in $\X$. The $\LL_t$ overlap in a way compatible with item (iv).}
\label{intro:fig:local}
\end{figure}

Here $\mathcal{P}(\N)$ denotes the power set of the natural numbers, with order given by inclusion. This is the prototypical example of a large poset of size at most $2^{\aleph_0}$; for instance, it contains chains and antichains of cardinality $2^{\aleph_0}$. Thus the theorem shows that $\X$ is very large and complex. Indeed, the posets $\LL_t$ are already very large, and spread out in some sense: they are upper, and their smallest elements are also minima of $\X$. Moreover, in the case of actions on trees, lamplike structures give a complete picture:

\begin{thmA}
\label{intro:thm:trees}
    Every cobounded quasi-parabolic action of $F_n$ on a simplicial tree corresponds to a lamplike hyperbolic structure, and is represented by the action on the Bass--Serre tree of a strictly ascending HNN-extension expression of $F_n$.
\end{thmA}

Lamplike structures admit an explicit description, which essentially gets rid of all of the dynamics that is intrinsic to $F_n$. Recall that each lamplighter $\Gamma \wr \Z$ has a standard ascending HNN-extension expression, with base group $\bigoplus_{k \geq 0} \Gamma$, and endomorphism given by the shift. This induces an action on the associated Bass--Serre tree, and consequently a quasi-parabolic hyperbolic structure, which we call the \emph{standard hyperbolic structure on $\Gamma \wr \Z$}. The non-elliptic hyperbolic structures on $\Gamma \wr \Z$ that are dominated by the standard one are called \emph{substandard}, and we denote by $\Hyp(\Gamma \wr \Z)^{subst}$ the corresponding subposet. Note that the lineal action coming from the projection to $\Z$ is included in $\Hyp(\Gamma \wr \Z)^{subst}$. In the following statement, $F_n^c$ denotes the compactly supported subgroup of $F_n$, which coincides with the derived subgroup when $n = 2$, but it strictly larger otherwise.

\begin{thmA}[Posets of lamplike structures]
\label{intro:thm:lamplike}

The poset $\LL_t$ is described as follows.
\begin{enumerate}
	\item If $t$ is irrational, then $\LL_t$ is isomorphic to $\Hyp(F_n^c \wr \Z)^{subst}$. 
	\item If $t \in \Z[1/n]$, then $\LL_t$ is isomorphic to $\Hyp(F_n \wr \Z)^{subst}$.
	\item If $t \in \Q \setminus \Z[1/n]$, then there exist retractions $\LL_t \to \Hyp(\Z \wr \Z)^{subst}$ and $\LL_t \to \Hyp(F_n^c \wr \Z)^{subst}$.
\end{enumerate}
\end{thmA}

This theorem allows us to understand $\LL_t$ in terms of a specific subposet of hyperbolic structures on a lamplighter, depending on arithmetic properties of $t$. One should not necessarily view \cref{intro:thm:lamplike} as a complete description of $\LL_t$, but rather as a bridge from the complicated dynamical world of Thompson groups to the more concrete and combinatorial setting of lamplighters. The fact that only the substandard hyperbolic structures play a role here is also interesting, since this indicates that Thompson groups are actually better behaved than lamplighters when it comes to hyperbolic structures. Indeed, while for finite cyclic bases the hyperbolic structures on lamplighters are very limited and can be entirely classified \cite{sahana:wreath}, when the base is infinite the poset of hyperbolic structures outside the subposet of substandard ones can be wild, as we will see in \cref{sec:lamplighters}.

\medskip

This connection to lamplighters is a special feature of the lamplike structures; it turns out that there are other, non-lamplike structures that are intrinsically dynamical:

\begin{thmA}
\label{intro:thm:non_lamp}
There exists a copy of $\mathcal{P}(\N)$ in $\X$ consisting entirely of non-lamplike structures.
\end{thmA}

\medskip

All in all, the structure of $\Hyp(F_n)$ is extremely intricate. This is more reminiscent of the situation for acylindrically hyperbolic groups \cite{ABO}, and is in stark contrast with other groups which one might expect to be analogous to $F_n$: on the one hand with certain finitely presented groups without free subgroups, whose poset of hyperbolic structures is much smaller \cite{sahana:wreath, AR:BS, confining, ABR2}, and on the other hand with the circle and Cantor counterparts of $F_n$ such as Thompson's groups $T$ and $V$ \cite{NL}, which are still finitely presented but have free subgroups.
Even the derived subgroup $F_n'$, which is not finitely generated but is very close to $F_n$ dynamically, has trivial poset of hyperbolic structures \cite{NL}. One might also wonder about the braided analogs of Thompson's groups, as in \cite{brin07,dehornoy06,brady08}, but these are substantially different, by virtue of admitting general-type hyperbolic actions \cite{fff_lodha_zar}.

\medskip

Another group that is very closely related to $F_n$ is a group that we denote by $F_n^{\pm}$. This is defined the same way as $F_n$, but it also allows orientation-reversing homeomorphisms. It splits as a semidirect product $F_n \rtimes \Z/2\Z$. However, its poset of hyperbolic structures is completely different:

\begin{thmA}[Hyperbolic structures on $F_n^{\pm}$]
\label{intro:thm:reverse}
The poset $\Hyp(F_n^{\pm})$ has the following structure.
\begin{enumerate}
   \item The poset of general type structures $\Hyp_{gt}(F_n^{\pm})$ is empty.
   \item The poset of quasi-parabolic structures $\Hyp_{qp}(F_n^{\pm})$ is empty.
   \item The poset of lineal structures $\Hyp_{\ell}(F_n^{\pm})$ is an antichain. The subposet of oriented lineal structures $\Hyp_{\ell}^+(F_n^{\pm})$ has size $1$ for $n = 2$, and is uncountable for $n > 2$. Its complement has size $1$ for $n = 2,3$, and is uncountable for $n > 3$;
   \item The poset of elliptic structures $\Hyp_e(F_n)$ is reduced to the smallest element, which is dominated by every lineal structure.
\end{enumerate}
\end{thmA}

\begin{figure}[htb]
\centering
\begin{tikzpicture}[line width=0.8pt]

\draw (0,0) -- (4,-2) -- (8,0);

\filldraw[white] (4,-2) circle (8pt);
\draw (4,-2) circle (8pt);
\node at (4,-2) {$*$};
\node at (-2.5,-2) {$\Hyp_e(F^{\pm})$:};

\filldraw[white] (0,0) circle (16pt);
\draw (0,0) circle (16pt);
\node at (0,1) {$[\pm(\chi_0 + \chi_1)]$};
\node at (0, 0) {$+$};

\filldraw[white] (8,0) circle (16pt);
\draw (8,0) circle (16pt);
\node at (8, 1) {$[\pm(\chi_0 - \chi_1)]$};
\node at (8, 0) {$-$};
\node at (-2.5,0) {$\Hyp_\ell(F^{\pm})$:};
\end{tikzpicture}

\caption{The poset $\Hyp(F^{\pm})$, as described in \cref{intro:thm:reverse}.}
\label{intro:fig:reverse}
\end{figure}
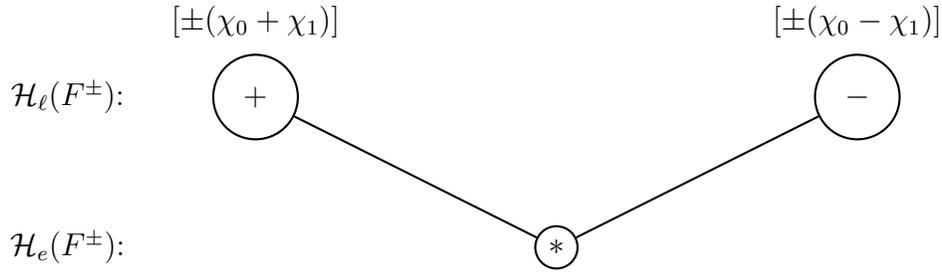

At several points in the paper, the proofs will become more complicated in case $n > 2$, however this is the only result where the structure itself is more complicated for larger $n$. The most interesting case is the one where $n = 2$, pictured in \cref{intro:fig:reverse}: taking a semidirect product by an order-$2$ automorphism, it is possible to make an extremely rich poset of hyperbolic structures collapse to just three elements. It is a major open problem whether this type of behavior can also occur in acylindrically hyperbolic groups \cite[Question 2]{erratum}.

\medskip

In order to prove our results, we first use some general algebraic arguments akin to those of \cite{ABO, confining, NL} to reduce the understanding of $\Hyp(F_n)$ to the understanding of \emph{confining subsets} relative to certain characters. This framework was introduced in \cite{caprace15} for groups of the form $H \rtimes \Z$, and extended in \cite{confining} to groups of the form $H \rtimes \Z^n$, which is our setting: $F_n \cong F_n' \rtimes \Z^n$. Given a character $\rho \colon \Z^n \to \R$, seen as a character of $F_n$, a $\rho$-confining subset is a set $Q \subseteq F_n'$ satisfying some conditions similar to those of the base of an ascending HNN-extension; see Definition~\ref{def:confining} for the precise definition.

In the course of proving our results about $\Hyp(F_n)$, we prove some new, general results about posets of hyperbolic structures and confining subsets that seem to be fundamental to the theory. Let us emphasize some of these here, deferring explanation of the notation and terminology for later. Let $G = H \rtimes \Z^n$ and let $\rho \colon \Z^n \to \R$ be a non-trivial character.
\begin{itemize}
    \item (\cref{cor:conjugate}) Let $Q_1,Q_2\subseteq H$ be $\rho$-confining subsets, and let $a\in \Z^n$ such that $\rho(a)>0$. Then $Q_1\preceq_\rho Q_2$ if and only if $Q_2^{a^k}\subseteq Q_1$ for some $k\ge 0$. (This establishes a simple algebraic description of the poset of confining subsets.)
    \item (\cref{prop:join}) The poset $\Conf^\ast_\rho(G)$ of strictly $\rho$-confining subsets is a join semilattice.
    \item (Propositions \ref{prop:subgroups_to_trees} and \ref{prop:trees_to_subgroups}). A $\rho$-confining subset $Q \subseteq H$ corresponds to a simplicial action on a simplicial tree if and only if $Q$ is equivalent to a subgroup and $\rho$ is discrete. Moreover, up to equivalence, this is the Bass--Serre tree of an ascending HNN-extension expression of $G$.
\end{itemize}

\subsection*{Outline} This paper is organized as follows. In \cref{sec:actions} we recall some background on hyperbolic actions and the poset $\Hyp(G)$ of hyperbolic structures on a group $G$. In \cref{sec:confining} we recall the notion of confining subsets, and also prove some new, general results. In \cref{sec:thompson} we discuss the groups $F_n$, and pin down some group theoretic and dynamical properties that will be useful for understanding confining subsets. We prove our first main result, \cref{intro:thm:global}, in \cref{sec:global}, with the key step being that there are no strictly $\rho$-confining subsets for $\rho \neq \chi_0, \chi_1$, two distinguished characters of $F_n$, corresponding to the slopes at $0$ and $1$ respectively. 

After two detours, first in \cref{sec:reversing} getting a full computation of $\Hyp(F_n^\pm)$, specifically proving \cref{intro:thm:reverse}, and second in \cref{sec:lamplighters} establishing results about confining subsets in lamplighter groups, the rest of the paper consists of a description of $\chi_0$-confining subsets of $F_n$ (and analogously for $\chi_1$). In \cref{sec:largest} we prove \cref{intro:thm:largest} on the two maximal elements of $\Hyp(F_n)$, introduce lamplike structures, and prove \cref{intro:thm:trees}. In \cref{sec:lamplike} and \cref{sec:lamplike:plural} we analyze the lamplike subposets $\LL_t$ of $\Hyp(F_n)$, consisting of those confining subsets that admit $t$ as a global fixed point, and prove \cref{intro:thm:local} and \cref{intro:thm:lamplike} connecting them to lamplighters. In \cref{sec:moving} we prove \cref{intro:thm:non_lamp} about non-lamplike confining subsets. Finally, we conclude in \cref{sec:questions} with some natural questions that emerge.

\subsection*{Acknowledgments} The authors are indebted to Matt Brin for pointing out references, and to Simon Machado for illuminating discussions about confining subsets in lamplighters. They thank Matt Brin, Jack Button, Anthony Genevois, and Henry Wilton for useful comments on a preliminary version. FFF is supported by the Herchel Smith Postdoctoral Fellowship Fund. MZ is supported by grant \#635763 from the Simons Foundation. FFF and MZ thank the organizers of the conference \emph{Topological and Homological Methods in Group Theory} held in Bielefeld in March 2024, where part of this work was done.

\section{Hyperbolic actions}\label{sec:actions}

In this section we recall some background on hyperbolic actions, the poset of hyperbolic structures, and Busemann pseudocharacters.

\subsection{General hyperbolic actions}\label{ssec:general_hyp_actions}

Let us recall the general classification of hyperbolic actions, where by a \emph{hyperbolic action} we mean an action of a group $G$ by isometries on a hyperbolic metric space $X$. An element $g\in G$ is called \begin{enumerate}
    \item \emph{elliptic} if $\langle g\rangle$ has bounded orbits, \item \emph{loxodromic} if the map $n\mapsto g^n x$ is a quasi-isometric embedding from $\Z$ to $X$ for some (equivalently every) $x\in X$, and 
    \item \emph{parabolic} if it is neither elliptic nor loxodromic. 
    \end{enumerate}
    A loxodromic element $g$ fixes precisely two points in the (Gromov) boundary $\partial X$ of $X$, namely the limits at $+\infty$ and $-\infty$ of the sequence $g^n x$ for some (any) $x\in X$. A parabolic element has exactly one fixed point in $\partial X$. 

Let $\Lambda(G)$ denote the set of limit points of $G$ in $\partial X$. More precisely, $\Lambda(G)$ is the intersection of $\partial X$ with the closure in $X\cup\partial X$ of some (equivalently any) $G$-orbit $G.x$. Now hyperbolic actions are classified as follows; this is essentially due to Gromov \cite[Section~8.2]{gromov87}, with the level of generality due to Hamann \cite{hamann17}, and we will use the phrasing from \cite[Theorem~2.7]{confining}.

\begin{cit}
Let $G$ be a group acting by isometries on a hyperbolic space $X$. Then precisely one of the following holds:
\begin{enumerate}
    \item $|\Lambda(G)|=0$, in which case the action is called \emph{elliptic}.
    \item $|\Lambda(G)|=1$, in which case the action is called \emph{parabolic} (or \emph{horocyclic}).
    \item $|\Lambda(G)|=2$, in which case the action is called \emph{lineal}.
    \item $|\Lambda(G)|=\infty$, in which case the action is called \emph{non-elementary}.
\end{enumerate}
\end{cit}

Let us collect some facts about such actions. An action being elliptic is equivalent to having bounded orbits. A parabolic action cannot be cobounded. An action being lineal is equivalent to saying that at least one loxodromic element exists and all loxodromic elements share the same pair of limit points in $\partial X$. We call a lineal action \emph{oriented} if the two limit points are fixed, and \emph{non-oriented} if some element of the group can swap them. Non-elementary actions further split into the following two cases:

\begin{definition}[Quasi-parabolic, general type]
Let $G$ be a group acting non-elementarily on a hyperbolic space $X$. If $G$ fixes a (necessarily unique) point of $\partial X$, or equivalently if all loxodromic elements have one fixed point in $\partial X$ in common, then we call the action \emph{quasi-parabolic} (or \emph{focal}). If $G$ does not fix any points in $\partial X$ then we call the action \emph{of general type}.
\end{definition}

Note that if a group admits a general type action on a hyperbolic space, then a standard ping-pong argument shows that the group contains non-abelian free subgroups.

\subsection{The poset of cobounded hyperbolic actions}\label{ssec:poset}

Given a group, one can inspect all of its cobounded hyperbolic actions, and ask whether there is some structure relating them. In \cite{ABO}, the first author together with Abbott and Osin introduced the poset $\Hyp(G)$ of (equivalence classes of) hyperbolic actions of $G$. Let us recall the construction of $\Hyp(G)$ now.

We will denote an action of $G$ on a set $X$ using the notation $G \curvearrowright X$. Given two cobounded hyperbolic actions $G\curvearrowright X$ and $G\curvearrowright Y$, we say that $G\curvearrowright X$ is \emph{dominated by} $G\curvearrowright Y$, denoted $G\curvearrowright X \preceq G\curvearrowright Y$, if there exists a coarsely Lipschitz map $f\colon Y\to X$ that is coarsely $G$-equivariant with respect to the two actions. Here $f$ is \emph{coarsely Lipschitz} if it distorts distances at most linearly, and \emph{coarsely $G$-equivariant} if for all $x\in X$ there is a uniform bound on the distance from $f(gx)$ to $gf(x)$ for all $g\in G$.

The relation $\preceq$ is clearly reflexive and transitive, so it is a preorder. We thus get an induced equivalence relation on cobounded hyperbolic actions of $G$, with two actions being equivalent if they dominate each other. We denote by $[G\curvearrowright X]$ the equivalence class of $G\curvearrowright X$, and write $\preccurlyeq$ for the partial order induced by $\preceq$ on the set of equivalence classes.

\begin{definition}[The poset $\Hyp(G)$]
We denote by $\Hyp(G)$ the poset of all equivalence classes of cobounded hyperbolic actions of $G$, with the partial order $\preccurlyeq$.
\end{definition}

It turns out that, up to equivalence, every element of $\Hyp(G)$ is represented by an action of $G$ on a (hyperbolic) Cayley graph of $G$, usually with respect to an infinite generating set. Let us set things up in the language of Cayley graphs, and then tie things together. Given any two generating sets $A$ and $B$ for $G$, say that $A$ is \emph{dominated by} $B$, denoted $A\preceq B$, if there is a uniform bound on the word length in $A$ of elements of $B$. Declare that $A$ and $B$ are \emph{equivalent} if they dominate each other. For example, if $G$ is finitely generated then all its finite generating sets are equivalent to each other. Also note that $G$ itself is dominated by all generating sets. Consider equivalence classes $[A]$ of generating sets $A$ of $G$ using this equivalence relation, such that the Cayley graph $\Cay(G,A)$ is hyperbolic. Call such an $[A]$ a \emph{hyperbolic structure on $G$}. The hyperbolic structures on $G$ form a poset, with $[A]\preccurlyeq [B]$ whenever $A\preceq B$. For example, if $G$ is a hyperbolic group then the equivalence class of finite generating sets is the largest element of this poset, and for any $G$ the equivalence class $[G]$ is the smallest element. (We should point out that if $A$ and $B$ are generating sets with $A\subseteq B$ then $[B]\preccurlyeq [A]$, so the order is reversed in some sense.)

The following is immediate from Proposition~3.12 of \cite{ABO} and the discussion after this proposition.

\begin{cit}
\label{cit:cayley:representative}

The poset of hyperbolic structures on $G$ is isomorphic to $\Hyp(G)$, via the poset isomorphism $[A]\mapsto [G\curvearrowright \Cay(G,A)]$.
\end{cit}

We may now view $\Hyp(G)$ either as the poset of equivalence classes of cobounded hyperbolic actions of $G$, or as the poset of hyperbolic structures on $G$, with the two viewpoints being useful for different purposes. We may also refer to a hyperbolic structure as being \emph{represented by} a certain action, with this isomorphism being implicitly understood. 

For a generating set $A$ of $G$ with $\Cay(G,A)$ hyperbolic, since the action of $G$ on $\Cay(G,A)$ is a cobounded hyperbolic action, we know it is either elliptic, lineal, quasi-parabolic, or general type. Also, if $[A]=[B]$ then the actions of $G$ on $\Cay(G,A)$ and $\Cay(G,B)$ have the same type, so $\Hyp(G)$ decomposes as a disjoint union
\[
\Hyp(G) = \Hyp_e(G) \sqcup \Hyp_\ell(G) \sqcup \Hyp_{qp}(G) \sqcup \Hyp_{gt}(G)\text{,}
\]
where these subsets denote the sets of (equivalence classes of) actions that are elliptic, lineal, quasi-parabolic, and general type, respectively. Let us also write $\Hyp_\ell^+(G)$ for the subposet of elements of $\Hyp_\ell(G)$ represented by oriented lineal actions. Clearly $[A]\in\Hyp_e(G)$ if and only if $\Cay(G,A)$ is bounded, and all such $[A]$ equal $[G]$, so this case is uninteresting: we just have $\Hyp_e(G)=\{[G]\}$. Thus when describing $\Hyp(G)$ for a given group, we just need to understand $\Hyp_\ell(G)$, $\Hyp_{qp}(G)$, and $\Hyp_{gt}(G)$.

\subsection{Pseudocharacters}\label{ssec:pseudochars}

In this subsection we recall some background on Busemann pseudocharacters of certain hyperbolic actions.

\begin{definition}[Quasimorphism, pseudocharacter, homogenization]
Let $G$ be a group and $q\colon G\to\R$ a function. We call $q$ a \emph{quasimorphism} if there exists a constant $D$ such that
\[
|q(gh)-q(g)-q(h)|\le D
\]
for all $g,h\in G$. We say that $q$ has \emph{defect} at most $D$. Note that a quasimorphism with defect $0$ is just a homomorphism to $\R$. If a quasimorphism becomes a homomorphism upon restriction to each cyclic subgroup, we call it a \emph{pseudocharacter}. Given any quasimorphism $q$, the function $p(g)\coloneqq \displaystyle \lim_{n\to\infty}\frac{q(g^n)}{n}$ is a pseudocharacter, called the \emph{homogenization} of $q$.
\end{definition}

\begin{remark}
We will only be interested in a specific pseudocharacter, called the Busemann pseudocharacter, defined below. This is one of the few cases where the terminology ``pseudocharacter'' has stuck; in most other contexts, the more common terminology is ``homogeneous quasimorphism.''
\end{remark}

The special case of defect $0$ is important:

\begin{definition}[Character, character sphere]
A quasimorphism with defect $0$, i.e., a homomorphism $G \to \R$, is called a \emph{character} of $G$. Two characters are \emph{equivalent} if they are positive scalar multiples of each other. The set of equivalence classes of non-trivial characters of $G$ forms the \emph{character sphere} $\Sigma(G)$ of $G$. A character is \emph{discrete} if its image is cyclic.
\end{definition}

Most of the groups we will care about here have the following property:

\begin{definition}
We will call a group \emph{agreeable} if every pseudocharacter is a character.
\end{definition}

In fact, a group is agreeable if and only if every element in the commutator subgroup has vanishing stable commutator length, thanks to Bavard duality \cite{bavard91}. One of the most prominent examples of agreeable groups is the following:

\begin{cit}[{\cite[Proposition~2.65]{scl}}]\label{amenable:agreeable}
Amenable groups are agreeable. In particular, abelian groups are agreeable.
\end{cit}

A further example, which is the most relevant for our purposes, is groups of piecewise linear homeomorphisms of the interval \cite{calegari:PL}; we will return to this point later on. The pseudocharacters we will be interested in -- the so called Busemann pseudocharacters -- arise from hyperbolic actions fixing a point in the boundary, as we now describe.

\begin{definition}[Busemann pseudocharacter]
Let $G$ be a group acting on a hyperbolic space $X$. Suppose the action fixes a point $\xi\in\partial X$ in the boundary. Let $(x_i)_{i\in\N}$ be a sequence of points in $X$ converging to $\xi$, and let $x\in X$. The function
\[
g\mapsto \displaystyle \limsup_{n\to\infty} \hspace{3pt} (d(x,x_n) - d(gx,x_n))
\]
is a quasimorphism, and we denote by $p\colon G\to\R$ its homogenization. It turns out that $p$ is independent of the choices of the $x_i$ and $x$, and only depends on $\xi$. We call $p$ the \emph{Busemann pseudocharacter} corresponding to this action of $G$ and the fixed point $\xi$.
\end{definition}

\begin{remark}\label{rem:left_right}
Our definition of the Busemann pseudocharacter here is the negative of the definition in \cite{confining}. The reason is that, throughout the present paper we use right conjugation, whereas in \cite{confining} they use left conjugation, so this is the convention that will make future citations to \cite{confining} accurate. For the record, in \cite{caprace15} the Busemann pseudocharacter is defined in the same way we define it here. Also note that with our convention, $\xi$ is an attracting (rather than repelling) fixed point for elements with positive $p$ value.
\end{remark}

The most important property of the Busemann pseudocharacter is the following:

\begin{cit}[{\cite[Section 4.1]{manning}}]
\label{cit:busemann:loxodromic}
    Let $G$ be a group acting on a hyperbolic space $X$ fixing $\xi \in \partial X$, and let $p$ be the Busemann pseudocharacter. Then $p(g) \neq 0$ if and only if $g$ is loxodromic. Therefore, $p$ is unbounded as soon as $G$ contains a loxodromic element.
\end{cit}

If the action is quasi-parabolic then there is a unique fixed point $\xi$, so we can simply refer to ``the Busemann pseudocharacter'' of the action. If the action is oriented lineal then there are two fixed points, but the resulting Busemann functions are negatives of each other, so again we will just say ``the Busemann pseudocharacter'', with an implicit sign ambiguity.

\subsection{Non-oriented lineal actions}

Let $G$ be a group with a non-oriented lineal action on a hyperbolic space $X$. Since the action on limit points is non-trivial, we cannot define the Busemann pseudocharacter. However, we can still define a meaningful object as follows.

\begin{definition}[Quasicocycle, defect, cocycle]
    Let $G$ be a group, $\varepsilon \colon G \to (\{ \pm 1 \}, \times)$ a homomorphism, and $\varphi \colon G \to \R$ a function. We call $\varphi$ an \emph{$\varepsilon$-quasicocycle} if there exists a constant $D$ such that
    \[ |\varphi(gh) - \varphi(g) - \varepsilon(g)\varphi(h)| \leq D\]
    for all $g, h \in G$. We say that $\varphi$ has \emph{defect} at most $D$. Note that $\varphi$ has defect $0$ if and only if the map $(\varphi, \varepsilon) \colon G \to \R \rtimes \{ \pm 1 \}$ is a homomorphism, in which case we call $\varphi$ an \emph{$\varepsilon$-cocycle}. Note also that if $\varepsilon$ is the trivial homomorphism, then we recover the definition of quasimorphism.
\end{definition}

\begin{cit}[{\cite[Proposition 2.10]{NL}}]
\label{cit:quasicocycle}
    Let $G$ be a group with a non-oriented lineal action on a hyperbolic space $X$ with limit points $\xi_{\pm}$. Let $\varepsilon(g)$ be the sign of the permutation that $G$ induces on $\{ \xi_+, \xi_- \}$. Let $K$ be the kernel of $\varepsilon$. Then the Busemann pseudocharacter $p \colon K \to \R$ extends to an $\varepsilon$-quasicocycle $\varphi \colon G \to \R$ of defect at most twice the defect of $p$. In particular $\varphi$ is an $\varepsilon$-cocycle if and only if $p$ is a character.
\end{cit}

Note that, while $\varepsilon$, $K$, and $p$ (up to sign) are canonically associated to the action, the construction of $\varphi$ requires a choice; this is due to the fact that the splitting $\mathrm{Isom}(\R) \cong \R \rtimes \{ \pm 1 \}$ is not canonical. We can now deduce the following criterion for excluding non-oriented lineal actions.

\begin{corollary}\label{cor:nonoriented}
Let $G$ be a finitely generated agreeable group such that every index-$2$ subgroup of $G$ is agreeable. If $G$ does not surject onto $D_\infty$, then every lineal action of $G$ is oriented.
\end{corollary}

\begin{proof}
Suppose that $G$ admits a non-oriented lineal action on the hyperbolic space $X$. Let $\varepsilon, K, p, \varphi$ be as in \cref{cit:quasicocycle}. Since $K$ is agreeable, $p$ is a character, and therefore $\varphi$ is an $\varepsilon$-cocycle, which gives us a homomorphism $(\varphi, \varepsilon) \ G \to \R \rtimes \{ \pm 1 \}$. Moreover, $K$ contains loxodromics, so $p$ is unbounded by \cref{cit:busemann:loxodromic}, and therefore $(\varphi, \varepsilon)$ has infinite image. Since $G$ is finitely generated, this produces a surjective homomorphism $G \to D_\infty$ (see the proof of \cite[Proposition 3.1]{NL}).
\end{proof}

For studying actions of $F_n^{\pm}$, we will need the following more precise statement.

\begin{proposition}
\label{prop:nonoriented:classification}

Let $G$ be a group, and let $X$ be a hyperbolic space with a non-oriented lineal action. Let $\varepsilon \colon G \to \{ \pm 1 \}$ and $p \colon \ker(\varepsilon) \to \R$ be as in \cref{cit:quasicocycle}. Then $\varepsilon$ and $p$ uniquely determine the equivalence class of $G \curvearrowright X$.
\end{proposition}

\begin{proof}
By the classification of oriented lineal hyperbolic structures in terms of pseudocharacters \cite[Section 2.3]{ABO}, the restricted action of $K \coloneqq \ker(\varepsilon)$ on $X$ is equivalent to the action on $\Cay(K, A)$, where
\[A \coloneqq \{ g \in K \mid |p(g)| \leq C \}\]
for $C > 0$ a large enough constant. By the Svarc--Milnor Lemma for cobounded actions \cite[Lemma 3.11]{ABO}, the action $K \curvearrowright X$ corresponds to the generating set
\(A' = \{ g \in K \mid d_X(x, gx) \leq C' \},\)
where $C' > 0$ is another large enough constant, and $x$ is a point of $X$. Therefore the $A$-length of elements in $A'$ is uniformly bounded, and vice-versa.

Let $s \in G$ be such that $\varepsilon(s) = -1$. Now, using Svarc--Milnor again, the action $K \curvearrowright X$ corresponds to the generating set
\[B' = \{ g \in G \mid d_X(x, gx) \leq C' \};\]
up to taking a large enough constant $C'$, we can assume that $s \in B'$. Since every element of $G$ may be written uniquely as $ks^i$ for $k \in K, i \in \{0, 1\}$, we have
\[A' \cup \{ s \} \subseteq B' \subseteq A' \cup (A')^k s \]
for some large enough $k$. In other words the generating set $B'$ gives the same hyperbolic structure as the generating set $A' \cup \{ s \}$. Moreover, since $A$ and $A'$ are equivalent generating sets of $K$, also $A' \cup \{ s \}$ and $B \coloneqq A \cup \{ s \}$ are equivalent as generating sets of $G$. Since $B$ only depends on $\varepsilon$ and $p$, we conclude.
\end{proof}

\section{Confining subsets}\label{sec:confining}

In \cite{caprace15}, Caprace, Cornulier, Monod, and Tessera introduced the notion of a confining subset in a group of the form $G=H\rtimes \Z$. This is a subset $Q\subseteq H$ such that the conjugation action of $\Z$ on $H$ satisfies certain properties with respect to $Q$, similar to the behavior of an ascending HNN-extension. In \cite{confining}, the first author together with Abbott and Rasmussen generalized this to groups of the form $G=H\rtimes\Z^n$, now with a dependence on a choice of character $\rho\colon \Z^n\to\R$. These confining subsets for groups of the form $G=H\rtimes\Z^n$ turn out to be the key to producing lineal and quasi-parabolic actions of $G$. As a remark, here we will always view semidirect products as internal, so for example in $H\rtimes\Z^n$ the action of $\Z^n$ on $H$ is just the implicitly understood conjugation action. In \cite{caprace15,confining}, semidirect products tend to be external, written $H\rtimes_\gamma\Z^n$ for some $\gamma\colon \Z^n\to\Aut(H)$. We will use the standard notation $x^y\coloneqq y^{-1}xy$ for conjugation; see \cref{rem:left_right}.

As a remark, we will always write semidirect products with the normal subgroup on the left, as is standard, though keep in mind that we tend to use right conjugation actions. Since we always view semidirect products internally, this will never matter.

\subsection{Definitions and properties}\label{ssec:confining_defs}

In this subsection we recall the definition of confining subset, and inspect some important properties.

\begin{definition}[Confining subset]\label{def:confining}
Let $G=H\rtimes \Z^n$ and let $\rho\colon \Z^n \to\R$ be a character. Call a subset $Q\subseteq H$ \emph{confining with respect to $\rho$}, or \emph{$\rho$-confining} if it is symmetric (i.e., $Q = Q^{-1}$) and the following properties hold:
\begin{description}
    \listlabel{StayInConf} For all $z\in\Z^n$ with $\rho(z)\ge 0$ we have $Q^z\subseteq Q$.
    \listlabel{GetInConf} For all $h\in H$ there exists $z\in \Z^n$ such that $h^z\in Q$.
    \listlabel{ProdConf} There exists $z_0\in \Z^n$ such that $(Q\cdot Q)^{z_0}\subseteq Q$.
\end{description}
We call $Q$ \emph{confining} if it is confining with respect to some $\rho$. We say $Q$ is a ``confining subset of'' either $H$ or $G$, depending on context.
\end{definition}

\begin{remark}[First consequences]
\label{rem:conf:consequences}
We note some direct consequences of the definition.
\begin{enumerate}
\item Since $Q$ is symmetric, \ProdConf\ implies that $1\in Q$.
\item If $Q$ is confining, then $G$ is generated by $Q$ together with $\Z^n$, thanks to \GetInConf.
\end{enumerate}
\end{remark}

Call $Q$ \emph{strictly confining} if $Q^z\subsetneq Q$ for some $z\in\Z^n$. Note that if $Q$ is confining but not strictly confining, then necessarily $Q=H$, and indeed $H$ is confining with respect to all $\rho$. If $Q$ is a confining subset that is a subgroup of $H$ we call it a \emph{confining subgroup}; in this case note that \ProdConf\ holds trivially.

\begin{remark}
    \emph{Confining subgroups} should not be confused with the better-studied notion of \emph{confined subgroup} \cite{confined:subgroup}.
\end{remark}

A very useful fact in all that follows is that we actually have quite a bit of control over the conjugating elements in \GetInConf\ and \ProdConf. More precisely, the elements can be drawn from certain cyclic subgroups, as this next result shows.

\begin{lemma}
\label{lem:conf:cyclic}

Let $\rho \colon \Z^n \to \R$ be a character, and let $Q \subseteq H$ be a $\rho$-confining subset. Let $a \in \Z^n$ be any element such that $\rho(a) > 0$. Then:
\begin{enumerate}
\item For all $h \in H$ there exists $k \geq 0$ such that $h^{a^k} \in Q$;
\item There exists $k \geq 0$ such that $(Q \cdot Q)^{a^k} \subseteq Q$.
\end{enumerate}
\end{lemma}

\begin{proof}
For the first item, \GetInConf\ gives an element $z \in \Z^n$ such that $h^z \in Q$. Let $k \geq 0$ be large enough that $\rho(a^k) \geq \rho(z)$. Then $h^{a^k} = (h^z)^{a^k z^{-1}} \in Q^{a^k z^{-1}} \subseteq Q$, where in the last step we used \StayInConf. The second item is proved similarly.
\end{proof}

Another useful consequence of the definitions is the following strengthening of \ProdConf.

\begin{lemma}
\label{lem:products}

Let $\rho \colon \Z^n \to \R$ be a character, and let $Q \subseteq H$ be a $\rho$-confining subset. Let $z_0 \in \Z^n$ be given by \ProdConf, and suppose that $\rho(z_0) \geq 0$. Then $(Q^k)^{z_0^{k-1}} \subseteq Q$.
\end{lemma}

Note that by \cref{lem:conf:cyclic}, this shows that the conjugating element taking $Q^k$ into $Q$ can always be chosen to be a power of a given $a \in \Z^n$ such that $\rho(a) > 0$.

\begin{proof}
We proceed by induction on $k$, the base case $k=1$ being trivially true. For $k\ge 2$ we have $(Q^k)^{z_0^{k-1}} = ((Q^{k-1})^{z_0^{k-2}} \cdot Q^{z_0^{k-2}})^{z_0}$. By induction $(Q^{k-1})^{z_0^{k-2}}\subseteq Q$, and by \StayInConf\ $Q^{z_0^{k-2}}\subseteq Q$. Thus, $(Q^k)^{z_0^{k-1}}\subseteq (Q \cdot Q)^{z_0}$, which is contained in $Q$ by \ProdConf.
\end{proof}

\medskip

Next we would like a way to compare two confining subsets. Given a character $\rho\colon \Z^n\to\R$, we define a subset $Z_\rho$ of $\Z^n$ as follows. Fix a generating set $\{a_1,\dots,a_n\}$ for $\Z^n$, and a constant $C_\rho>0$ such that $|\rho(a_i)| \leq C_\rho$ for all $1\le i\le n$ and such that there exists $z_{top}\in \Z^n$ with $\rho(z_{top})=C_\rho$. Now set
\[
Z_\rho \coloneqq \{z\in\Z^n \mid |\rho(z)|\le C_\rho\}\text{,}
\]
so in particular $Z_\rho$ generates $\Z^n$. (Note that for $n\ge 2$ this is an infinite generating set.) As a remark, the Cayley graph of $\Z^n$ with respect to the generating set $Z_\rho$ is quasi-isometric to a line, and the action of $\Z^n$ on this hyperbolic space is oriented lineal. We also have that for $G=H\rtimes \Z^n$, if $Q\subseteq H$ is $\rho$-confining for some $\rho\colon \Z^n\to\R$, then $Q\cup Z_\rho$ is a symmetric generating set for $G$. Choosing a different finite generating set for $\Z^n$ and/or a different constant $C_\rho$ can change the generating set $Q\cup Z_\rho$ for $G$, but cannot change it up to the equivalence relation on generating sets of $G$, so up to equivalence we are always free to choose the $a_i$ and $C_\rho$ as we see fit. Also note that we can always assume $Z_{-\rho}=Z_\rho$.

Consider the following relation on confining subsets. Given $G=H\rtimes \Z^n$ and $\rho\colon\Z^n\to\R$ as before, if $Q_1,Q_2\subseteq H$ are $\rho$-confining subsets, we write
\[
Q_1\preceq_\rho Q_2
\]
if $Q_2$ is bounded in the alphabet $Q_1\cup Z_\rho$. Note that $Q_1\preceq_\rho Q_2$ for $\rho$-confining subsets is clearly equivalent to $Q_1\cup Z_\rho\preceq Q_2\cup Z_\rho$ for generating sets. Since $\preceq_\rho$ is a preorder, we get an equivalence relation $\sim_\rho$ on $\rho$-confining subsets and an induced partial order $\preccurlyeq_\rho$ on equivalence classes $[Q]$ of $\rho$-confining subsets $Q$.

\begin{definition}[The posets $\Conf_\rho(G)$ and $\Conf^\ast_\rho(G)$]\label{def:conf_poset}
The poset of equivalence classes of $\rho$-confining subsets with the partial order $\preccurlyeq_\rho$ is denoted by $\Conf_\rho(G)$. The subposet of \emph{strictly} $\rho$-confining subsets is denoted by $\Conf^\ast_\rho(G)$. Note that $\Conf_\rho(G)$ is a cone over $\Conf^\ast_\rho(G)$, with cone point the equivalence class of $H$.
\end{definition}

If $Q_1 \subseteq Q_2$ then clearly $Q_2 \preceq_\rho Q_1$. More generally, it turns out that if $Q_1^z \subseteq Q_2$ for some $z\in \Z^n$ then $Q_2 \preceq_\rho Q_1$, and, perhaps surprisingly, the converse of this is true. This will be an easy consequence of the following characterization of boundedness in the alphabet $Q\cup Z_\rho$ for confining $Q$.

\begin{proposition}\label{prop:conj_is_all}
Let $Q\subseteq H$ be a $\rho$-confining subset. Then a subset $B\subseteq H$ is bounded in the alphabet $Q\cup Z_\rho$ if and only if $B \subseteq Q^z$ for some $z \in \Z^n$.
\end{proposition}

\begin{proof}
First suppose $B \subseteq Q^z$. Let $m$ be the length of $z$ in the alphabet $Q\cup Z_\rho$. Now observe that every element of $B$ can be written as $z^{-1}qz$ for some $q\in Q$, and this has length $2m+1$ in the alphabet $Q\cup Z_\rho$.

Now suppose that $B$ is bounded in the alphabet $Q\cup Z_\rho$. Say every element of $B$ can be written as a product of at most $m$ elements of $Q\cup Z_\rho$. In order to only have to deal with a single form, note that this implies every element $b$ of $B$ can be written in the form $b=q_1 z_1\cdots q_m z_m$ for some $q_i\in Q$, $z_i\in Z_\rho$ (with several $q_i$ and/or $z_i$ possibly being the identity).
Now fix $z_{top}\in Z_\rho$ with $\rho(z_{top})=C_\rho$, as in the construction of $Z_\rho$. For any $z\in Z_\rho$, since $|\rho(z)|\le C_\rho$ we have $\rho(z_{top} z)\ge 0$. In particular, $Q^{z_{top} z}\subseteq Q$ for all $z\in Z_\rho$ by \StayInConf, so $Q^z\subseteq Q^{z_{top}^{-1}}$.
More generally, for any $k \in \Z$ we have $(Q^{z_{top}^{-k}})^z\subseteq Q^{z_{top}^{-k-1}}$ for all $z\in Z_\rho$, as can be seen by conjugating both sides of this expression by $z_{top}^{k+1}$ by \StayInConf.

Now we return to our arbitrary $b\in B$ with $b=q_1 z_1\cdots q_m z_m$. Since each $z_i$ is in $Z_\rho$, the above allows us to sequentially move all the $z_i$ to the left, yielding $b=z_1\cdots z_m q_1'\cdots q_m'$, with $q_m'\in Q^{z_{top}^{-1}}$, $q_{m-1}'\in Q^{z_{top}^{-2}}$, and so forth up to $q_1'\in Q^{z_{top}^{-m}}$. Since $b\in H$, we must have $z_1\cdots z_m=1$, and since $Q^{z_{top}^{-k}}\subseteq Q^{z_{top}^{-m}}$ for all $0\le k\le m$ by \StayInConf, we get $b\in (Q^m)^{z_{top}^{-m}}$. Since $b$ was arbitrary, this implies that $B\subseteq (Q^m)^{z_{top}^{-m}}$. By \cref{lem:products} we can choose $z' \in \Z^n$ such that $(Q^m)^{z'}\subseteq Q$, and so we conclude that $B\subseteq Q^{z^{-1}}$ where $z =  z_{top}^mz' \in \Z^n$.
\end{proof}

\begin{corollary}[Characterization of $\preceq_\rho$]\label{cor:conjugate}
Let $Q_1,Q_2\subseteq H$ be $\rho$-confining subsets, and fix $a \in \Z^n$ such that $\rho(a) > 0$. Then $Q_1 \preceq_\rho Q_2$ if and only if $Q_2^{a^k} \subseteq Q_1$ for some $k \geq 0$.
\end{corollary}

\begin{proof}
By \cref{prop:conj_is_all}, we know that $Q_1 \preceq_\rho Q_2$ if and only if $Q_2^z \subseteq Q_1$ for some $z \in \Z^n$. This gives one direction; for the other one, let $k \geq 0$ be large enough that $\rho(a^k) \geq \rho(z)$. Then
\[Q_2^{a^k} = (Q_2^z)^{a^k z^{-1}} \subseteq Q_1^{a^k z^{-1}} \subseteq Q_1,\]
thanks to \StayInConf, since $\rho(a^k z^{-1}) \geq 0$.
\end{proof}

To the best of our knowledge, \cref{cor:conjugate} has not appeared before in the literature, but it encapsulates a very natural and straightforward characterization of the equivalence relation on confining subsets, and the partial order on their equivalence classes. In particular, this characterization facilitates the proof of a structural property of the poset $\Conf^\ast_\rho(G)$. Recall that the \emph{join} of two elements in a poset is their least upper bound, and a poset is a \emph{join semilattice} if every two elements admit a (necessarily unique) join.

\begin{proposition}
\label{prop:join}

Let $G=H\rtimes \Z^n$ and let $\rho\colon \Z^n \to\R$ be a character. The poset $\Conf^\ast_\rho(G)$ is a join semilattice. Given $\rho$-confining subsets $Q_1, Q_2 \subseteq H$, the join of $[Q_1]$ and $[Q_2]$ is $[Q_1 \cap Q_2]$.
\end{proposition}

\begin{proof}
First we show that $Q_1 \cap Q_2$ is also a $\rho$-confining subset. For $z \in \Z^n$ with $\rho(z) \geq 0$ we have $(Q_1 \cap Q_2)^z = Q_1^z \cap Q_2^z \subseteq Q_1 \cap Q_2$, which proves \StayInConf.
Let $h \in H$ and let $z_i$ ($i=1,2$) be such that $h^{z_i} \in Q_i$. By \cref{lem:conf:cyclic} we may choose $z \in \Z^n$ that works as both $z_1$ and $z_2$. Then $h^z \in Q_1 \cap Q_2$, which proves \GetInConf.
Finally, let $z_i$ ($i=1,2$) be such that $(Q_i \cdot Q_i)^{z_i} \in Q_i$. Again, we may choose $z \in \Z^n$ that works as both $z_1$ and $z_2$. Then
\[
((Q_1 \cap Q_2) \cdot (Q_1 \cap Q_2))^{z} \subseteq ((Q_1 \cdot Q_1) \cap (Q_2 \cdot Q_2))^{z} \subseteq Q_1 \cap Q_2,
\]
which proves \ProdConf.

Clearly $[Q_1 \cap Q_2]$ is a common upper bound of $[Q_1]$ and $[Q_2]$. Let $R$ be a $\rho$-confining subset such that $[R]$ is another common upper bound. By \cref{cor:conjugate} there exist $z_1, z_2 \in \Z^n$ such that $R \subseteq Q_1^{z_1} \cap Q_2^{z_2}$. Suppose without loss of generality that $\rho(z_1) \geq \rho(z_2)$. Then $\rho(z_1 z_2^{-1}) \geq 0$ and so \[Q_1^{z_1} = (Q_1^{z_1 z_2^{-1}})^{z_2} \subseteq Q_1^{z_2}.\]
This shows that $R \subseteq (Q_1 \cap Q_2)^{z_2}$ and so by \cref{cor:conjugate} again, $[Q_1 \cap Q_2] \preccurlyeq_\rho [R]$.
\end{proof}

\begin{remark}
In some cases that have been considered before, $\Conf^\ast_\rho(G)$ is also a meet semi-lattice, and thus a lattice \cite{sahana:wreath, AR:BS, ABR2}. However this is not true in general, and we will see that the groups $F_n$ provide an example (\cref{rem:meet}).
\end{remark}

In practice, the confining subsets one tends to encounter in the literature are always equivalent to a confining subgroup, see for example \cite{caprace15, sahana:wreath, AR:BS, ABR2}. This can be characterized as follows:

\begin{proposition}\label{prop:equiv_to_subgroup}
Let $G=H\rtimes \Z^n$ and let $\rho\colon \Z^n \to\R$ be a character. Let $Q\subseteq H$ be a $\rho$-confining subset. Then the following are equivalent:
\begin{enumerate}
    \item $Q$ is equivalent to a $\rho$-confining subgroup.
    \item $Q$ is equivalent to $\langle Q\rangle$.
    \item There exists $z\in \Z^n$ such that $\langle Q\rangle^z\subseteq Q$.
\end{enumerate}
\end{proposition}

\begin{proof}
It follows easily from the definitions that if $Q$ is a confining subset, then $\langle Q \rangle$ is a confining subgroup. Clearly (ii) implies (i), and \cref{cor:conjugate} shows that (ii) and (iii) are equivalent, so it suffices to prove that (i) implies (iii). If $Q$ is equivalent to a $\rho$-confining subgroup $R$, then by \cref{cor:conjugate} we have $Q^z\subseteq R$ and $R^{z'}\subseteq Q$ for some $z,z'\in\Z^n$. Since $R$ is a subgroup we have $\langle Q\rangle^z\le R$, so $\langle Q\rangle^{zz'}\subseteq Q$ and we are done.
\end{proof}

We will see however that lamplighters with infinite base group admit (many) confining subsets that are not confining subgroups (\cref{ex:balls}, \cref{ex:machado}), and we will see more examples of this among non-lamplike confining subsets of $F_n$ (\cref{sec:moving}).

\subsection{Confining subsets and hyperbolic actions}\label{ssec:conf_hyp}

We first make the following observation about pseudocharacters of groups of the form $H \rtimes \Z^n$:

\begin{lemma}
\label{lem:factor}
Let $p \colon H \rtimes \Z^n \to \R$ be a pseudocharacter. If $p(H) = 0$, then $p$ factors as $H \rtimes \Z^n \to \Z^n \xrightarrow{\rho} \R$ for some character $\rho$ of $\Z^n$.
\end{lemma}

\begin{proof}
Since $p(H) = 0$, the left exactness of the pseudocharacter functor \cite[Remark 2.90]{scl} implies that $p$ factors through a pseudocharacter $\rho \colon \Z^n \to \R$. But $\Z^n$ is agreeable (\cref{amenable:agreeable}), so $\rho$ is actually a character.
\end{proof}

Now we spell out the connection between confining subsets and hyperbolic actions.

\begin{cit}[Confining subsets to hyperbolic actions {\cite[Theorem~1.1]{confining}}]\label{cit:confining_to_action}
Let $G=H\rtimes \Z^n$, and let $\rho\colon \Z^n\to \R$ be a non-trivial character. If $Q\subseteq H$ is confining with respect to $\rho$, then $\Cay(G,Q\cup Z_\rho)$ is hyperbolic. Moreover, if $Q$ is strictly confining then $[Q\cup Z_\rho]\in \Hyp_{qp}(G)$, and if $Q$ is not strictly confining then $[Q\cup Z_\rho]\in \Hyp_\ell^+(G)$. Finally, the Busemann pseudocharacter for $G\curvearrowright \Cay(G,Q\cup Z_\rho)$ sends $H$ to $0$, so induces a character $\Z^n\to\R$, which is equivalent to $\rho$ in the quasi-parabolic case and to $\pm\rho$ in the oriented lineal case.
\end{cit}

In \cite[Theorem~1.1]{confining} the last claim in the quasi-parabolic case is that the character is ``proportional'' to $\rho$, but the proof of \cite[Lemma~1.1]{confining} makes clear that it is in fact equivalent to $\rho$, meaning the constant of proportionality is positive. Also note that the Busemann pseudocharacter for $G\curvearrowright \Cay(G,Q\cup Z_\rho)$ is in fact a character.

\begin{cit}[Hyperbolic actions to confining subsets {\cite[Theorem~3.17]{confining}}]
\label{cit:action_to_confining}
Let $G=H\rtimes \Z^n$, and suppose $G$ acts isometrically on a hyperbolic space $X$. Assume the action is cobounded, and either oriented lineal or quasi-parabolic. Let $p\colon G\to\R$ be the Busemann pseudocharacter associated to this action, and assume that $p(H)=0$, so $p$ induces a (non-trivial) character $\rho\colon \Z^n\to\R$. Then there exists a subset $Q\subseteq H$ that is confining with respect to $\rho$, such that $X$ is $G$-equivariantly quasi-isometric to the Cayley graph $\Cay(G,Q\cup Z_\rho)$. If the action is oriented lineal then $Q=H$, and if the action is quasi-parabolic then $Q$ is strictly confining.
\end{cit}

Note that we have specified that the characters $\rho$ used above are non-trivial. If $\rho=0$ then $Z_0=\Z^n$, and $Q=H$ is the only confining subset with respect to $\rho=0$, so $[Q\cup Z_0]=[Q\cup \Z^n]=[G]$. In other words, $\rho=0$ corresponds precisely to $\Hyp_e(G)=\{[G]\}$, and this is the end of the story for the case when $\rho$ is trivial.

\medskip

Before we begin imposing conditions on our groups, let us prove two more general results, which are specializations of \cref{cit:confining_to_action} and \cref{cit:action_to_confining} that establish a relationship between confining subgroups and actions on trees (\emph{all trees are assumed to be simplicial}). For the next two proofs, we assume some familiarity with Bass--Serre theory \cite{serre}.

\begin{proposition}[Confining subgroups yield trees]\label{prop:subgroups_to_trees}
Let $G=H\rtimes \Z^n$, and let $\rho\colon \Z^n\to \R$ be a (non-trivial) discrete character. Let $Q$ be a $\rho$-confining subset that is a subgroup. Then the hyperbolic structure $[Q\cup Z_\rho]$ arising from \cref{cit:confining_to_action} is represented by a simplicial action on a tree.

More precisely, it is represented by the action on the Bass--Serre tree of an ascending HNN-extension expression of $G$, with base group $Q \ker(\rho)$, and stable letter an element $a \in \Z^n$ such that $\rho(a)$ is positive and minimal.
\end{proposition}

\begin{proof}
Since the character is discrete, we may assume that it is a surjection $\rho \colon \Z^n \to \Z$. Choose a basis $\{a_0,\dots,a_{n-1}\}$ for $\Z^n$ such that $\rho(a_0)=1$ and $\rho(a_i)=0$ for all $1\le i\le n-1$. We are free to choose $C_\rho=1$, so that $Z_\rho=\{a_0^{k_0}\cdots a_{n-1}^{k_{n-1}}\mid k_0\in\{-1,0,1\}\}$. Let $K\coloneqq Q\langle a_1,\dots,a_{n-1}\rangle$; since $\rho(a_i)=0$ for all $1\le i\le n-1$, each such $a_i$ normalizes $Q$ by \StayInConf, and since we are assuming that $Q$ is a subgroup of $G$, we get that $K$ is a subgroup of $G$.

We claim that $G$ decomposes as an ascending HNN-extension with base $K$ and stable letter $a_0$. For this we need to check the three conditions from \cite[Lemma 3.1]{ascending}. First, $K^{a_0}\le K$ since $Q^{a_0}\le Q$ by \StayInConf. Second, $G$ is generated by $K\cup\{a_0\}$, since $Q\le K$, $Z_\rho\subseteq \{1,a_0,a_0^{-1}\}K$, and $G$ is generated by $Q\cup Z_\rho$. Finally, $\langle a_0\rangle\cap K=\{1\}$ since $\rho(a_0)=1$ and $K$ lies in the kernel of the composition of $\rho$ with the projection $G\to\Z^n$. This concludes the proof of the claim.

Let $X$ be the Bass--Serre tree of this ascending HNN-extension expression of $G$, and let $v \in V(X)$ be the vertex whose stabilizer is $K$. The action of $G$ on $X$ has a fundamental domain of diameter $1$, so by the Svarc--Milnor Lemma for cobounded actions \cite[Lemma 3.11]{ABO}, this action corresponds to the hyperbolic structure with generating set
\(B = \{ g \in G \mid d(v, gv) \leq 3 \}.\)
We note that $d_T(v, gv) \leq 1$ for every $g \in A = Q \cup Z_\rho$, so $A \subseteq B$. By Bass--Serre theory, every element of $B$ can be written as a product of powers of $a_0$ and elements in $K$, and the length of this product is bounded uniformly because the distance from $v$ to $gv$ for $g\in B$ is bounded uniformly by $3$. It follows that $A$ and $B$ are generating sets representing the same hyperbolic structure, which is the one coming from the action of $G$ on this Bass--Serre tree.
\end{proof}

\begin{proposition}[Trees yield confining subgroups] \label{prop:trees_to_subgroups}
Let $G=H\rtimes \Z^n$, and suppose $G$ acts simplicially on a tree $X$. Assume the action is cobounded, and either oriented lineal or quasi-parabolic. Assume the Busemann pseudocharacter associated to this action sends $H$ to $0$, and so induces a (non-trivial) character $\rho\colon \Z^n\to\R$. Then $\rho$ is discrete and the $\rho$-confining subset arising from \cref{cit:action_to_confining} can be chosen to be a subgroup.
\end{proposition}

\begin{proof}
Since the Busemann pseudocharacter sends $H$ to $0$, every element of $H$ is elliptic, as trees do not admit parabolic isometries. The restriction to the $\Z^n$ retract is an action of an abelian group on a tree that admits a loxodromic element, so it must be oriented lineal (see e.g. \cite[Proposition 3.2]{NL}). Since the action is on a tree, there exists a $\Z^n$-invariant line, which we denote by $L$, and the image of $\rho$ is equal to $\langle \ell \rangle$, where $\ell \in \N$ is the minimal positive translation length of an element of $\Z^n$. Because the action is simplicial, $\rho$ is discrete. 
If the action is oriented lineal then our confining subset is $H$, which is a subgroup and so we are done.

Now assume the action is quasi-parabolic. Let $\xi$ be the unique $G$-fixed point in the boundary of $X$, and note that the $\Z^n$-invariant line $L$ has $\xi$ as one of its endpoints. By the definition of the Busemann pseudocharacter, if $z \in \Z^n$ satisfies $\rho(z) > 0$, then $z^n x \xrightarrow{n \to \infty} \xi$ for all $x \in X$ (Remark \ref{rem:left_right}). Choose $z_0\in\Z^n$ such that $\rho(z_0) = \ell > 0$. Since the action of $G$ on $X$ is cobounded, we can choose a bounded fundamental domain for the action, say with vertex set $D$, and without loss of generality we can assume that $v_0$ lies in $L$ for some $v_0\in D$. For any $v$ in $X$, write $[v,\xi)$ for the unique ray from $v$ to $\xi$.

Since $D$ is bounded, the geodesic projection of $D$ onto $L$ is also bounded, and so we can choose a vertex $w$ of $L$ such that $w$ lies in $[v,\xi)$ for all $v\in D$. Since $\xi$ is fixed by all of $G$, we see that $w$ is fixed by $\Stab_G(v)$ for all $v\in D$. Set $Q=\Stab_H(w)$, and we claim that $Q$ is a $\rho$-confining subgroup. Being a subgroup, $Q$ is symmetric and satisfies \ProdConf. To verify \StayInConf, note that for $z\in \Z^n$ with $\rho(z)\ge 0$ and $q\in Q$, we have that $zw$ lies in $[w,\xi)$, and so $q(zw)=zw$, thus confirming that $z^{-1}qz\in Q$. For \GetInConf, note that for any $h\in H$ the lines $L$ and $h(L)$ must contain a common ray to $\xi$, which is fixed pointwise by $h$ because $\rho(h) = 0$; in particular $h$ fixes some vertex $w'$ of $L$ lying in $[w,\xi)$. Up to choosing $w'$ closer to $\xi$, there exists $k \geq 0$ such that $z_0^k w = w'$, and hence $z_0^{-k}hz_0^k\in Q$.

It remains to prove that $Q$ corresponds to $G\curvearrowright X$ under the correspondence from \cref{cit:action_to_confining}. For this, we need to show that the generating set $Q \cup Z_\rho$, where $Z_\rho = \{ z \in \Z^n \mid |\rho(z)| \leq \ell \}$, is equivalent to the generating set given by the Svarc--Milnor Lemma for cobounded actions \cite[Lemma 3.11]{ABO}, namely $B \coloneqq \{ g \in G \mid d(v_0, gv_0) \leq 3d \}$, where $d$ is the diameter of our bounded fundamental domain with vertex set $D$. Since all values of $\rho(G)$ are attained by $\langle z_0 \rangle$, by Bass--Serre theory the generating set $B$ is equivalent to the generating set
\[A \coloneqq \{ z_0^{\pm 1} \} \cup \bigcup_{v\in D} \Stab_G(v).\]
Next, note that $\Stab_G(w) = \Stab_H(w) \Stab_{\Z^n}(w) = Q \ker(\rho)$. Since $\ker(\rho) \subseteq Z_\rho$, we see that the generating set $Q \cup Z_\rho$ is equivalent to the generating set
\[A' \coloneqq \{z_0^{\pm 1}\}\cup \Stab_G(w).\]

We are left to prove that $[A] = [A']$. On the one hand, $\Stab_G(v) \leq \Stab_G(w)$ for all $v\in D$, and so $A \subseteq A'$. In the other direction, choose $k \geq 0$ such that $w' \coloneqq z_0^k v_0$ lies in $[w,\xi)$; note that this is possible since $v_0$ was assumed to be in $L$. Then
\[z_0^{-k} \Stab_G(w) z_0^k \subseteq z_0^{-k} \Stab_G(w') z_0^k = \Stab_G(v_0),\]
which shows that $\Stab_G(w) \subseteq A^{2k+1}$, and therefore $A' \subseteq A^{2k+1}$; this concludes the proof.
\end{proof}

Thus, for discrete characters, the correspondence between confining subsets and hyperbolic actions restricts to a correspondence between confining subgroups and actions on trees.

\begin{remark}
\label{rem:simplicial}
    In \cref{prop:trees_to_subgroups}, we assumed that the action on the tree $X$ is simplicial, and not merely by isometries. This is necessary to ensure that the character $\rho$ is discrete: a counterexample is the action defined via a dense embedding of $\Z^2$ into $\R$, with its standard simplicial structure.

    However, this assumption is void in the case of quasi-parabolic actions, which is what we will be most interested in. Indeed, suppose that $G$ admits a quasi-parabolic action by isometries on a simplicial tree $X$ and let $g \in G$. Choose a vertex $v$ of valency at least $3$, so $gv$ must have the same valency and therefore be a vertex. Let $w$ be any other vertex, so $d(gv, gw) = d(v, w)$ is an integer, and hence $gw$ is also a vertex. This shows that every quasi-parabolic action by isometries on a simplicial tree is automatically simplicial.

    In particular, the second conclusion of \cref{prop:trees_to_subgroups} holds for all actions by isometries on simplicial trees. If the action is oriented lineal, then the confining subset arising from \cref{cit:action_to_confining} is the whole of $H$, and if the action is quasi-parabolic, then it is automatically simplicial by the above, so we can apply \cref{prop:trees_to_subgroups} as stated.
\end{remark}

\subsection{Oriented lineal and quasi-parabolic structures for agreeable, $\Z^n$-split groups}\label{ssec:lin_qp}

The above results allow us to give a more explicit description of lineal and quasi-parabolic structures in the following cases:

\begin{definition}[$\Z^n$-split]
A group $G$ is \emph{$\Z^n$-split} if it admits a surjective map $G\to\Z^n$ that splits, so that $G=H\rtimes \Z^n$ for $H$ the kernel of this map. Call an expression $G=H\rtimes \Z^n$ a \emph{$\Z^n$-splitting} of a $\Z^n$-split group.
\end{definition}

Different choices of sections $\Z^n\to G$ will yield different $\Z^n$-splittings, which is why we are being careful to say ``a'' $\Z^n$-splitting. A special type of $\Z^n$-splitting will be most relevant for us:

\begin{definition}[Sharply $\Z^n$-split]
    A $\Z^n$-split $G = H \rtimes \Z^n$ is called \emph{sharp} if $H$ equals the kernel of the torsion-free abelianization of $G$. A group is \emph{sharply $\Z^n$-split} if it admits a sharp $\Z^n$-splitting.
\end{definition}

The relevance of sharp $\Z^n$-splittings lies in the following, which is just a consequence of the definitions.

\begin{lemma}\label{lem:sharp}
Let $G = H \rtimes \Z^n$ be a sharp $\Z^n$-splitting. Then every character of $G$ vanishes on $H$. If moreover $G$ is agreeable, then every pseudocharacter of $G$ vanishes on $H$. \qed
\end{lemma}

In particular, the hypothesis of \cref{cit:action_to_confining} is always satisfied for sharp $\Z^n$-splittings.

\medskip

Our next goal is to characterize lineal and quasi-parabolic actions of agreeable, sharply $\Z^n$-split groups. First let us discuss the (much easier) lineal case.

\begin{proposition}\label{prop:lineal}
Let $G$ be an agreeable, sharply $\Z^n$-split group. Let $G=H\rtimes\Z^n$ be a sharp $\Z^n$-splitting. Then the poset $\Hyp_\ell^+(G)$ is an antichain, and the function $\varphi \colon \Sigma(\Z^n)\to \Hyp_\ell^+(G)$ sending $[\rho]$ to $[H\cup Z_\rho]$ is a well defined 2-to-1 surjection with each fiber of the form $\{[\rho],[-\rho]\}$.
\end{proposition}

\begin{proof}
The proof is inspired by the work in Subsection~4.2 of \cite{confining}, which covers a special case. First note that $\Hyp_\ell^+(G)$ is an antichain simply because this is always the case \cite[Theorem~4.22(a)]{ABO}. Now we claim $\varphi$ is well defined. This is because if $\rho$ and $\rho'$ are equivalent, then $H\cup Z_\rho$ and $H\cup Z_{\rho'}$ dominate each other, and the output $[H\cup Z_\rho]$ really lies in $\Hyp_\ell^+(G)$ by \cref{cit:confining_to_action}. Since every pseudocharacter of $G$ vanishes on $H$ by \cref{lem:sharp}, \cref{cit:action_to_confining} applies, and gives us surjectivity. Finally, we need to show that $\varphi$ is 2-to-1 with the claimed fibers. Recall that we always assume $Z_{\rho} = Z_{-\rho}$, so $\varphi([\rho])=\varphi([-\rho])$. Now suppose $[\rho]\ne[\pm\rho']$, and we want to prove that $[H\cup Z_\rho]\ne[H\cup Z_{\rho'}]$. Since $[\rho]\ne[\pm\rho']$ we have $\ker(\rho)\ne \ker(\rho')$. Without loss of generality let us say that $\ker(\rho')\not\le \ker(\rho)$, and choose some $z\in \ker(\rho')\setminus \ker(\rho)$. Then $\langle z\rangle$ is a subgroup of $\ker(\rho')$, hence a subset of $Z_{\rho'}$, that is unbounded with respect to the word length coming from $H\cup Z_\rho$. We conclude that $H\cup Z_\rho\not\preceq H\cup Z_{\rho'}$, so $[H\cup Z_\rho]\ne[H\cup Z_{\rho'}]$.
\end{proof}

Now we focus on the quasi-parabolic case.

If $G = H \rtimes \Z^n$, and $\rho\colon\Z^n\to\R$ is a (non-trivial) character of $\Z^n$, write $\Hyp_{qp}^\rho(G)$ for the subposet of equivalence classes of those quasi-parabolic actions whose Busemann pseudocharacter induces a character $\Z^n\to\R$ equivalent to $\rho$.

\begin{proposition}
\label{prop:qp}

Let $G$ be an agreeable, sharply $\Z^n$-split group. Let $\rho \colon \Z^n \to \R$ be a non-trivial character. Then the posets $\Hyp_{qp}^\rho(G)$ and $\Conf^\ast_\rho(G)$ are isomorphic.
\end{proposition}

\begin{proof}
We claim that $[Q]\mapsto [Q\cup Z_\rho]$ is a well defined poset isomorphism from $\Conf^\ast_\rho(G)$ to $\Hyp_{qp}^\rho(G)$. Since $Q_1\preceq_\rho Q_2$ implies $Q_1\cup Z_\rho\preceq Q_2\cup Z_\rho$, it is a well defined poset map. Since $Q_1\cup Z_\rho\preceq Q_2\cup Z_\rho$ implies $Q_1\preceq_\rho Q_2$, it is injective. It just remains to prove that this map is surjective. Since $G$ is agreeable, \cref{cit:action_to_confining} says that every element of $\Hyp_{qp}^\rho(G)$ is of the form $[Q\cup Z_\rho]$ for some strictly confining $Q$, so we are done.
\end{proof}

Now we can tie together $\Hyp_\ell^+(G)$ and $\Hyp_{qp}^\rho(G)$.

\begin{proposition}\label{prop:linvsqp}
Let $G$ be an agreeable, sharply $\Z^n$-split group. Let $\rho \colon \Z^n \to \R$ be a non-trivial character. The element $[H\cup Z_\rho]$ of $\Hyp_\ell^+(G)$ is a lower bound of every element of $\Hyp_{qp}^\rho(G)$, and is not related to any elements of any $\Hyp_{qp}^{\rho'}(G)$ for $[\rho']\ne[\pm\rho]$.
\end{proposition}

\begin{proof}
The first claim is immediate since $[H\cup Z_\rho]\preccurlyeq [Q\cup Z_\rho]$ for all $\rho$-confining $Q$. We need to show the second claim. Suppose $[\rho']\ne [\pm\rho]$, so $\ker(\rho)\ne\ker(\rho')$. Extending $\rho$ and $\rho'$ from non-zero homomorphisms $\Z^n\to\R$ to non-zero linear transformations $\R^n\to\R$, the kernels $\ker(\rho)$ and $\ker(\rho')$ become distinct codimension-$1$ subspaces of $\R^n$. We can therefore find a sequence of points in $\ker(\rho')$ that gets arbitrarily far away from $\ker(\rho)$. Taking the $C_{\rho'}$-neighborhood and intersecting with $\Z^n$, we can choose a sequence of points in $Z_{\rho'}$ that gets arbitrarily far away from $Z_\rho$. This shows that for any $\rho'$-confining subset $Q$, the set $Q\cup Z_{\rho'}$ is unbounded in the alphabet $H\cup Z_{\rho}$, and so $[H\cup Z_\rho]\not\preccurlyeq[Q\cup Z_{\rho'}]$. Finally, it is obvious that $[Q\cup Z_{\rho'}]\not\preccurlyeq[H\cup Z_\rho]$ since \cref{prop:lineal} says that $[H\cup Z_{\rho'}]\not\preccurlyeq [H\cup Z_\rho]$. We conclude that $[H\cup Z_\rho]$ is not related to any elements of any $\Hyp_{qp}^{\rho'}(G)$.
\end{proof}

Finally, let us prove one useful criterion that can be used to show that $\Conf^\ast_\rho(G)$ is empty, for certain $\Z^n$-splittings and certain $\rho$. Here a group is \emph{boundedly generated} by a collection of subgroups if it is generated by their union $U$, and is bounded with respect to the word length coming from the generating set $U$. For an element $g \in G$ we denote by $C_H(g) \coloneqq \{ h \in H \mid [g, h] = 1 \}$ its centralizer in $H$.

\begin{proposition}
\label{prop:empty_conf}
Let $G = H \rtimes \Z^n$ be a $\Z^n$-splitting. Let $\{a_1,\dots,a_n\}$ be a basis for $\Z^n$ and $\rho\colon \Z^n \to \R$ a character. If $\rho(a_i) \neq 0$, then every $\rho$-confining subset contains $C_H(a_i)$. In particular, if $H$ is boundedly generated by those $C_H(a_i)$ with $\rho(a_i)\ne 0$, then $\Conf^\ast_\rho(G)= \emptyset$.
\end{proposition}

\begin{proof}
Suppose that $\rho(a_i) \neq 0$, and let $c\in C_H(a_i)$. By \cref{lem:conf:cyclic} there exists $k \in \Z$ such that $c^{a_i^k} \in Q$. But $c^{a_i^k} = c$ and so $C_H(a_i) \subseteq Q$.

If $H$ is boundedly generated by such $C_H(a_i)$, then $H=Q^\ell$ for some $\ell\in\N$. Now \cref{lem:products} tells us that $Q$ is equivalent to $H$.
\end{proof}

\section{Brown--Higman--Thompson groups}\label{sec:thompson}

In this section we recall some background on the Brown--Higman--Thompson groups $F_n$. We recall the following standard dynamical definitions, which we will repeatedly use in all that follows without further mention. For simplicity, we state them in our setting of one-dimensional actions:

\begin{definition}
\label{def:dynamical}

Let $G \leq \Homeo([0, 1])$. We will always view homeomorphisms as acting on $[0,1]$ from the right, so $x.g$ denotes the image of $x\in [0,1]$ under $g\in\Homeo([0,1])$.
\begin{enumerate}
\item An element $g \in G$ is \emph{orientation preserving} if $x < y$ implies $x.g < y.g$. If every element of $G$ is orientation preserving, we say that $G$ is \emph{orientation preserving}, and write $G \leq \Homeo^+([0, 1])$.
\item For an element $g \in G$, its \emph{support} is the set $\Supp(g) = \{ x \in [0, 1] \mid x.g \neq x \}$. For a set $S \subseteq G$, we denote $\Supp(S) = \cup_{g \in S} \Supp(g)$. We say that $S$ is \emph{supported on $I \subseteq [0, 1]$} if $\Supp(S) \subseteq I$.
\item The group $G$ acts \emph{proximally} if for every $x, y, z \in (0, 1)$ there exists a sequence $g_i : i \in \N$ such that $x.g_i \to z$ and $y.g_i \to z$.
\end{enumerate}
\end{definition}

A crucial fact, that will be used throughout the paper, is that elements with disjoint support commute.

\begin{remark}
\label{rem:proximal}

Suppose that the diagonal action of $G$ on $(0, 1) \times (0, 1)$ has a dense orbit $O$. Then $G$ acts proximally on $(0, 1)$. Indeed, let $x, y, z \in (0, 1)$, without loss of generality assume $x < y$. Because $O$ is dense, there exists $(x', y') \in O$ such that $x' < x < y < y'$. Using density again, there exists a sequence $g_i$ ($i \in \N$) such that $(x', y').g_i \to (z, z)$. Then $x.g_i \to z$ and $y.g_i \to z$.
\end{remark}

Recall that, for $n\ge 2$, $F_n$ is the group of all piecewise linear orientation preserving homeomorphisms of $[0,1]$ with all slopes powers of $n$ and breakpoints in $\Z[1/n]$. It turns out that $F_n$ admits a standard infinite presentation
\[
F_n=\langle x_0,x_1,\dots\mid x_j^{x_i} = x_{j+(n-1)} \text{ for all } i<j\rangle \text{.}
\]
See for example Proposition~4.8 of \cite{brown87} (following Propositions~4.1 and 4.4 of \cite{brown87}, which confirm that we are talking about the same group). Although this presentation (and a related finite presentation) is useful in many contexts, in our case the dynamical definition will be more important.

We introduce the following notation for piecewise linear homeomorphisms:

\begin{definition}
For a piecewise linear element $g \in \Homeo([0, 1])$ and an element $x \in [0, 1]$, we denote by $g'_+(x)$ the right slope at $x$, and by $g'_-(x)$ the left slope at $x$. If the left and right slopes coincide, or if $x \in \{0, 1\}$ and so only one is defined, then we just write $g'(x)$.
\end{definition}

\subsection{Transitivity properties}

The action of $F_n$ on $(0,1)$ stabilizes $(0,1)\cap \Z[1/n]$. However, the action of $F_n$ on $(0,1)\cap \Z[1/n]$ is not transitive when $n > 2$. We refer the reader to \cite[Section 3]{hydelodha} for a characterization of the orbits; all that we will need is the following:

\begin{cit}[{\cite[Section 3]{hydelodha}}]
\label{cit:proximal}

There exists a dense subset $O \subseteq (0, 1) \cap \Z[1/n]$ that is $F_n$-invariant and on which the action of the commutator subgroup $F_n'$ is highly transitive. That is to say, for any $k\in\N$ and any two tuples $0 < x_1 < \cdots < x_k < 1$ and $0 < y_1 < \cdots < y_k < 1$ with $x_i, y_i \in O$, there exists $g \in F_n'$ such that $x_i.g = y_i$ for all $i$.

In particular, every subgroup $F_n' \leq H \leq F_n$ is proximal on $(0, 1)$.
\end{cit}

The ``in particular'' follows from \cref{rem:proximal}.

\subsection{Abelianization}

It is clear from the presentation that the abelianization $\Ab(F_n)$ is isomorphic to $\Z^n$. In this section, we will describe this abelianization dynamically. Two characters play a special role:

\begin{definition}
\label{def:chi01}
    Let $\chi_0 \colon F_n \to \Z$ be 	defined as $\chi_0(g) \coloneqq \log_n g'(0)$. Similarly, $\chi_1(g) \coloneqq \log_n g'(1)$. We denote $F_n^c \coloneqq \ker(\chi_0) \cap \ker(\chi_1)$.
\end{definition}

We call $F_n^c$ the \emph{compactly supported subgroup} of $F_n$, since it consists of those homeomorphisms whose support is contained in a compact subset of $(0, 1)$. The quotient $F_n / F_n^c$ is isomorphic to $\Z^2$. This is the end of the story in the important case of Thompson's group $F$, when $n  = 2$. (For $n > 2$, since the abelianization is $\Z^n$, one can find an additional $n-2$ linearly independent characters.)

We will see later that $F_n$ is sharply $\Z^n$-split, i.e., the abelianization map $F_n\to\Z^n$ splits. The following is the key result in this direction. Here an element of $F_n$ is called a \emph{one-bump function} if its support equals a single open interval.

\begin{proposition}\label{prop:modelsplitting}
There exist elements $a_0, \ldots, a_{n-1} \in F_n$ such that:
\begin{enumerate}
\item Each $a_i$ is a one-bump function with support $I_i = (\ell_i, r_i)$ for some $\ell_i,r_i$. Moreover, $\ell_0 = 0$ and $\chi_0(a_0)=1$, and $r_{n-1} = 1$ and $\chi_1(a_{n-1})=1$.
\item The $I_i$ are pairwise disjoint, in fact $r_{i-1} < \ell_i$ for $i = 1, \ldots, (n-1)$.
\item The $a_i$ generate the abelianization of $F_n$.
\end{enumerate}
\end{proposition}

\begin{proof}
In \cite{burillo01}, the generators $x_i$ in the standard presentation for $F_n$ are described in terms of the so called tree pair model for elements of $F_n$, and from this it is a straightforward but tedious exercise to view the $x_i$ as piecewise linear homeomorphisms of $[0,1]$. Setting $y_i=x_ix_{i+1}^{-1}$ for each $0\le i\le n-2$, and $y_{n-1}=x_{n-1}^{-1}$, it is easy to see that $y_0,\dots,y_{n-1}$ are one-bump functions, $\chi_0(y_0)=\chi_1(y_{n-1})=1$, and $\{y_0,\dots,y_{n-1}\}$ generates $F_n$. Up to conjugating each $y_i$, and recalling that $F_n$ acts proximally on $(0,1)$ (\cref{cit:proximal}), we may replace each $y_i$ with an element $a_i$ with support $I_i = (\ell_i, r_i)$ such that $r_0 < \ell_1 < r_1 < \ell_2 < r_2 < \cdots < \ell_{n-2} < r_{n-2} < \ell_{n-1}$. Now the elements $a_0,\dots,a_{n-1}$ satisfy all the desired conditions.
\end{proof}

See \cref{fig:bumps} for a visualization of the elements $a_i$ constructed in \cref{prop:modelsplitting}.

\begin{figure}[htb]
\centering
\begin{tikzpicture}[line width=0.8pt,rotate=45]

\draw[dashed] (0,0) -- (4,0);
\draw (0,0) to[out=40,in=140] (1,0);
\draw (1.5,0) to[out=40,in=140] (2.5,0);
\draw (3,0) to[out=-40,in=-140] (4,0);

\node at (0.5,0.5) {$a_0$};
\node at (2,0.5) {$a_1$};
\node at (3.5,-0.5) {$a_2$};

\end{tikzpicture}
\caption{A graph describing $\{ a_0, a_1, a_2 \}$ in $F_3$, as in \cref{prop:modelsplitting}. Note that in order to have $\chi_0(a_0)=\chi_1(a_2)=1$, we need $a_0$ to lie above the $y=x$ line, and $a_2$ to lie below it.}
\label{fig:bumps}
\end{figure}
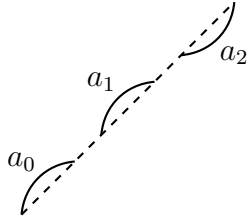

\subsection{Studying $\Hyp(F_n)$ via confining subsets}

In this section, we establish that $F_n$ meets the criteria for the results of \cref{sec:confining} to hold. We start with:

\begin{cit}[{\cite{calegari:PL, commutingconjugates}}]
\label{cit:F_agreeable}
Every subgroup of $F_n$ is agreeable.
\end{cit}

This allows us to exclude non-oriented lineal actions.

\begin{corollary}\label{cor:nonoriented:f}
Every lineal action of $F_n$ is oriented.
\end{corollary}

\begin{proof}
By \cref{cor:nonoriented} and \cref{cit:F_agreeable}, it suffices to show that $F_n$ does not surject onto $D_\infty$, and this is immediate since every proper quotient of $F_n$ is abelian \cite[Theorem~4.13]{brown87}.
\end{proof}

We can also exclude general type actions, as a consequence of the following and the well-known Ping-Pong Lemma:

\begin{cit}[\cite{brin85}]
\label{cit:nofree}
    $F_n$ does not contain non-abelian free subgroups. Therefore, $F_n$ does not admit general type actions.
\end{cit}

Finally, we note:

\begin{corollary}
\label{cor:Fsplit}
$F_n = F_n' \rtimes \Z^n$ is sharply $\Z^n$-split.
\end{corollary}

\begin{proof}
The elements $a_0, \ldots, a_{n-1}$ from \cref{prop:modelsplitting} are disjointly supported, and so they commute. This defines a section $\Z^n \to F_n$ of the abelianization map.
\end{proof}

At this point we can write $F_n = F_n' \rtimes \Z^n$, where $\Z^n = \langle a_0, \ldots, a_{n-1} \rangle$ as in \cref{prop:modelsplitting}. We fix this $\Z^n$-splitting for the rest of the paper. In particular, note that the restriction of $\chi_0$, respectively $\chi_1$, to this copy of $\Z^n$ maps the product $a_0^{i_0}\cdots a_{n-1}^{i_{n-1}}$ to $i_0$, respectively $i_{n-1}$.

\subsection{Stabilizers in $F_n$ and its subgroups}

Later on, we will need some properties of point and interval stabilizers in $F_n$. We first fix the following notation:

\begin{definition}[Fixators, rigid stabilizers]
For an interval $I \subseteq [0, 1]$ (open, closed, or half-open on either side) and a group $G \leq \Homeo([0, 1])$, the \emph{fixator} of $I$ in $G$ is the subgroup
\[
\Fix_G(I) \coloneqq \{ g \in G \mid x.g = x \text{ for all } x \in I \}.
\]
When $I = \{ x \}$ is a degenerate interval, we simply write $\Fix_G(x)$ for $\Fix_G(\{x\})$. The \emph{rigid stabilizer} of $I$ in $G$ is the subgroup $G(I) \coloneqq \left\{ g \in G \mid \overline{\Supp(g)} \subseteq I \right\}$.
\end{definition}

When the interval is written explicitly, we avoid redundancy of parentheses, so for example we write $G[x, y]$ instead of $G([x, y])$. Note that with this notation, we have for example $F_n^c = F_n(0, 1)$.

\begin{definition}[$F_n$-like homeomorphisms]
We say that a homeomorphism between open intervals $I \to J$ is \emph{$F_n$-like} if it is orientation preserving, piecewise linear, with a set of breakpoints that is countable discrete and contained in $\Z[1/n]$, and all slopes powers of $n$.
\end{definition}

Note that when $I = J = (0, 1)$, there are more $F_n$-like homeomorphisms than elements of $F_n$, since in the first case we allow for breakpoints to accumulate at either $0$ or $1$.

\begin{cit}[{\cite[Theorem E16.7]{bieristrebel}}]
\label{cit:bieristrebel}

Let $0 < x < y < 1$ be such that $x, y \in \Z[1/n]$. Then there exists an $F_n$-like homeomorphism $\Phi_{x, y} \colon (x, y) \to (0, 1)$ that conjugates $F_n[x, y]$ to $F_n$.
\end{cit}

Note that the conjugating element really needs to have (countably) infinitely many breakpoints, not merely finitely many; otherwise, it could only intertwine $F_n[x, y]$ and $F_n[0, 1]$ if the orbits were compatible \cite{hydelodha}.

In the case of rigid stabilizers of open intervals, the isomorphism works even when $x, y$ are not $n$-ary (and moreover, it does not require the previous theorem), as we now show.

\begin{proposition}[Rigid stabilizers of open intervals]
\label{prop:rstab:open}

Let $0 < x < y < 1$. Then there exists an $F_n$-like homeomorphism $\Phi_{x, y} \colon (x, y) \to (0, 1)$ that conjugates $F_n(x, y)$ to $F_n^c$.
\end{proposition}

\begin{proof}
Recall from \cref{cit:proximal} that $F_n$ acts transitively on ordered tuples of elements in some dense orbit $O\subseteq (0,1)\cap \Z[1/n]$. Choose a sequence of elements $x < x_i < y_i < y$ such that $x_i, y_i \in O$, the sequence $(x_i)_i$ is strictly decreasing and converges to $x$, and the sequence $(y_i)_i$ is strictly increasing and converges to $y$. Moreover, choose a sequence of elements $0 < p_i < p_{i-1} < q_{i-1} < q_i < 1$ such that $p_i, q_i \in O$ and analogously $p_i \to 0, q_i \to 1$.

Since $F_n$ is transitive on ordered pairs in $O$, $F_n$ can take each $x_i<y_i$ to $p_i<q_i$. In this way, we get $F_n$-like homeomorphisms $\Phi_i \colon (x_i, y_i) \to (p_i, q_i)$ that conjugate $F_n(x_i, y_i)$ to $F_n(p_i, q_i)$ simply by conjugating by an appropriate element of $F_n$. Moreover, by defining the elements of $F_n$ inductively and by pieces, we may ensure that each $\Phi_i$ extends $\Phi_{i-1}$.

This defines an orientation preserving homeomorphism $\Phi \colon (x, y) \to (0, 1)$ that conjugates the directed union of the groups $F_n(x_i, y_i)$, namely $F_n(x, y)$, to the directed union of the $F_n(p_i, q_i)$, namely $F_n(0, 1) = F_n^c$. In fact, each of the $\Phi_i$ is constructed via conjugation in $F_n$, and so has finitely many breakpoints. It follows that the set of breakpoints of $\Phi$ is countable and discrete in $(x, y)$, and so $\Phi$ is $F_n$-like.
\end{proof}

\cref{prop:rstab:open} shows that the behavior of rigid stabilizers of open intervals is independent of the nature of the endpoints. However, the behavior of stabilizers of points does depend on the nature of the points, with different behavior when a fixed point is $n$-ary, rational but not $n$-ary, or irrational. We will now describe the behavior in each of the three cases; when $n = 2$ this was already worked out in \cite{golansapir}.

The $n$-ary case is immediate, because of the piecewise nature of $F_n$:

\begin{lemma}[$n$-ary fixed points]
\label{lem:fix:nary}

Let $t \in \Z[1/n] \cap (0, 1)$, Then $\Fix_{F_n}(t)$ decomposes as $F_n[0, t] \times F_n[t, 1]$.\qed
\end{lemma}

Note however that this product decomposition is not always stable under replacing $F_n$ in the notation with a subgroup. For example using $F_n'$, under this decomposition $\Fix_{F_n'}(t)$ corresponds to those elements $(g, h)$ of $F_n[0, t] \times F_n[t, 1]$ such that $\Ab(g) = - \Ab(h)$, and it is easy to see that $\Fix_{F_n'}(t)$ is identified with a subgroup of $F_n(0, t] \times F_n[t, 1)$ whose projection onto each factor is surjective. In particular, $F_n'(t)$ does not decompose as $F_n'[0,t]\times F_n'[t,1]$.

\medskip

The irrational case is also relatively straightforward:

\begin{lemma}[Irrational fixed points]
\label{lem:fix:irrational}

Let $t \in (0, 1) \setminus \Q$. Then every element $g \in F_n$ that fixes $t$ also fixes a neighborhood of $t$. Thus $\Fix_{F_n}(t)$ decomposes as $F_n[0, t] \times F_n[t, 1] = F_n[0, t) \times F_n(t, 1]$.
\end{lemma}

\begin{proof}
Suppose that, at $t$, the element $g$ takes the form $x \mapsto n^k x + y$, for $k \in \Z$ and $y \in \Z[1/n]$. If $k \neq 0$, then $t = t.g$ gives $t = y/(1-n^k) \in \Q$. Since $t$ is irrational, we conclude that $k=0$. Moreover, $t$ cannot be a break point, so the slope of $g$ has to be identically $1$ in a neighborhood of $t$, i.e., $g$ must fix a neighborhood of $t$. The decomposition into a direct product is now immediate.
\end{proof}

Finally there is the case when the fixed point is rational but not $n$-ary, which is more delicate:

\begin{lemma}[Non $n$-ary rational fixed points]
\label{lem:fix:rational}

Let $t \in (0, 1) \cap \Q \setminus \Z[1/n]$. Then the map $\Fix_{F_n}(t) \to \Z$ given by $g \mapsto \log_n g'(t)$ is a well defined non-trivial homomorphism. The image only depends on the $F_n$-orbit of $t$. Thus $\Fix_{F_n}(t)$ fits into a short exact sequence
\[1 \to F_n[0, t) \times F_n(t, 1] \to \Fix_{F_n}(t) \to \Z \to 1;\]
where the element of $\Fix_{F_n}(t)$ mapping to the generator of the image in $\Z$ may be chosen to belong to $F_n'$ and to be supported on an arbitrarily small open neighborhood of $t$.
\end{lemma}

\begin{proof}
Since $t \notin \Z[1/n]$, it cannot be a breakpoint, so $g'(t)$ is well defined for all $g \in F_n$, and it is easily seen that we get the homomorphism in the statement, by applications of the chain rule and log rules.
To see that this homomorphism is non-trivial, write $t = p/q$, where $p$ and $q$ are coprime, and so by hypothesis $q$ is not a power of $n$. We multiply the numerator and denominator suitably to obtain an expression $t = p'/q' n^\ell$, where $q'$ is coprime to $n$, and greater than $1$. Let $o\ge 1$ be the order of $n$ in $(\Z/q'\Z)^\times$; so $q'$ divides $(n^o - 1)$. By multiplying the numerator and denominator again, we obtain an expression
\[
t = \frac{m}{(n^o - 1) n^\ell} = \frac{1}{n^o-1} \frac{m}{n^\ell}\text{.}
\]
Setting $y = m/n^\ell \in \Z[1/n]$, we obtain the expression $t = n^o t - y$.

Let $I$ be an open interval containing $t$ such that the map $x \mapsto n^o x - y$, which fixes $t$, maps $I$ inside a compact subset of $(0, 1)$. This can then be completed to an element of $F_n$, and even to an element of $F_n'$, up to multiplying by suitable conjugates of the elements $a_i$ from \cref{prop:modelsplitting} supported away from $t$. Since $o \ge 1$, this shows that the homomorphism in the statement is non-trivial, with image generated by an element of $F_n'$, and the kernel is easily identified.

Finally, let $g \in \Fix_{F_n}(t)$ and let $c \in F_n$. Then $g^c \in \Fix_{F_n}(t.c)$, and moreover $t.c \in \Q \setminus \Z[1/n]$ and the slope of $g$ at $t$ coincides with the slope of $g^c$ at $t.c$. This shows that the image of $\Fix_{F_n}(t) \to \Z$ given by $g \mapsto \log_n g'(t)$ does not change if we replace $t$ by a rational in the same orbit, and concludes the proof.
\end{proof}

\begin{remark}
The proof actually shows that the generator of the image of $\log_n g'(t)$ coincides with the order of $n$ modulo $q'$, where $q'$ is the largest factor of the denominator of $t$ that is coprime to $n$. The invariance over $F_n$-orbits is also reflected in this property \cite{hydelodha}.
\end{remark}

\section{Global structure of $\Hyp(F_n)$}\label{sec:global}

In this section we analyze confining subsets in $F_n$, and find that for ``almost all'' $\rho$ there are no quasi-parabolic actions of $F_n$ with Busemann pseudocharacter equivalent to $\rho$. More precisely, this happens for any $\rho\ne \chi_0,\chi_1$. Since $F_n$ is agreeable (\cref{cit:F_agreeable}) and sharply $\Z^n$-split (\cref{cor:Fsplit}), we can use the results from \cref{sec:confining}, and so really what we are proving is that for any $\rho\ne \chi_0,\chi_1$, the set $\Conf^\ast_\rho(F_n)$ is empty. Here $\Conf^\ast_\rho(F_n)$ is as in \cref{def:conf_poset}. Throughout this, we fix the sharp $\Z^n$-splitting $F_n = F_n' \rtimes \mathbb{Z}^n$ as in \cref{prop:modelsplitting}, where $\Z^n$ is the subgroup generated by $\{a_0,\dots,a_{n-1}\}$. In particular, each $a_i$ is a one-bump function with support $I_i = (\ell_i, r_i)$, these intervals are pairwise disjoint, and we have $\ell_0 = 0, r_{n-1} = 1$.

We should briefly remark that there is a connection between hyperbolic actions and Bieri--Neumann--Strebel (BNS) invariants of groups, as in \cite{bieri87}; see for example \cite[Subsection~4.3]{confining}. The BNS-invariants of the $F_n$ have been completely determined \cite[Theorem~8.1]{bieri87}, and thanks to an equivalent characterization of the BNS-invariants due to Brown \cite[Corollary~7.4]{brown87trees}, it is clear that there are no $\rho$-confining subsets for $\rho \neq \chi_0, \chi_1$ whose corresponding hyperbolic actions are on $\R$-trees. This does not immediately help us here though, because $\R$-trees are just a special case of hyperbolic spaces, so the connection to BNS-invariants is more just ``in spirit''.

\begin{proposition}\label{prop:usually_empty_conf}
Let $\rho$ be a character of $\Z^n$ such that there exist $i\ne j$ with $\rho(a_i)\ne 0$ and $\rho(a_j)\ne 0$. Then $\Conf^\ast_\rho(F_n)=\emptyset$.
\end{proposition}

\begin{proof}
For each $k$, note that the centralizer $C_{F_n'}(a_k)$ contains the subgroup $\Fix_{F_n'}(I_k)$ of all elements of $F_n'$ fixing $I_k$. By \cref{prop:empty_conf}, it therefore suffices to prove that $F_n'$ is boundedly generated by $\Fix_{F_n'}(I_i)\cup \Fix_{F_n'}(I_j)$. We claim that every element of $F_n'$ is a product of at most three elements of this union.
Let $f\in F_n'$, so in particular $f \in F_n^c$ and so the support of $f$ is contained in $(\varepsilon,\delta)$ for some $0<\varepsilon<\delta<1$. Since $i\ne j$ we know that $I_i$ and $I_j$ are disjoint. Say without loss of generality that $i<j$. If $\varepsilon>r_i$ then $f\in \Fix_{F_n'}(I_i)$ and we are done, so suppose $\varepsilon\le r_i$. Since $I_i$ and $I_j$ are disjoint, we can choose some element $f'\in \Fix_{F_n'}(I_j)$ such that $\varepsilon.f'>r_i$. Now the support of $f^{f'}$ is contained in $(r_i,1)$ and hence $f^{f'}\in\Fix_{F_n'}(I_i)$. We conclude that
\[
f\in f'\Fix_{F_n'}(I_i)(f')^{-1} \subseteq \Fix_{F_n'}(I_j)\Fix_{F_n'}(I_i)\Fix_{F_n'}(I_j)\text{,}
\] and we are done.
\end{proof}

At this point we are free to assume that $\rho(a_i)\ne 0$ for just a single $i$. Our next goal is to prove that in fact $\Conf^\ast_\rho(F_n)=\emptyset$ also holds if this $i$ is anything other than $0$ or $n-1$.

\begin{proposition}\label{prop:middle_empty_conf}
Let $\rho$ be a character of $\Z^n$ such that $\rho(a_i)\ne 0$ for some $0<i<n-1$, and $\rho(a_j)=0$ for all $j\ne i$. Then $\Conf^\ast_\rho(F_n)=\emptyset$.
\end{proposition}

\begin{proof}
For notational ease, let us write $a$ for $a_i$, $I$ for the interval $I_i$, and $\ell$ and $r$ for the endpoints $\ell_i$ and $r_i$ of $I$. Up to replacing $a$ by its inverse, let us assume that $\rho(a) > 0$. Let $Q\subseteq F_n'$ be a $\rho$-confining subset. We need to prove that $Q=F_n'$. Since $0<i<n-1$, we know that $(r_{i-1},\ell)\ne\emptyset$ and $(r,\ell_{i+1})\ne\emptyset$. By \cref{prop:empty_conf}, we know that $C_{F_n'}(a) \subseteq Q$, and so $\Fix_{F_n'}(I)\subseteq Q$. 

We claim that there exists $q_\ell\in Q$ such that $x_\ell.q_\ell<\ell$ for some $x_\ell\in (\ell,r)$. First let $q_\ell'\in F_n'$ be any element supported on $(r_{i-1},r)$, such that $\ell.q_\ell'<\ell$, so then we also get $x_\ell'.q_\ell'<\ell$ for some $x_\ell'\in (\ell,r)$. By \cref{lem:conf:cyclic}, we can choose $k \geq 0$ such that $q_\ell\coloneqq (q_\ell')^{a^k}\in Q$. Let $x_\ell\coloneqq x_\ell'.a^k$. Then $x_\ell\in (\ell,r)$ and
\[
x_\ell.q_\ell = (x_\ell'.a^k).(q_\ell')^{a^k} = (x_\ell'.q_\ell').a^k = x_\ell'.q_\ell'<\ell
\]
since $a^k$ fixes $[0,\ell]$. A symmetric argument produces an element $q_r\in Q$ and $x_r\in(\ell,r)$ such that $x_r.q_r>r$.

By now we know that $Q$ contains $\Fix_{F_n'}(I)\subseteq Q$ and contains the elements $q_\ell$ and $q_r$ constructed above. Now we are ready to begin proving that $Q=F_n'$, for which it suffices to show that $F_n'$ is bounded in the generating set $Q$. Let us assume that $x.a>x$ for all $x$ in the support $I$ of $a$; since $a$ is a one-bump function, the alternative is that $x.a<x$ for all $x\in I$, and this case works analogously. Let $f\in F_n'$, so in particular $f \in F_n^c$ and so $f$ is supported on $(\varepsilon,\delta)$ for some $0<\varepsilon<\delta<1$. If $\delta\le \ell$ then $f\in \Fix_{F_n'}(I)$, so $f\in Q$ and we are done. Now suppose $\ell < \delta < r$. Since $a$ is a one-bump function satisfying $x.a>x$ for all $x$ in its support $(\ell,r)$, we can choose $k\ge 0$ such that $\delta.a^{-k}<x_\ell$, and so $f^{a^{-k}q_\ell}\in \Fix_{F_n'}(I) \subseteq Q$. This implies that $f \in Q^{q_\ell^{-1} a^k} \subseteq (Q^3)^{a^k} \subseteq Q^3$, where the last inclusion is given by \StayInConf. Finally, suppose $r \le \delta$. Choose $f'\in \Fix_{F_n'}(I)\subseteq Q$ such that $\delta.f'<x_r.q_r$, so $\delta.f'q_r^{-1}<r$. Like in the previous case, we can now choose $k\ge 0$ such that $\delta.f'q_r^{-1} a^{-k}<x_\ell$, so $f^{f'q_r^{-1} a^{-k}q_\ell}\in  \Fix_{F_n'}(I) \subseteq Q$, and so $f\in Q^{q_\ell^{-1} a^k q_r (f')^{-1}} \subseteq Q^7$. We conclude that in all cases $f\in Q^7$, which concludes the proof.
\end{proof}

This leaves us with only four characters of $\Z^n$ to consider up to equivalence, since precisely one of $a_0$ or $a_{n-1}$ can have a positive or negative character value, and the other must be sent to zero. In fact by \cref{prop:modelsplitting} these four characters of $\Z^n$ are equivalent to the four characters $\pm\chi_0$ and $\pm\chi_1$ of $F_n$, restricted to $\Z^n$. Our next result is that $\Conf^\ast_{-\chi_0}(F_n)$ and $\Conf^\ast_{-\chi_1}(F_n)$ are also empty, leaving $\chi_0$ and $\chi_1$ as the only candidates for yielding quasi-parabolic actions. (Then later we will show that these actually yield an incredible number of quasi-parabolic actions.)

\begin{proposition}\label{prop:last_empty_conf}
The sets $\Conf^\ast_{-\chi_0}(F_n)$ and $\Conf^\ast_{-\chi_1}(F_n)$ are empty.
\end{proposition}

\begin{proof}
We will prove $\Conf^\ast_{-\chi_0}(F_n)=\emptyset$, and the other statement is analogous. Let $Q\subseteq F_n'$ be $(-\chi_0)$-confining. By \cref{prop:empty_conf}, we know that $C_{F_n'}(a_0) \subseteq Q$, and so $\Fix_{F_n'}(I_0)=F_n'[r_0,1)\subseteq Q$. By the same argument as in the proof of \cref{prop:middle_empty_conf}, $Q$ also contains an element $q$ such that $x.q>r_0$ for some $x\in (0,r_0)$. Now we are ready to show that $Q=F_n'$, for which it suffices to show that $F_n' \subseteq Q^3$. Let $f\in F_n'$, say supported on $(\varepsilon,\delta)$ for some $0<\varepsilon<\delta<1$. Choose $k\ge 0$ such that $\varepsilon.a_0^k>x$, so $\varepsilon.a_0^k q > r_0$. Now $f^{a_0^k q}\in F_n'[r_0,1)\subseteq Q$, so $f\in (Q^3)^{a_0^{-k}}$. Since $(-\chi_0)(a_i^{-k})=k\ge 0$, we conclude that $f\in Q^3$ by \StayInConf, which concludes the proof.
\end{proof}

We now have everything in place to give a global description of $\Hyp(F_n)$. We refer the reader to \cref{intro:fig:global} in the introduction for a picture.

\begin{theorem}[\cref{intro:thm:global}]\label{thm:global}
The poset $\Hyp(F_n)$ has the following structure.
\begin{enumerate}
   \item $\Hyp_{gt}(F_n) = \emptyset$.
   \item $\Hyp_{qp}(F_n)$ consists of two incomparable isomorphic join semilattices $\X_0$ and $\X_1$.
   \item $\Hyp_{\ell}(F_n) = \Hyp_{\ell}^+(F_n)$ is an uncountable antichain, with two distinguished elements $[\pm\chi_0]$ and $[\pm\chi_1]$ sitting under $\X_0$ and $\X_1$, respectively.
   \item $\Hyp_e(F_n)$ consists of the unique smallest element.
\end{enumerate}
\end{theorem}

\begin{proof}
The last item is always true. The fact that $\Hyp_{gt}(F_n) = \emptyset$ is the content of \cref{cit:nofree}. The fact that $\Hyp_{\ell}(F_n) = \Hyp_{\ell}^+(F_n)$ is the content of \cref{cor:nonoriented:f}. This is an uncountable antichain by \cref{prop:lineal}, with two distinguished elements corresponding to $[\pm\chi_0]$ and $[\pm\chi_1]$.

We define $\X_0 \coloneqq \Hyp_{qp}^{\chi_0}(F_n)$, and similarly $\X_1 \coloneqq \Hyp_{qp}^{\chi_1}(F_n)$. They are join semilattices by \cref{prop:join}. They are isomorphic because the automorphism $\Homeo[0, 1] \to \Homeo[0, 1]$ given by conjugation by $x \mapsto (1 - x)$ preserves $F_n$ and conjugates $\chi_0$ to $\chi_1$. By construction, they sit above $[\pm\chi_0]$ and $[\pm\chi_1]$, respectively, and above no other lineal actions by \cref{prop:linvsqp}. Finally, that $\Hyp_{qp}(F_n)$ consists only of $\X_0$ and $\X_1$ follows from the propositions proved above in this section that show no other character admits strictly confining subsets.
\end{proof}

The results of this section will also be enough to give a complete description of the poset of hyperbolic structures on the group $F_n^\pm$, which will be the subject of the next section. After that, we will make a detour on lamplighters and their hyperbolic structures, and then the rest of the paper will be devoted to analyzing the rich structure of the poset $\X_0 \cong \X_1$, which up to isomorphism is $\Conf^\ast_{\chi_0}(F_n)$.

\section{Hyperbolic structures on $F_n^\pm$}\label{sec:reversing}

Recall that $F_n^\pm$ is the group of homeomorphisms of $[0, 1]$ defined the same way as $F_n$ but where the orientation is not necessarily preserved. Let $\alpha \in F_n^\pm$ be the homeomorphism $x \mapsto (1-x)$. Then $F_n^\pm \cong F_n \rtimes_{\alpha} \Z/2\Z$. From this we easily obtain:

\begin{lemma}
\label{lem:reversing:agreeable}

The group $F_n^\pm$ is agreeable.
\end{lemma}

\begin{proof}
Every pseudocharacter $\rho \colon F_n^{\pm} \to \R$ restricts to a character on $F_n$, because $F_n$ is agreeable. Moreover, by homogeneity we have $\rho(\alpha) = 0$, and $\rho(g^\alpha) = \rho(g)$ for all $g \in F_n$. Thus $\rho$ preserves all relations of $F_n^\pm$, so it is a homomorphism.
\end{proof}

We will need the following more algebraic description of the abelianization of $F_n$, see e.g., \cite{zaremsky17F_n}. There exist characters $\psi_0, \ldots, \psi_{n-2} \colon F_n \to \Z$ that satisfy the relation
\begin{equation}
\label{eq:relation}
\psi_0 + \cdots + \psi_{n-2} = \chi_0 - \chi_1.
\end{equation}
We have $(n+1)$ characters and one relation, so this encodes the whole rank-$n$ abelianization of $F_n$. (We will not need to know any more details about the $\psi_i$ for our purposes, and so just refer the reader to \cite{zaremsky17F_n} for more precise information.) Conjugation by $\alpha$ induces an order-$2$ automorphism on $F_n$ that permutes these characters as follows (this is easy to see from the description of these characters in \cite{zaremsky17F_n}):
\[
\chi_0 \longleftrightarrow \chi_1; \qquad \psi_i \longleftrightarrow -\psi_{(n-2)-i} \text{ for } i = 0, \ldots, (n-2).
\]
The following is immediate from this description:

\begin{lemma}
\label{lem:eigenspaces}
Consider $H^1(F_n; \R)$ as the real vector space with $\{ \chi_0, \chi_1, \psi_0, \ldots, \psi_{n-2} \}$ as the generating set and one relation given by \cref{eq:relation}. Consider the action of $\alpha$ on $H^1(F_n; \R)$, and let $H^1(F_n; \R)^{\alpha}_+$ denote the $+1$ eigenspace and $H^1(F_n; \R)^{\alpha}_-$ the $-1$ eigenspace. Then
\begin{enumerate}
\item $H^1(F_n; \R)^{\alpha}_+ = \langle \chi_0 + \chi_1, \psi_i - \psi_{(n-2)-i} \mid i = 0, \ldots, \lfloor \frac{n}{2} \rfloor-1 \rangle$. This has $\lfloor \frac{n}{2} \rfloor + 1$ generators, of which the last one is $0$ if $n$ is even, and no other relation, and thus has dimension $\lceil \frac{n}{2} \rceil$.
\item $H^1(F_n; \R)^{\alpha}_- = \langle \chi_0 - \chi_1, \psi_i + \psi_{n-i} \mid i = 0, \ldots, \lfloor \frac{n}{2} \rfloor-1 \rangle$. This has $\lfloor \frac{n}{2} \rfloor + 1$ generators, and one relation coming from \cref{eq:relation}, and thus has dimension $\lfloor \frac{n}{2} \rfloor$. \qed
\end{enumerate}
\end{lemma}

This, together with \cref{intro:thm:global}, classifies quasi-parabolic structures on $F_n^{\pm}$, namely, there are none:

\begin{proposition}
\label{prop:reversing:qp}

The group $F_n^{\pm}$ does not admit any cobounded quasi-parabolic action on a hyperbolic space.
\end{proposition}

\begin{proof}
Suppose by contradiction that $F_n^\pm \curvearrowright X$ is a cobounded quasi-parabolic action. Let $p \colon F_n^\pm \to \R$ be the Busemann pseudocharacter, which is a character by \cref{lem:reversing:agreeable}, and is non-zero by assumption. Since $F_n$ has index $2$ in $F_n^{\pm}$, the restricted action is also quasi-parabolic, which implies that $p|_{F_n}$ is a non-trivial character of $F_n$ that appears as the pseudocharacter of a quasi-parabolic action of $F_n$. \cref{intro:thm:global} implies that $p|_{F_n}$ is equivalent to either $\chi_0$ or to $\chi_1$, but \cref{lem:eigenspaces} implies that $p|_{F_n} \in \langle \chi_0 + \chi_1, \psi_i - \psi_{(n-2)-i} \mid i = 0, \ldots, \lfloor \frac{n}{2} \rfloor-1 \rangle$. These two properties are incompatible.
\end{proof}

On the other hand, we can describe the lineal actions thanks to \cref{lem:eigenspaces}:

\begin{proposition}
\label{prop:reversing:lineal}

The poset $\Hyp_\ell(F_n^\pm)$ is an antichain of the following form.
\begin{enumerate}
\item If $n = 2$, then $\Hyp_\ell(F_2^\pm)$ has size $2$, with one oriented and one non-oriented structure.
\item If $n = 3$, then $\Hyp_\ell(F_3^\pm)$ is uncountable, with uncountably many oriented structures and one non-oriented structure.
\item If $n > 3$, then $\Hyp_\ell(F_n^\pm)$ is uncountable, with uncountably many oriented and non-oriented structures.
\end{enumerate}
\end{proposition}

\begin{proof}
Since $F_n^{\pm}$ is agreeable by \cref{lem:reversing:agreeable}, it follows from \cite[Theorem 4.22]{ABO} that $\Hyp_\ell^+(F_n^\pm)$ has size $1$ if $H^1(F_n^\pm; \R)$ has dimension $1$, and is uncountable if $H^1(F_n^\pm; \R)$ has dimension greater than $1$. The restriction to $F_n$ induces an isomorphism $H^1(F_n^\pm; \R) \to H^1(F_n; \R)^{\alpha}_+$, so \cref{lem:eigenspaces} gives the size of $\Hyp_\ell^+(F_n^\pm)$ as in the statement.

Now consider a non-oriented lineal action of $F_n^{\pm}$ on a hyperbolic space $X$. Let $\varepsilon(g)$ be the sign of the permutation that $g$ induces on the limit points. Because $F_n$ has no non-oriented lineal action (\cref{cor:nonoriented:f}), it is contained in the kernel of $\varepsilon$. Since $\varepsilon$ is non-trivial by assumption, $\varepsilon(\alpha) = -1$ necessarily. Let $p \colon F_n \to \R$ be the Busemann pseudocharacter for the action of $F_n$ on $X$, which is a character by \cref{cit:F_agreeable}. Use \cref{cit:quasicocycle} to extend it to an $\varepsilon$-cocycle $\varphi \colon F_n^{\pm} \to \R$. Using the cocycle relation we see that for any $g \in F_n$
\begin{align*}
    p(g^\alpha) &= \varphi(\alpha g \alpha) = \varphi(\alpha) + \varepsilon(\alpha) (\varphi(g) + \varepsilon(g) \varphi(\alpha)) \\
    &= \varphi(\alpha) - (\varphi(g) + \varphi(\alpha)) = -\varphi(g) = -p(g).
\end{align*}
Therefore $p \in H^1(F_n; \R)^\alpha_-$, and this space is described in \cref{lem:eigenspaces}. Conversely, given $p \in H^1(F_n; \R)^\alpha_-$, we can construct a $\varepsilon$-cocycle $\varphi$ given by $\varphi(g \alpha^i) = (-1)^i p(g)$ for $g \in F_n$ and $i \in \{0, 1\}$, and this defines a non-oriented lineal action \cite[Lemma 2.14]{NL}. When $n > 3$, so the dimension of $H^1(F_n; \R)^\alpha_-$ is greater than $1$, this shows that there exist uncountably many non-oriented lineal actions of $F_n^{\pm}$, whose restrictions to $F_n$ are pairwise non-equivalent (because their Busemann pseudocharacters are not scalar multiples of one another); and since $F_n$ has index $2$ this implies that the actions were non-equivalent to start with. When $n = 2,3$, so the dimension of $H^1(F_n; \R)^\alpha_-$ equals $1$, there is a unique choice for $\varepsilon$ and for $p$ (up to sign) so \cref{prop:nonoriented:classification} shows that there exists a unique non-oriented lineal action.
\end{proof}

\begin{remark}
To the best of our knowledge, $F_n^\pm$ for $n\ge 4$ is the first known example of a group whose sets of oriented and non-oriented lineal actions are both uncountable. We should remark that an analogous proof shows that the same holds for $\Z^4 \rtimes (\Z/2\Z)$, with the action for example swapping the first and fourth, and second and third, basis vectors.
\end{remark}

We conclude with the following general statement:

\begin{theorem}[\cref{intro:thm:reverse}]\label{thm:reverse}

The poset $\Hyp(F_n^{\pm})$ has the following structure.
\begin{enumerate}
   \item $\Hyp_{gt}(F_n^{\pm}) = \emptyset$.
   \item $\Hyp_{qp}(F_n^{\pm}) = \emptyset$.
   \item $\Hyp_{\ell}(F_n^{\pm})$ is an antichain. The subposet $\Hyp_{\ell}^+(F_n^{\pm})$ has size $1$ for $n = 2$, and is uncountable for $n > 2$. Its complement has size $1$ for $n = 2,3$, and is uncountable for $n > 3$.
   \item $\Hyp_e(F_n^{\pm})$ consists of the smallest element.
\end{enumerate}
\end{theorem}

\begin{proof}
The last item is always true. As in \cref{intro:thm:global}, (i) follows from the fact that $F_n$, and thus $F_n^{\pm}$, has no free subgroups (\cref{cit:nofree}). Statements (ii) and (iii) follow from Propositions \ref{prop:reversing:qp} and \ref{prop:reversing:lineal} respectively.
\end{proof}

\section{Lamplighters with infinite lamps}\label{sec:lamplighters}

We now make a small detour from the world of Thompson groups into the world of lamplighters, which will feature prominently in a more detailed description of $\Hyp(F_n)$. By \emph{lamplighters} we mean groups of the form
\[
\Gamma \wr \Z \coloneqq \left( \bigoplus_{\Z} \Gamma \right) \rtimes \Z \text{,}
\]
where $\Z$ conjugates the direct sum by shifting the coordinates. Confining subsets of certain lamplighters, with $\Gamma$ a subgroup of $F_n$, will play an important role in our description of the poset $\X_0$ (and thus also $\X_1$). However, the structure of confining subsets of lamplighters has not been previously studied when the lamp groups $\Gamma$ are infinite. In this section, we introduce this setting and prove some general properties.

For a lamplighter $\Gamma \wr \Z$, write $L(\Gamma) \coloneqq \bigoplus_{\Z} \Gamma$, and consider the corresponding splitting $\Gamma \wr \Z = L(\Gamma) \rtimes \Z$. Write $z$ for the generator of $\Z$, so that conjugation by $z$ on $L(\Gamma)$ is just the familiar \emph{right shift}
\[\sigma \colon L(\Gamma) \to L(\Gamma) : \sigma(g)_i = g_{i-1}.\]
The character we will use is just the identity $\id\colon \Z\to\Z$, although in order to be more notationally suggestive we will write $\sigma$ instead of $\id$, as the next definition explains:

\begin{definition}[$\sigma$-confining]
Call an $\id$-confining subset \emph{$\sigma$-confining}, with the adverb ``strictly'' included as appropriate. Write $\Conf_\sigma(\Gamma \wr \Z)$ for the poset of equivalence classes of $\sigma$-confining subsets, and $\Conf^\ast_\sigma(\Gamma \wr \Z)$ for the subposet of strictly $\sigma$-confining subsets. Further related notation, for example $Z_\sigma$ and $\Hyp_{qp}^\sigma(\Gamma\wr\Z)$, will also be used similarly.
\end{definition}

We will also make use of the following notation.

\begin{notation}\label{not:projection}
We denote by $L_{>0}(\Gamma)$ the subgroup of $L(\Gamma)$ naturally isomorphic to $\bigoplus_{\Z_{>0}} \Gamma$, and we let $\pi_{>0} \colon L(\Gamma) \to L_{>0}(\Gamma)$ denote the retraction. We analogously denote $L_{\geq 0}(\Gamma)$, $\pi_{\geq 0}$, and so on.
\end{notation}

The notion of confining subset is particularly easy to understand in this case: 

\begin{remark}\label{rem:shift}
A subset $Q \subseteq L(\Gamma)$ represents an element of $\Conf_\sigma(\Gamma \wr \Z)$ if and only if it satisfies the following:
\begin{enumerate}
\item $\sigma(Q) \subseteq Q$;
\item For all $g \in L(\Gamma)$ there exists $k \geq 0$ such that $\sigma^k(g) \in Q$;
\item There exists $k \geq 0$ such that $\sigma^k(Q \cdot Q) \subseteq Q$.
\end{enumerate}
\end{remark}

We will mostly use the characterization from \cref{rem:shift} for the rest of the section. As a remark, we will often equivocate between a confining subset and its equivalence class, so for example we may refer to a confining subset as an element of $\Conf_\sigma(\Gamma \wr \Z)$, and no confusion should arise.

\subsection{Right-heavy confining subsets}

We will be interested in a special family of confining subsets:

\begin{definition}[Right-heavy]
\label{def:rh}

A $\sigma$-confining subset $Q \subseteq L(\Gamma)$ is \emph{right-heavy} if it contains all of $L_{\geq 0}(\Gamma)$. We denote by $\Conf_\sigma^{RH}(\Gamma \wr \Z)$ the subposet of all (equivalence classes of) right-heavy confining subsets; note that this includes $[L(\Gamma)]$.
\end{definition}

Recall that a non-elliptic hyperbolic structure on $\Gamma \wr \Z$ is called \emph{substandard} if it is dominated by the standard action of $\Gamma \wr \Z$ on the Bass--Serre tree associated with the ascending HNN-extension with base group $L_{\geq 0}(\Gamma)$ and stable letter $z$. This definition includes the standard lineal structure factoring through the quotient $\Gamma \wr \Z \to \Z$; all other substandard structures are quasi-parabolic. We denote by $\Hyp(\Gamma \wr \Z)^{subst}$ the subposet of substandard actions. The following is a consequence of the definitions:

\begin{proposition}[Substandard actions]
\label{prop:substandard}
The isomorphism from \cref{cit:confining_to_action}: 
\[\Conf_\sigma(\Gamma \wr \Z) \to \Hyp_{qp}^\sigma(\Gamma \wr \Z) \cup [\pm \sigma] \quad \colon \quad [Q] \to [Q \cup Z_\sigma]\]
induces an isomorphism $\Conf_\sigma^{RH}(\Gamma \wr \Z) \to \Hyp(\Gamma \wr \Z)^{subst}$. Under this isomorphism, the smallest element $[L(\Gamma)]$ is mapped to the standard lineal structure of $\Gamma \wr \Z$; and the largest element $[L_{\geq 0}(\Gamma)]$ is mapped to the standard quasi-parabolic structure of $\Gamma \wr \Z$. \qed
\end{proposition}

In case $\Gamma$ is finite cyclic, all $\sigma$-confining subsets are right-heavy \cite{sahana:wreath}. However, this is not the case when $\Gamma$ is infinite, as we will now see. This construction also provides, to our knowledge, the first explicit example of a confining subset that is not equivalent to any confining subgroup (and we will see more examples later). We should mention that it turns out the confining subset of the Baumslag--Solitar group $BS(1,n)$ denoted $Q^-$ in \cite{AR:BS}, which corresponds to an action of $BS(1,n)$ on the hyperbolic plane $\mathbb{H}^2$ via an embedding $BS(1,n)\to \PSL_2(\R)$, is also not equivalent to any confining subgroup; this is clear from the description of $Q^-$ in \cite[Lemma~3.1]{AR:BS} together with our result \cref{prop:conj_is_all}.

\begin{example}
\label{ex:balls}

Let $\Gamma$ be a finitely generated infinite group. Fix a finite symmetric generating set, let $B_i$ be the ball of radius $2^i$ in $\Gamma$ with respect to this generating set for $i \geq 0$, and set $B_i = \{ 1 \}$ for all $i < 0$.
Let $Q \coloneqq \bigoplus_{i \in \Z} B_i \subseteq L(\Gamma)$. We claim that $Q$ is $\sigma$-confining. It is symmetric because the $B_i$ are. It satisfies \StayInConf\ because the $B_i$ are nested, and it satisfies \ProdConf\ because moreover $B_i \cdot B_i \subseteq B_{i+1}$. As for \GetInConf, let $g \in L(\Gamma)$, say $g = (g_i)_{i \in \Z}$, and let $k \geq 0$ be large enough that $g_i = 1$ for all $i \leq -k$. Let $m \geq 0$ be such that $g_i \in B_m$ for all $i \in \Z$. Then $\sigma^{m+k}(g) \in Q$. To see that $Q$ is not equivalent to any $\sigma$-confining subgroup, note that $\langle Q \rangle = L_{\geq 0}(\Gamma)$, and no conjugate of this by a power of $z$ is contained in $Q$, so we are done by \cref{prop:equiv_to_subgroup}.
\end{example}

Of course, this construction can be modified in many different ways, which shows that $\Conf_\sigma(\Gamma \wr \Z)$ can be quite complicated for an infinite group $\Gamma$. However, for the rest of the paper, only right-heavy confining subsets will play a role. This is an aspect in which Thompson groups are more well-behaved than lamplighters.

\subsection{Variety of confining subsets}

Despite being more restrictive, right-heavy confining subsets still exhibit complex behavior. First, we note the following general construction of right-heavy confining subsets.

\begin{cit}[{\cite[Theorem 1.4]{sahana:wreath}}]
\label{cit:subgroups}

For a subgroup $H \leq \Gamma$, define
\[
Q_H \coloneqq \left(\bigoplus_{i < 0} H\right) \times L_{\geq 0}(\Gamma) \leq L(\Gamma).
\]
Then $Q_H$ is a right-heavy $\sigma$-confining subset, and the map $H \to Q_H$ defines an order-reversing embedding of the poset $\Sub(\Gamma)$ of subgroups of $\Gamma$ into $\Conf_\sigma^{RH}(\Gamma \wr \Z)$, sending $\{ 1 \}$ to the largest element $[L_{\geq 0}(\Gamma)]$ and $\Gamma$ to the smallest element $[L(\Gamma)]$. 
\end{cit}

It turns out that when $\Gamma$ is finite cyclic, all confining subsets are of this form \cite{sahana:wreath}. This already shows that the poset of right-heavy confining subsets can be very large.

\begin{corollary}
\label{cor:P(N):subgroups}

Let $\Gamma$ be a group that contains either an infinite free product $\ast_{i \in I} H_i$ or an infinite direct sum $\oplus_{i \in I} H_i$. Then $\Conf_\sigma^{RH}(\Gamma \wr \Z)$ contains a copy of $\mathcal{P}(\N)$.
\end{corollary}

\begin{proof}
This follows from \cref{cit:subgroups}, together with the fact that, under these assumptions $\Sub(\Gamma)$ contains a copy of $\mathcal{P}(\N)$. Indeed, in the case of the direct sum we may extract a countable subset of $I$ and find an infinite direct sum $H = \oplus_{i \in \N} H_i$. We then define, for $X \in \mathcal{P}(\N)$, the subgroup $H_X = \langle H_i \mid i \in X \rangle \leq H$. The case of free products is analogous (and already essentially contained in \cite{sahana:wreath}).
\end{proof}

\begin{example}
\label{ex:P(N):subgroups}
\cref{cor:P(N):subgroups} shows that $\Hyp(\mathbb{F} \wr \Z)$, where $\mathbb{F}$ denotes a non-abelian free group, contains a copy of $\mathcal{P}(\N)$ consisting of quasi-parabolic structures. This was established in \cite{sahana:wreath}, and was the first such example. The statement of \cref{cor:P(N):subgroups} concerning direct sums shows how to obtain this also without free subgroups, and even among amenable groups. For example, it holds for $F_n^c \wr \Z$, and for the finitely generated $3$-step solvable group $(\Z \wr \Z) \wr \Z$.
\end{example}

Next, we show that the existence of confining subsets that are not equivalent to confining subgroups persists among the right-heavy ones. The following construction is due to Simon Machado, whom we thank for this example:

\begin{example}
\label{ex:machado}

Let $\Gamma$ be a group containing an infinite cyclic group $Z = \langle z \rangle$ (i.e., a non-torsion group). We now construct a right-heavy confining subset in $\Gamma \wr \Z$ that is not equivalent to a confining subgroup. Let $\alpha \colon Z \to \R/\Z$ be a homomorphism such that the image of $z$ is irrational. It follows that the image $\alpha(Z)$ is dense in $\R/\Z$. For $i < 0$, we define $A_i \coloneqq \{ x \in Z \mid \alpha(x) \in (-2^i, 2^i) \}$. We set
\[
Q_\alpha \coloneqq \left(\bigoplus_{i < 0} A_i\right) \times L_{\geq 0}(\Gamma).
\]
Now \StayInConf\ follows from the fact that the $A_i$ are nested, and \GetInConf\ is immediate because $Q_\alpha$ is right-heavy. \ProdConf\ follows from the fact that $A_i \cdot A_i \subseteq A_{i+1}$. To see that $Q_\alpha$ is not equivalent to any $\sigma$-confining subgroup, note that every element of $A_i$ has a power that does not belong to $A_i$. This shows that $Q$ does not contain a subgroup, and in particular it cannot contain a conjugate of $\langle Q \rangle$, so we are done by \cref{prop:equiv_to_subgroup}.
\end{example}

Finally, let us remark that all of the examples that we have seen so far split as a direct sum over $\Z$. This need not be the case, as the next result shows:

\begin{example}
\label{ex:nonsplit}

Let $\Gamma$ be a group that contains an infinite direct sum $H = \oplus_{i \in \N} H_i$ (as in \cref{cor:P(N):subgroups}, the same construction works for infinite free products). We now construct a right-heavy confining subgroup for $\Gamma \wr \Z$ that  is not equivalent to any confining subset that splits as a direct sum over $\Z$. For $n \geq 2$, define $\pi_n \colon H \to H_{\geq n}$ to be the retraction onto $\langle H_i \mid i \geq n \rangle$ that kills the $H_i$ coordinate for $i = 1, \ldots, (n-1)$. We set
\[Q \coloneqq \{ (g_i)_{i \in \Z} \in L(\Gamma) \mid g_{-1} \in H\text{; }g_{-n} = \pi_n(g_{-1}) \text{ for all } n\in\N\}.\]
It follows from the definition, and the fact that $\pi_n$ is a homomorphism, that  $Q$ is right-heavy -- which implies \GetInConf\ -- and a subgroup -- which implies \ProdConf. Finally, \StayInConf\ follows from the following computation:
\[
\sigma(g)_{-n} = g_{-n-1} = \pi_{n+1}(g_{-1}) = \pi_{n+1} \pi_2(g_{-1}) = \pi_{n+1}(g_{-2}) = \pi_n(\sigma(g)_{-1}).
\]
Now suppose that there exists a confining subset $R$ that splits as a direct sum $\oplus_{i \in \Z} R_i$ and such that $Q$ is equivalent to $R$. In particular, by \cref{cor:conjugate}, there exist $p, q \in \Z$ such that $\sigma^p(R) \subseteq Q \subseteq \sigma^q(R)$. Up to replacing $R$ by a conjugate, we may assume that $q = 0$, i.e., $\sigma^p(R) \subseteq Q \subseteq R$. Since $Q \subseteq R$, it follows that $R$ must contain $R' \coloneqq \bigoplus_{i < 0} H_{\geq |i|} \times L_{\geq 0} (\Gamma)$. But then $\sigma^p(R)$ contains the element $(g_i)_{i \in \Z}$ where $g_{-1} \in H_{p+1}$ and $g_i = 1$ otherwise, and this is impossible since this element does not belong to $Q$.
\end{example}

This example crucially exploits infinite direct sums (or infinite free products), which are abundant in $F_n$. However, such examples can also arise in much simpler situations, even when $\Gamma$ is as small as $\Z /2 \Z \times \Z / 2\Z$ \cite[Example 4.16]{sahana:wreath}.

\subsection{Split representatives}

 Despite \cref{ex:nonsplit}, right-heavy confining subsets are equivalent to ones that split into direct sums, at least at the macroscopic level.

\begin{definition}[Split]
\label{def:splitreps:lamplighters}
A right-heavy confining subset $Q$ is \emph{split} if it splits as a product
\[Q = \pi_{<0}(Q) \cdot L_{\geq 0}(\Gamma).\]
\end{definition}

Here is a useful equivalent characterization:

\begin{lemma}
\label{lem:splitreps:char:lamplighters}
Let $Q \in \Conf_{\sigma}(\Gamma \wr \Z)$. Then $Q$ is split if and only if $Q = Q \cdot L_{\geq 0}(\Gamma)$.
\end{lemma}

\begin{proof}
If $Q$ is split, then the property in the statement holds. Conversely, if $Q$ satisfies this property, we must show that $Q$ is split. First, $Q$ contains $1 \cdot L_{\geq 0}(\Gamma)$, so it is right-heavy. Clearly $Q \subseteq \pi_{<0}(Q) \cdot L_{\geq 0}(\Gamma)$. For the opposite inclusion, if $g \cdot h \in \pi_{<0}(Q) \cdot L_{\geq 0}(\Gamma)$, let $q \in Q$ be such that $\pi_{<0}(q) = g$; then $g \cdot h = q \cdot \pi_{\geq 0}(q)^{-1} h \in Q \cdot L_{\geq 0}(\Gamma)$.
\end{proof}

Note how all right-heavy confining subsets we have looked at so far are split. This is not a coincidence.

\begin{lemma}
\label{lem:splitreps:lamplighters}

A right-heavy confining subset $Q \in \Conf_\sigma^{RH}(\Gamma \wr \Z)$ is equivalent to the split confining subset $Q_{split} \coloneqq Q \cdot L_{\geq 0}(\Gamma)$.
\end{lemma}

We will refer to the split confining subset in \cref{lem:splitreps:lamplighters} as the \emph{split representative} of $Q$.

\begin{proof}
We first verify that $Q_{split}$ is a confining subset. It is symmetric by construction. For \StayInConf, we have:
\[\sigma(Q_{split}) = \sigma(Q) \cdot \sigma(L_{\geq 0}(\Gamma)) \subseteq Q \cdot L_{\geq 0}(\Gamma).\]
\GetInConf\ is clear since $Q_{split}$ contains $Q$. As for \ProdConf, let $k \geq 0$ be such that $\sigma^k(Q \cdot Q) \subseteq Q$. Then for $q_1 h_1, q_2 h_2 \in Q_{split}$ we observe that setting $h_1' \coloneqq h_1^{q_2} \in L_{\geq 0}(\Gamma)$ we have:
\[\sigma^k(q_1 h_1 q_2 h_2) = \sigma^k(q_1 q_2 h_1' h_2) \in \sigma^k(Q \cdot Q) \cdot \sigma^k(L_{\geq 0}(\Gamma) \cdot L_{\geq 0}(\Gamma)) \subseteq Q\cdot L_{\geq 0}(\Gamma);\]
and so $\sigma^k(Q_{split} \cdot Q_{split}) \subseteq Q_{split}$.

Finally, we check that $Q$ and $Q_{split}$ are equivalent. Clearly $Q \subseteq Q_{split}$, so by \cref{cor:conjugate} it suffices to show that there exists $k \in \Z$ such that $\sigma^k(Q_{split}) \subseteq Q$. We choose $k$ such that $\sigma^k(Q \cdot Q) \subseteq Q$, and verify:
\[\sigma^k(Q_{split}) = \sigma^k(Q \cdot L_{\geq 0}(\Gamma)) \subseteq \sigma^k(Q \cdot Q) \subseteq Q;\]
where we used that $L_{\geq 0}(\Gamma)\subseteq Q$, i.e., that $Q$ is right-heavy.
\end{proof}

\section{The largest hyperbolic structure}\label{sec:largest}

From this section onwards, we study the local structure of $\Hyp(F_n)$. \cref{intro:thm:global} reduces this to understanding the poset $\X_0$ of quasi-parabolic structures whose Busemann pseudocharacter is equivalent to $\chi_0$. Recall that $\chi_0$ restricts to $\Z^n = \langle a_0, \ldots, a_{n-1} \rangle$ as the map $\Z^n \to \Z : \prod a_i^{k_i} \mapsto k_0$, in particular $\chi_0(a_0) = 1$. The element $a_0$ has support equal to the interval $(0, r_0)$, and its slope at $0$ is exactly $n$.
Recall moreover that $\X_0$ is identified with the poset $\Conf^\ast_{\chi_0}(F_n)$ of strictly $\chi_0$-confining subsets of $F_n'$.

\begin{notation}
From now on, we denote $\X \coloneqq \X_0$, $\chi \coloneqq \chi_0$, $a \coloneqq a_0$, $r \coloneqq r_0$.
We still use the notation $a_1, \ldots, a_{n-1}$ for the other generators of $\Z^n$, which are supported on $(\ell_i, r_i)$, with in particular $r < \ell_1$. 
\end{notation}

The main goal of this section is to identify the largest element of $\X$, from part (i) of \cref{intro:thm:local}.

\begin{proposition}
\label{prop:contains}
Let $Q$ be a $\chi$-confining subset and let $t \in (0,r)$. Then there exists $k \geq 0$ such that $F_n'[t.a^k,1) \subseteq Q$.
\end{proposition}

\begin{proof}
By \cref{prop:empty_conf}, the centralizer of $a$ in $F_n'$ is contained in $Q$. In particular $F_n'[r, 1) \subseteq Q$. We claim that there exists $q \in Q$ such that for some $x \in (0,r)$ we have $x.q > r$. First let $q_* \in F_n'$ be any element satisfying $x_*.q_*>r$, for some $x_* \in (0,r)$, that is moreover supported on $(0,\ell_1)$. By \cref{lem:conf:cyclic}, there exists $i \geq 0$ such that $q \coloneqq q_*^{a^i} \in Q$. Let $x \coloneqq x_*.a^i$. Then $x \in (0,r)$ and
\[ x.q = (x_*.a^i).q_*^{a^i} = (x_*.q_*).a^i = x_*.q_* > r \]
since $a^i$ fixes $[r,1]$ pointwise.

By now we know that $Q$ contains $F_n'[r,1)$ and contains $q$. Since $x.q>r$ implies $F_n'[x,1)^q \subseteq F_n'[r,1)$, we have
\[ F_n'[x,1) \subseteq q F_n'[r,1)q^{-1} \subseteq Q^3\text{.} \]
By \cref{lem:products}, there exists $j \geq 0$ such that $(Q^3)^{a^j} \subseteq Q$. Therefore
\[
F_n'[x.a^j,1) = F_n'[x, 1)^{a^j} \subseteq (Q^3)^{a^j} \subseteq Q.
\]
Since $x.a^j < r$, we can choose $k \ge 0$ such that $x.a^j<t.a^k$, so $F_n'[t.a^k,1)\subseteq Q$ as desired.
\end{proof}

Of course each $F_n'[t,1)$ is easily seen to itself be a $\chi$-confining subgroup. For any $t,t'\in (0,r)$, we can choose $k,\ell\in\Z$ such that $t<t'.a^k$ and $t'<t.a^\ell$, so by \cref{cor:conjugate} the $\chi$-confining subgroups $F_n'[t,1)$ and $F_n'[t',1)$ are equivalent. Let us denote this equivalence class by $[F_n'[*,1)]$. Now \cref{prop:contains} implies:

\begin{corollary}
\label{cor:largest}

The element $[F_n'[*,1)]$ of the poset $\Conf^\ast_{\chi_0}(F_n)$ is the largest element.\qed
\end{corollary}

This largest element is the most familiar non-elementary hyperbolic action of $F_n$. This leads to the following more precise version of \cref{intro:thm:largest}:

\begin{theorem}[\cref{intro:thm:largest}]\label{thm:largest:geom}
The poset $\Hyp(F_n)$ has exactly two distinct non-elementary maximal hyperbolic structures. The first one is represented by the action on the Bass--Serre tree of the ascending HNN-extension expression of $F_n$ with base $F_n[t, 1]$ and stable letter $a$, for some $t \in (0, r)$. The second one is represented by the same action precomposed by conjugation by the homeomorphism $x \mapsto (1 - x)$.
\end{theorem}

\begin{proof}
By \cref{intro:thm:global}, the non-elementary maximal hyperbolic structures correspond to the largest elements of the posets $\X_0$ and $\X_1$. By \cref{cor:largest}, the largest element of $\X_0$ is $[F_n'[*,1)]$. Since $\chi_0$ is discrete and $F_n'[t,1)$ is a subgroup, by \cref{prop:subgroups_to_trees} we know that this largest element of $\X_0$ is represented by an action on a tree, and the second part of the statement of \cref{prop:subgroups_to_trees} implies that it is the claimed action. The analogous statement holds for the largest element of $\X_1$.
\end{proof}

\medskip

\cref{prop:contains} shows that every confining subset $Q$ contains $F_n'(t, 1)$ for some $t \in (0, r)$. The minimal $t$ for which this holds has a special meaning, in case $Q$ is a confining sub\emph{group}:

\begin{proposition}
\label{prop:contains:subgroup}

Let $Q$ be a $\chi$-confining subgroup. Then the following numbers coincide:
\begin{itemize}
\item The infimum of $t \in [0, 1)$ such that $Q$ contains $F_n'(t, 1)$;
\item The supremum of $t \in [0, 1)$ such that every element of $Q$ fixes $t$.
\end{itemize}
In particular, $Q = F_n'$ if and only if it has no global fixed point in $(0, 1)$, so every strictly $\chi$-confining subgroup has a global fixed point $t\in (0,r)$.
\end{proposition}

\begin{proof}
Let $t_{sup}$ be the supremum of all global fixed points of $Q$ in $(0, 1)$. By \cref{prop:contains}, there exists $t \in (0, r)$ such that $Q$ contains $F_n'(t, 1)$; it follows that $t_{sup} \leq t$, and up to choosing a larger $t$ we may assume that $t_{sup} < t$. Being an accumulation point of global fixed points, $t_{sup}$ is itself a global fixed point.

Now let $t_{inf}$ be the infimum of the orbit $t.Q$. Then for every $q \in Q$ we have that $F_n'(t.q, 1) = F_n'(t,1)^q \subseteq Q^3 = Q$, where we are using that $Q$ is a group. Choosing a sequence $q_i \in Q$ such that $t.q_i$ converges to $t_{inf}$ we see that $F_n'(t_{inf}, 1) \subseteq Q$, which shows that $t_{inf}$ is the quantity in the first bullet point. Moreover, $t_{inf}$ is fixed, because $t.Q$ is an orbit and $Q$ is a group. Thus by definition $t_{inf} \leq t_{sup}$, and the other inequality follows from the argument in the first paragraph above.
\end{proof}

Motivated by this proposition, can now define the posets $\LL_t$ from \cref{intro:thm:local}. We will define $\LL_t$ for all $t \in (0, r)$, but we will later see that $\LL_t = \LL_{t.a^{-1}}$ (\cref{lem:LLt:equality}).

\begin{definition}[Lamplike confining subsets]
\label{def:lamplike:moving}

We define $\LL_t$ to be the set of confining subsets that admit $t$ as a global fixed point, where $t \in (0, r)$. By abuse of notation, we also denote by $\LL_t$ the subposet of $\Conf^\ast_\chi(F_n)$ consisting of equivalence classes of confining subsets that admit a representative in $\LL_t$.
We say that a confining subset $Q$ is \emph{lamplike} if $[Q] \in \LL_t$ for some $t$.
\end{definition}

Note that the property of being lamplike is invariant under the equivalence relation on confining subsets, thanks to \cref{cor:conjugate}. We will see a strengthening of this in \cref{lem:LLt:upper}.

\medskip

We can now prove \cref{intro:thm:trees} from the introduction:

\begin{theorem}[\cref{intro:thm:trees}]
    Every cobounded quasi-parabolic action of $F_n$ on a simplicial tree corresponds to a lamplike hyperbolic structure, and is represented by the action on the Bass--Serre tree of a strictly ascending HNN-extension expression of $F_n$.
\end{theorem}

\begin{proof}
    Consider a cobounded quasi-parabolic action of $F_n$ on a tree, which is automatically simplicial by \cref{rem:simplicial}. By \cref{intro:thm:global}, the corresponding Busemann pseudocharacter is $\chi$ (equal to $\chi_0$ or $\chi_1$), and so this action represents an element of $\X$. By \cref{prop:trees_to_subgroups}, the confining subset may be chosen to be a subgroup, which must be strictly confining since the action is quasi-parabolic. Then \cref{prop:contains:subgroup} gives a global fixed point; the structure is therefore lamplike, by definition. \cref{prop:subgroups_to_trees} shows that, in turn, this confining subgroup corresponds to the action on the Bass--Serre tree of an ascending HNN-extension expression of $F_n$.
\end{proof}

\section{Lamplike confining subsets at $t$}\label{sec:lamplike}

The goal of this section is to study the subposet $\LL_t$ for a fixed $t \in (0, r)$. Recall that by definition, $Q \in \LL_t$ if $Q$ is a $\chi$-confining subset that fixes $t$.

\begin{notation}
For the rest of this section, we fix an element $t \in (0, r)$, and denote by $t_k \coloneqq t.a^k$ for $k\in\Z$. In particular $t = t_0$. Note that the points $t_k \in (0, r)$ are linearly ordered, and accumulate to $0$ as $k \to -\infty$ and to $r$ as $k \to \infty$.
\end{notation}

The following easy fact is the starting point of the connection to lamplighters:

\begin{lemma}
\label{lem:LLt:orbit}
If $Q \in \LL_t$, then $t_k$ is a global fixed point of $Q$ for every $k \leq 0$. Moreover, if every element of $Q$ fixes a neighborhood of $t$, then every element of $Q$ fixes a neighborhood of $t_k$ for every $k \leq 0$.
\end{lemma}

\begin{proof}
In general, if $x \in (0, r)$ is a global fixed point of $Q$, then so is $x.a^{-1}$. Indeed,
\[x.a^{-1}q = (x.q^a).a^{-1} = x.a^{-1}\]
for any $q\in Q$, where we used that $q^a \in Q$ by \StayInConf. The same holds for neighborhoods.
\end{proof}

Our main goal is \cref{intro:thm:lamplike}, which describes $\LL_t$ in terms of posets of right-heavy confining subsets of a certain lamplighters. The key property of confining subsets of $F_n$ is the fact that they contain $F_n'[t, 1)$ for some $t \in (0, r)$ (\cref{prop:contains}); this is what will ensure that the resulting confining subsets are all right-heavy.

\subsection{Split representatives}

The first step in the proof of \cref{intro:thm:lamplike} is to find ``split'' representatives, analogous to the situation for lamplighters in \cref{lem:splitreps:lamplighters}.

\begin{definition}
\label{def:splitreps:thompson}
A confining subset $Q \in \LL_t$ is \emph{$t$-split} if
\[Q = \{ g \in F_n' \mid g\text{ fixes } t \text{ and there exists } q \in Q \text{ such that } g|_{[0, t]} = q|_{[0, t]} \}.\]
\end{definition}

In particular, a $t$-split confining subset must contain $F_n'[t, 1)$. The definition is analogous to \cref{def:splitreps:lamplighters}; however we need to make the statement internal because the stabilizer $\Fix_{F_n}(t)$ does not necessarily split cleanly as a product (this fails precisely when $t \in \Q \setminus \Z[1/n]$, by \cref{lem:fix:rational}). We now prove the analogs of Lemmas \ref{lem:splitreps:char:lamplighters} and \ref{lem:splitreps:lamplighters}.

\begin{lemma}
\label{lem:splitreps:char:thompson}
Let $Q \in \LL_t$. Then $Q$ is $t$-split if and only if $Q = Q \cdot F_n'[t, 1)$.
\end{lemma}

\begin{proof}
Clearly if $Q$ is $t$-split, then the property in the statement holds. Conversely, suppose that $Q$ satisfies this property, and let $g \in F_n'$ such that there exists $q \in Q$ with $g|_{[0, t]} = q|_{[0, t]}$. Then $h \coloneqq q^{-1}g$ acts trivially on $[0, t]$ and belongs to $F_n'$, so $h \in F_n'[t, 1)$. Thus $g = q h \in Q \cdot F_n'[t, 1) = Q$.
\end{proof}

\begin{lemma}
\label{lem:splitreps:thompson}

Every $Q \in \LL_t$ is equivalent to the $t$-split confining subset \[Q_{split} \coloneqq Q \cdot F_n'[t, 1).\]
\end{lemma}

We will refer to the split confining subset in \cref{lem:splitreps:thompson} as a \emph{$t$-split representative} of $Q$. The proof will be similar to \cref{lem:splitreps:lamplighters} thanks to \cref{prop:contains}.

\begin{proof}
Fix $s \in (0, r)$ such that $Q$ contains $F_n'[s, 1)$, which is possible by \cref{prop:contains}.
We first verify that $Q_{split}$ is a confining subset. It is symmetric by construction. For \StayInConf, given $z \in \Z^n$ such that $\chi(z) \geq 0$, we compute:
\[Q_{split}^z = Q^z \cdot F_n'[t, 1)^z \subseteq Q \cdot F_n'[t, 1),\]
since both $Q$ and $F_n'[t, 1)$ are $\chi$-confining.
\GetInConf\ is clear since $Q_{split}$ contains $Q$.
As for \ProdConf, let $k \geq 0$ be such that $(Q^4)^{a^k} \subseteq Q$ and $t.a^k > s$; such a $k$ exists thanks to \cref{lem:products}.
Now take $q_1 h_1, q_2 h_2 \in Q_{split}$. Then we have
\[
(q_1 h_1 q_2 h_2)^{a^{2k}} \in \bigg(\big(Q^{a^k} \cdot F_n'[t.a^k, 1)\big)^2\bigg)^{a^k} \subseteq \bigg(\big(Q \cdot F_n'[s, 1)\big)^2\bigg)^{a^k} \subseteq (Q^4)^{a^k} \subseteq Q \text{.}
\]

Finally, we check that $Q$ and $Q_{split}$ are equivalent. Clearly $Q \subseteq Q_{split}$, so by \cref{cor:conjugate} it suffices to show that there exists $z \in \Z^n$ such that $Q_{split}^z \subseteq Q$. We choose $z = a^{2k}$ where $k$ is such that such that $(Q \cdot Q)^{a^k} \subseteq Q$ and $t.a^k > s$, and verify:
\[
Q_{split}^{a^{2k}} = \big(Q^{a^k} \cdot F_n'[t.a^k, 1)\big)^{a^k} \subseteq \big(Q \cdot F_n'[s, 1)\big)^{a^k} \subseteq (Q \cdot Q)^{a^k} \subseteq Q. \qedhere
\]
\end{proof}

\subsection{Isolated lamps}

Before we prove \cref{intro:thm:lamplike}, which will involve some subtleties depending on the arithmetic properties of $t$, we start by describing a smaller subposet $\LL_{(t)}$, which is easier to describe and will exhibit less dependency on $t$.

\begin{definition}
\label{def:Fnt:isolated}

We define
\[F_n(t_\bullet) \coloneqq \{ g \in F_n \mid g \text{ fixes a neighborhood of } t_k \text{ for all } k \leq 0 \} \leq F_n\]
and
\[
F_n'(t_\bullet) \coloneqq F_n(t_\bullet) \cap F_n'.
\]
It is easy to see that $F_n'(t_\bullet)$ is a $\chi$-confining subgroup. We define $\LL_{(t)}$ to consist of those $Q \in \LL_t$ such that $Q \subseteq F_n'(t_\bullet)$. As with $\LL_t$, by abuse of notation we also use $\LL_{(t)}$ to denote the corresponding subposet of $\LL_t$. Note that $Q \in \LL_{(t)}$ if and only if it fixes a neighborhood of $t$: this is a consequence of \cref{lem:LLt:orbit} and the fact that $Q \subseteq F_n'$.
\end{definition}

\begin{remark}
\label{rem:fixzero}
    If $g \in F_n(t_\bullet)$, then $g$ fixes elements arbitrarily close to $0$, so its slope at $0$ must be equal to $1$. Being piecewise linear, it follows that $g$ must fix pointwise a neighborhood of $0$. Note that this does not require that a neighborhood of any $t_k$ is fixed. In particular, it also holds for the group $F_n[t_\bullet]$ we will work with later (\cref{def:Fnt}).
\end{remark}

In order to deal with isolated lamps, we will use the following natural analogs of $t$-split representatives.

\begin{definition}
\label{def:splitreps:isolated}

A confining subset $Q \in \LL_{(t)}$ is called \emph{$(t)$-split} if
\begin{align*}
Q = \{ g \in F_n' \mid & \,\, g \text{ fixes a neighborhood of } t \text{ and}\\
& \,\, \text{there exists } q \in Q \text{ such that } g|_{[0, t]} = q|_{[0, t]} \}.
\end{align*}
\end{definition}

We have the following analog of \cref{lem:splitreps:thompson}.

\begin{lemma}
\label{lem:splitreps:isolated}

Every $Q \in \LL_{(t)}$ is equivalent to the $(t)$-split confining subset $Q\cdot F_n'(t, 1)$.
\end{lemma}

\begin{proof}
Let $Q_{split}=Q\cdot F_n'[t,1)$ be the $t$-split confining subset equivalent to $Q$ that was given by \cref{lem:splitreps:thompson}. Then \cref{prop:join} gives a confining subset $Q_{split} \cap F_n'(t_\bullet)$ and $Q \succeq_\chi Q_{split} \cap F_n'(t_\bullet) \succeq_\chi Q_{split} \sim_\chi Q$. Therefore $Q$ is equivalent to $Q_{split} \cap F_n'(t_\bullet)$. Moreover
\[Q_{split} \cap F_n'(t_\bullet) = Q \cdot F_n'[t, 1) \cap F_n'(t_\bullet) = Q \cdot F_n'(t, 1). \qedhere\]
\end{proof}

Next we relate $F_n'(t_\bullet)$ to a direct sum, which will correspond to $L_{< 0}$ in a lamplighter.

\begin{lemma}
\label{lem:Fnt:isolated:iso}

The restrictions $F_n(t_\bullet)|_{[t, 1]}$ and $F_n(t_\bullet)|_{[t_{k}, t_{k+1}]}$ ($k<0$) yield an isomorphism
\[\zeta^t \colon F_n(t_\bullet) \to \left( \bigoplus_{k < 0} F_n(t_{k}, t_{k+1}) \right) \times F_n(t, 1]\]
and an injective homomorphism
\[\zeta^t \colon F_n'(t_\bullet) \to \left( \bigoplus_{k < 0} F_n(t_{k}, t_{k+1}) \right) \times F_n(t, 1)\]
that surjects onto each factor.
\end{lemma}

\begin{proof}
Every element of $F_n(t_\bullet)$ fixes every $t_k$ for $k \leq 0$, so the restrictions to $[t_{k}, t_{k+1}]$ ($k<0$) and to $[t, 1]$ define a map
\[\zeta^t \colon F_n(t_\bullet) \to \left( \prod_{k < 0} F_n[t_{k}, t_{k+1}] \right) \times F_n[t, 1].\]
Since every element fixes a neighborhood of $t$ and each $t_k$, the restrictions actually belong to $F_n(t_{k}, t_{k+1})$ and $F_n(t, 1]$. Moreover, by \cref{rem:fixzero}, if $g \in F_n'(t_\bullet)$ then $g$ fixes a neighborhood~$I$ of $0$, which implies that $g|_{[t_{k}, t_{k+1}]}$ is the identity for $|k|$ large enough. All in all we have proved that the homomorphism $\zeta^t$ can be redefined with the desired target:
\[\zeta^t \colon F_n(t_\bullet) \to \left( \bigoplus_{k < 0} F_n(t_{k}, t_{k+1}) \right) \times F_n(t, 1].\]
Since $F_n(t_\bullet)$ acts faithfully on $[0, 1]$, and the action fixes $t_k$ for all $k\le 0$, $\zeta^t$ is injective. Moreover, given an element $h$ in the direct sum, we can extend its non-trivial coordinates to elements of $F_n$, and multiply them together to obtain an element $g \in F_n$ such that $\zeta^t(g) = h$. The fact that the restrictions of $h$ to $[t, 1)$ and to $[t_{k}, t_{k+1}]$ for each $k<0$ fix neighborhoods of the respective endpoints ensures that $g \in F_n(t_\bullet)$. Thus, $\zeta^t$ is an isomorphism.

If $g \in F_n'(t_\bullet)$ then $\chi_1(g) = 0$ and so the restriction of $g$ to $[t, 1]$ defines an element of $F_n(t, 1)$. Thus we obtain the injective homomorphism
\[\zeta^t \colon F_n'(t_\bullet) \to \left( \bigoplus_{k < 0} F_n(t_{k}, t_{k+1}) \right) \times F_n(t, 1).\]
We are left to prove that the image surjects onto both factors. We prove this for the factor $F_n(t, 1)$, and the other proof is similar. Let $g \in F_n(t_\bullet) \cap F_n[0, 1)$ be an element with prescribed image in $F_n(t, 1)$. Then $g \in F_n^c$, so using proximality (\cref{cit:proximal}) we find an element $c \in F_n$ such that $g^c$ is supported on a compact subset of $(t_{-1}, t)$. Then $(g^c)^{-1} g = [c, g] \in F_n'(t_\bullet)$ and has the same image as $g$ in $F_n(t, 1)$.
\end{proof}

We are now ready to prove a version of \cref{intro:thm:lamplike} for the poset $\LL_{(t)}$, which will make it easier to prove the actual theorem for $\LL_t$. Recall the poset $\Hyp(F_n^c \wr \Z)^{subst}$ of substandard actions of $F_n^c \wr \Z$, that is, (non-elliptic) actions dominated by the standard action on the Bass--Serre tree.

\begin{proposition}
\label{prop:isolatedlamps}

Let $t \in (0, r)$. Then $\LL_{(t)}$ is isomorphic to $\Hyp(F_n^c \wr \Z)^{subst}$.
\end{proposition}

\begin{proof}
By \cref{prop:rstab:open}, there exists an isomorphism $\eta^t_{-1} \colon F_n(t_{-1}, t) \to F_n^c$. For all $k < 0$ define the isomorphism $\eta^t_{k-1} \colon F_n(t_{k-1}, t_{k}) \to F_n^c$ via $\eta^t_{k-1}(g) = \eta^t_{-1}(g^{a^{|k|}})$. This defines an isomorphism
\[
\eta^t \colon \left( \bigoplus_{k < 0} F_n(t_{k}, t_{k+1}) \right) \times F_n(t, 1) \to \left( \bigoplus_{k < 0} F_n^c \right) \times F_n(t, 1),
\]
acting as $\eta_k^t$ on the $F_n(t_k,t_{k+1})$ factor and the identity on the $F_n(t,1)$ factor. Recall the map $\zeta^t$ from \cref{lem:Fnt:isolated:iso}, and set $\theta^t\coloneqq \eta^t\zeta^t$, so $\theta^t$ is an injective homomorphism
\[
\theta^t \colon F_n'(t_\bullet) \to \left( \bigoplus_{k < 0} F_n^c \right) \times F_n(t, 1).
\]
We adapt the definition of the shift $\sigma$ to consider it as a (non-injective, surjective) endomorphism
\[
\sigma \colon \bigoplus_{k < 0} F_n^c \to \bigoplus_{k < 0} F_n^c \quad\text{ defined via }\quad \sigma(g)_i = g_{i-1}.
\]
Now, modulo the right factor of $F_n(t, 1)$, the map $\theta^t$ intertwines conjugation by $a$ with the shift $\sigma$:
\begin{equation}
\label{intertwine}
\pi_{<0} (\theta^t(g^a)) = \pi_{<0}((\sigma\times\id)\theta^t (g)) = \sigma(\pi_{<0}(\theta^t(g)));
\end{equation}
this is a direct consequence of the way we defined $\theta^t$.

We are now ready to define the desired isomorphism $\LL_{(t)}\to \Hyp(F_n^c \wr \Z)^{subst}$, which we will denote by $\Phi$. We identify $\Hyp(F_n^c \wr \Z)^{subst}$ with $\Conf_\sigma^{RH}(F_n^c \wr \Z)$, the poset of (equivalence classes of) right-heavy confining subsets, including $[L(F_n^c)]$ (\cref{prop:substandard}). By \cref{lem:splitreps:isolated}, we are allowed to work with $(t)$-split confining subsets, in the sense of \cref{def:splitreps:isolated}. For a $(t)$-split confining subset $Q \in \LL_{(t)}$, we define $\Phi(Q) \subseteq L(F_n^c)$ to be
\[\Phi(Q) \coloneqq \{ g \in L(F_n^c) \mid \pi_{<0}(g) \in \pi_{<0}(\theta^t(Q)) \}.\]
In words, $\Phi(Q)$ is the set of elements in $L(F_n^c)$ whose projection to the negative coordinates matches that of the image of some element of $Q$ under the map $\theta^t$. Equivalently, $\Phi(Q)$ is the largest subset of $L(F_n^c)$ that satisfies
\[\pi_{<0}(\Phi(Q)) = \pi_{<0}(\theta^t(Q)).\]

We start by noticing that $\Phi(Q)$ is right-heavy, and moreover it is split, in the sense of \cref{def:splitreps:lamplighters}. It is clearly symmetric, and the fact that it is a confining subset is a consequence of the key property of the map $\Phi$:
\begin{equation}
\label{intertwine:Q}
\sigma(\pi_{<0}(\Phi(Q))) = \pi_{<0}(\sigma(\Phi(Q))) = \pi_{<0}(\theta^t(Q^a)),
\end{equation}
which is just a rephrasing of \cref{intertwine}. We check the three properties, as they are stated in \cref{rem:shift}, appealing to \cref{intertwine:Q} without further mention. For \StayInConf:
\[\sigma(\pi_{<0}(\Phi(Q))) = \pi_{<0}(\theta^t(Q^a)) \subseteq \pi_{<0}(\theta^t(Q)).\]
\GetInConf\ is clear since $\Phi(Q)$ is right-heavy. As for \ProdConf, let $k \geq 0$ be such that $(Q \cdot Q)^{a^k} \subseteq Q$, which exists by \cref{lem:conf:cyclic}. Then
\[\sigma^k(\pi_{<0}(\Phi(Q) \cdot \Phi(Q))) = \pi_{<0}(\theta^t((Q \cdot Q)^{a^k})) \subseteq \pi_{<0}(\theta^t(Q))= \pi_{<0}( \Phi(Q)).\]

Two $(t)$-split confining subsets in $F_n$ can be distinguished by their restrictions to $[0, t]$. Similarly, two split confining subsets of $L(F_n^c)$ can be distinguished by their images under the projection $\pi_{<0}$. This implies that $Q_1 \subseteq Q_2$ if and only if $\Phi(Q_1) \subseteq \Phi(Q_2)$, for $(t)$-split confining subsets $Q_1$ and $Q_2$. Thanks to \cref{cor:conjugate}, we can deduce that the preorders $\preceq_\chi, \preceq_\sigma$ are also well-behaved, as we now explain. First, $Q_1 \preceq_\chi Q_2$ if and only if $Q_2^{a^k} \subseteq Q_1$ for some $k\ge 0$, which is equivalent to $Q_2^{a^k} \cdot F_n'(t, 1) \subseteq Q_1$ for some $k\ge 0$. This happens if and only if $\Phi(Q_2^{a^k} \cdot F_n'(t, 1)) \subseteq \Phi(Q_1)$, or equivalently $\sigma^k(\Phi(Q_2))\cdot L_{\geq 0}(F_n^c) \subseteq \Phi(Q_1)$. This in turn is equivalent to saying $\sigma^k(\Phi(Q_2)) \subseteq \Phi(Q_1)$ for some $k\ge 0$, which is the same as $\Phi(Q_1) \preceq_\sigma \Phi(Q_2)$. In the second, respectively third-to-last, equivalence, note that the reason we moved to a $(t)$-split representative (using \cref{lem:splitreps:isolated}), respectively from a split representative (using \cref{lem:splitreps:lamplighters}), is because $\Phi$ is only defined on $(t)$-split representatives, and moreover always produces a split representative. The fact that $Q_1\preceq_\chi Q_2$ is equivalent to $\Phi(Q_1)\preceq_\sigma \Phi(Q_2)$ tells us that $\Phi$ is well defined on equivalence classes $[Q]$, and is injective.

By now, $\Phi$ is a well defined injective poset map from $\LL_{(t)}$ to $\Conf_\sigma^{RH}(F_n^c \wr \Z)$. To see that $\Phi$ is an isomorphism, we can now simply construct the inverse $\Phi^{-1}$, defined on a split representative $R \in \Conf_\sigma^{RH}(F_n^c \wr \Z)$ (which exists thanks to \cref{lem:splitreps:lamplighters}) via
\[\Phi^{-1}(R) \coloneqq \{ g \in F_n'(t_\bullet) \mid \pi_{<0}(\theta^t (g)) \in \pi_{<0}(R) \},\]
and apply analogs of all the above arguments to this. The fact that there exist elements of $F_n'(t_\bullet)$ that achieve a given image under $\pi_{<0}\theta^t$ is part of \cref{lem:Fnt:isolated:iso}.
\end{proof}

\begin{remark}
\label{rem:lamplike:subgroup}
    Under the isomorphism in \cref{prop:isolatedlamps}, the property of being represented by a confining subgroup is preserved. Indeed, consider an element of $\LL_{(t)}$ represented by a confining subgroup $Q$. By \cref{lem:splitreps:isolated}, we may assume that $Q$ is $(t)$-split, and still a subgroup. Then the corresponding element of $\Conf_\sigma^{RH}(F_n^c \wr \Z)$ is $\Phi(Q) = \pi_{<0}^{-1}(\pi_{<0} \theta^t(Q))$, which is a group since all maps involved are homomorphisms. Conversely, let $R \in \Conf_\sigma^{RH}(F_n^c \wr \Z)$ be a right-heavy confining subgroup. By \cref{lem:splitreps:char:lamplighters}, we may assume that $R$ is split, and still a subgroup. Then the corresponding element of $\LL_{(t)}$ is represented by $\Phi^{-1}(R) = (\pi_{<0} \theta^t)^{-1}(\pi_{<0}(R))$, which again is a group since all maps involved are homomorphisms.

    The same correspondence will hold for all of the isomorphisms in \cref{intro:thm:lamplike}, with the same proof. As a consequence, \cref{ex:machado} can be translated into an example of a $\chi$-confining subset of $F_n$ that is not equivalent to any $\chi$-confining subgroup.
\end{remark}

Note that the non-strictly confining subset $L(F_n^c)$ in $\Conf_\sigma^{RH}(F_n^c \wr \Z)$ is mapped under $\Phi^{-1}$ to the strictly confining subset $F_n'(t_\bullet)$ in $\LL_{(t)}$; in particular, this isomorphism is just an isomorphism of abstract posets, and does not always respect the notions of lineal and quasi-parabolic.

This allows us to obtain part of the largeness statement in \cref{intro:thm:local}:

\begin{corollary}
\label{cor:P(N):LLt}

For each $t \in (0, r)$ the poset $\LL_t$ contains a copy of $\mathcal{P(\N)}$.
\end{corollary}

\begin{proof}
\cref{ex:P(N):subgroups} gives a copy of $\mathcal{P}(\N)$ inside $\Hyp(F_n^c \wr \Z)^{subst}$. \cref{prop:isolatedlamps} translates this into a copy of $\mathcal{P}(\N)$ in $\LL_{(t)}$ and hence in $\LL_t$. 
\end{proof}

We are now ready to prove \cref{intro:thm:lamplike}. The fact that elements can act non-trivially near $t$ will require us to treat values of $t$ with different arithmetic properties separately, so we split these proofs up into subsections.

\subsection{Irrational lamps}

The irrational case is easiest. Let $t \in (0, r) \setminus \Q$. Then $t_k = t.a^k$ is also irrational for all $k \leq 0$, because $a^k$ is a rational function. Now \cref{lem:fix:irrational} implies that every element of $g \in F_n'$ that fixes $t_k$ also fixes a neighborhood of $t_k$. It follows that $\LL_t = \LL_{(t)}$ in this case, and so \cref{prop:isolatedlamps} implies:

\begin{proposition}[\cref{intro:thm:lamplike} part (i)]
\label{prop:irrationallamps}

If $t$ is irrational, then $\LL_t$ is isomorphic to $\Hyp(F_n^c \wr \Z)^{subst}$.\qed
\end{proposition}

\subsection{$n$-ary lamps}

The $n$-ary case is also not that difficult, in fact we will see it is quite analogous to the irrational case. Let $t \in (0, r) \cap \Z[1/n]$. Then $t_k \in \Z[1/n]$ as well for all $k \leq 0$. In this case, elements of $F_n'$ can fix $t_k$ without fixing any neighborhood thereof, so we need to reframe the arguments for \cref{prop:isolatedlamps} taking this into account. This requires only minor changes, so we give the precise statements, but only a brief indication for the proofs.

We start by giving the following definition in general, since it will also be used in the rational case.

\begin{definition}
\label{def:Fnt}

Let $t \in (0, r)$. We define
\[F_n[t_\bullet] \coloneqq \{ g \in F_n \mid g \text{ fixes } t_k \text{ for all } k \leq 0 \} \leq F_n\]
and
\[F_n'[t_\bullet] \coloneqq F_n[t_\bullet] \cap F_n'.\]
\end{definition}

Recall from \cref{lem:LLt:orbit} that $Q \in \LL_t$ (that is, $Q$ fixes $t$) if and only if $Q \subseteq F_n[t_\bullet]$ (that is, $Q$ fixes $t_k$ for all $k \leq 0$).

\medskip

This next result is proved just as in \cref{lem:Fnt:isolated:iso}.

\begin{lemma}
\label{lem:Fnt:iso}

Suppose that $t \in (0, r) \cap \Z[1/n]$. The restrictions $F_n[t_\bullet]|_{[t, 1]}$ and $F_n[t_\bullet]|_{[t_{k}, t_{k+1}]}$ yield an isomorphism
\[\zeta^t \colon F_n[t_\bullet] \to \left( \bigoplus_{k < 0} F_n[t_{k}, t_{k+1}] \right) \times F_n[t, 1]\]
and an injective homomorphism
\[\zeta^t \colon F_n'[t_\bullet] \to \left( \bigoplus_{k < 0} F_n[t_{k}, t_{k+1}] \right) \times F_n[t, 1)\]
that surjects onto each factor. \qed
\end{lemma}

This leads to the following:

\begin{proposition}[\cref{intro:thm:lamplike} part (ii)]
\label{prop:narylamps}

Let $t \in (0, r) \cap \Z[1/n]$. Then $\LL_t$ is isomorphic to $\Hyp(F_n \wr \Z)^{subst}$. 
\end{proposition}

\begin{proof} The proof is exactly analogous to \cref{prop:isolatedlamps}. The only difference is that the map $\eta^t$ is defined via \cref{cit:bieristrebel} and has a different range:
\[\eta^t \colon \left( \bigoplus_{k < 0} F_n[t_{k}, t_{k+1}] \right) \times F_n[t, 1) \to \left( \bigoplus_{k < 0} F_n \right) \times F_n[t, 1),\]
which hence also changes the definition of $\theta^t$.
\end{proof}

\subsection{Rational lamps}

Finally, we deal with the non-$n$-ary rational case, which is the most difficult. Let $t \in (0, r) \cap \Q \setminus \Z[1/n]$. In this case, there are elements of $g \in F_n'$ that fix $t$ but do not fix a neighborhood of $t$, like in the $n$-ary case, but we also have that the slope is not allowed to change at $t$, like in the irrational case. This makes it so that we do not have the splitting behavior that we had in the other cases, and so $\LL_t$ is bound to behave differently than the two cases we have seen so far. The key is to separate the analysis into two steps, according to the short exact sequence
\[1 \to F_n[0, t) \times F_n(t, 0] \to \Fix_{F_n}(t) \to \Z \to 1\]
from \cref{lem:fix:rational}. We will see that the kernel of the map $\Fix_{F_n}(t) \to \Z $ corresponds to the case of isolated lamps that we have already treated. The quotient is studied in the following lemma:

\begin{lemma}
\label{lem:rational:slopes}

The slope homomorphisms of $F_n[t_\bullet]$ at $t_k : k \leq 0$ yield a homomorphism
\[\xi^t \colon F_n[t_\bullet] \to \bigoplus_{k \leq 0} \Z\] 
whose kernel is precisely $F_n(t_\bullet)$, and which admits a homomorphic section
\[\xi_{sec}^t \colon \bigoplus_{k \leq 0} \Z \to F_n'[t_\bullet].\]
Moreover, $\xi^t$ intertwines conjugation by $a$ with the shift:
\[\xi^t(g^a) = \sigma(\xi^t(g)); \qquad \xi^t_{sec}(\sigma(h)) = \xi^t_{sec}(h)^a.\]
\end{lemma}

\begin{proof}
Every element of $F_n[t_\bullet]$ fixes every $t_k : k \leq 0$, and the slope of an element cannot change at any $t_k$, so for each $k\le 0$ the map $F_n[t_\bullet] \to \Z$ given by $g\mapsto \log_n g'(t_k)$ is a well defined homomorphism. This map is defined on all of $\Fix_{F_n}(t_k)$, and \cref{lem:fix:rational} states that the image is non-trivial and generated by some positive integer $o$ independent of $k$. Moreover, a preimage of the generator can be chosen to be in $F_n'$ and supported on an arbitrarily small neighborhood of $t_k$. Thus, by identifying the image $o \Z$ with $\Z$, we obtain a surjective homomorphism $\xi^t_k \colon F_n[t_\bullet] \to \Z$, whose kernel coincides with the subgroup of elements fixing a neighborhood of $t_k$. Putting all of these together defines the map $\xi^t$ from the statement; the image lies in the direct sum by \cref{rem:fixzero}. The kernel is the subgroup of elements that fix pointwise a neighborhood of $t_k$ for all $k \leq 0$, i.e., $F_n(t_\bullet)$.

Now fix an element $g_0 \in F_n'$ such that $\log_n g_0'(t) = o$, and such that the support of $g_0$ is contained in a neighborhood $I$ of $t$ such that the translates $I.a^k$ are pairwise disjoint for all $k\le 0$. We define $g_k \coloneqq g_0^{a^k} \in F_n'$ for all $k < 0$, so $\log_n g_k'(t_k) = o$ and the support of $g_k$ is contained in $I.a^k$. The $g_k$ pairwise commute, and naturally generate a copy of $\oplus_{k \leq 0} \Z$, which defines the section $\xi^t_{sec}$. 

Finally, $\xi^t(g^a) = \sigma(\xi^t(g))$ just follows from the fact that the slope of $g^a$ at $t_{k+1} = t_k.a$ coincides with the slope of $g$ at $t_k$, and $\xi^t_{sec}(\sigma(h)) = \xi^t_{sec}(h)^a$ follows directly from the way we chose the $g_k$. 
\end{proof}

\begin{proposition}[\cref{intro:thm:lamplike} part (iii)]
\label{prop:rationallamps}

Let $t \in (0, r) \cap \Q \setminus \Z[1/n]$. Then there exist order-preserving retractions $\Xi \colon \LL_t \to \Hyp(\Z \wr \Z)^{subst}$ and $\Cap \colon \LL_t \to \LL_{(t)}$.
\end{proposition}

The third part of \cref{intro:thm:lamplike} is a combination of \cref{prop:rationallamps} and \cref{prop:isolatedlamps}. Recall that a retraction of posets is an order-preserving surjection that admits an order-preserving section.

\begin{proof}
First we construct the retraction $\Xi$. Recall that $\Hyp(\Z \wr \Z)^{subst}$ is identified with $\Conf_\sigma^{RH}(\Z \wr \Z)$: right-heavy confining subsets, including the trivial one $[L(\Z)]$. The definition of right-heavy requires that these contain $L_{\geq 0}(\Gamma)$, however in order to make the notations match up, let us define for the rest of this proof a confining subset of $\Z \wr \Z$ as right-heavy if it contains $L_{> 0}(\Gamma)$, and define $\Conf_\sigma^{RH}(\Z \wr \Z)$ accordingly. This of course makes no difference, up to applying one shift.

As in \cref{prop:isolatedlamps}, we define the map $\Xi \colon \LL_t \to \Conf_\sigma^{RH}(\Z \wr \Z)$ at the level of $t$-split representatives, which is allowed by \cref{lem:splitreps:thompson}. For a $t$-split confining subset $Q \in \LL_t$, we define $\Xi(Q) \subseteq L(\Z)$ as
\[\Xi(Q) \coloneqq \{ g \in L(\Z) \mid \pi_{\leq 0}(g) \in \xi^t(Q) \},\]
where $\xi^t$ is given by \cref{lem:rational:slopes}. Equivalently, $\Xi(Q)$ is the largest subset of $L(\Z)$ that satisfies
\[\pi_{\leq 0}(\Xi(Q)) = \xi^t(Q).\]
As in \cref{prop:isolatedlamps}, we notice that $\Xi(Q)$ is right-heavy, split, and the fact that it is a confining subset follows using the same arguments, replacing \cref{intertwine:Q} with the simpler:
\[\sigma(\pi_{\leq 0}(\Xi(Q)) = \xi^t(Q^a),\]
which is ensured by \cref{lem:rational:slopes}. Next we claim that $\Xi$ is order-preserving. Suppose that $Q_1\preceq_\chi Q_2$, so by \cref{cor:conjugate} we can choose $k\ge 0$ such that $Q_2^{a^k}\subseteq Q_1$. Then $\xi^t(Q_2^{a^k})\subseteq \xi^t(Q_1)$, so $\sigma^k(\xi^t(Q_2))\subseteq \xi^t(Q_1)$. This shows that $\sigma^k(\Xi(Q_2))\subseteq \Xi(Q_1)$, and so $\Xi(Q_1)\preceq_\sigma \Xi(Q_2)$.

We are left to prove that it is a retraction. For a right-heavy, split confining subset $R \subseteq L(\Z)$ we define $\Xi_{sec}(R) \subseteq F_n'$ as
\[\Xi_{sec}(R) = \xi_{sec}(\pi_{\leq 0}(R)) F_n'[t, 1),\]
where $\xi_{sec}^t$ is given by \cref{lem:rational:slopes}. Then $\Xi_{sec}(R)$ is right heavy, $t$-split by \cref{lem:splitreps:thompson}, and the fact that it is a confining subset is similarly a consequence of:
\[\Xi_{sec}(\sigma(R)L_{>0}(\Z)) = \xi_{sec}(R)^a F_n'[t, 1).\]
That $\Xi_{sec}$ is a section of $\Xi$ is a consequence of the fact that $\xi^t_{sec}$ is a section of $\xi^t$, from \cref{lem:rational:slopes}.

\medskip

Now we construct the retraction $\Cap$. For $Q \in \LL_t$, we define
\[
\Cap(Q) \coloneqq Q \cap F_n'(t_\bullet).
\]
\cref{prop:join}  shows that $\Cap(Q)$ is a confining subset, and by \cref{def:Fnt:isolated}, $\Cap(Q) \in \LL_{(t)}$. Note moreover that $\Cap$ restricts to the identity on $\LL_{(t)}$, and therefore it is a retraction; it remains to show that it is order-preserving.

Suppose that $Q_1 \preceq_\chi Q_2$, and let $k \geq 0$ be such that $Q_2^{a^k} \subseteq Q_1$, given by \cref{cor:conjugate}. Then
\[(Q_2 \cap F_n'(t_\bullet))^{a^k} \subseteq Q_2^{a^k} \cap F_n'(t_\bullet)^{a^k} \subseteq Q_1 \cap F_n'(t_\bullet)\]
which proves that $\Cap(Q_1) \preceq_\chi \Cap Q_2$.
\end{proof}

\begin{remark}
\label{rem:injectivity:rational}

\cref{prop:rationallamps} gives two retractions that exhibit the largeness of the poset $\LL_t$ and give two approaches to study it. These two retractions seem to be independent, and to describe the confining subsets completely, so one might hope that the product $\Xi \times \Cap \colon \LL_t \to \Hyp(\Z \wr \Z)^{subs} \times \LL_{(t)}$, where the right-hand side is equipped with the product order, is an isomorphism of posets. Unfortunately, this is not true.

Some compelling evidence making it seem like $\Xi\times\Cap$ could be an isomorphism is given by the following fact. Suppose that $Q \in \LL_t$ is such that $\Xi(Q) = L_{\geq 0}(\Z)$. Then every element of $Q$ has slope $1$ at $t_k$ (and so it fixes a neighborhood of $t_k$) for all $k < 0$ small enough. Thus, up to conjugation, $Q$ belongs to $\LL_{(t)}$. This shows that the fiber under $\Xi$ of the largest hyperbolic structure in $\Hyp(\Z \wr \Z)^{subs}$ is precisely $\LL_{(t)}$, which is a retract via $\Cap$.

However, this behavior does not persist if we swap the factors. Namely, if $Q \in \LL_t$ is such that $\Cap(Q) = F_n'(t, 1) \in \LL_{(t)}$, then this is only telling us that for every element $g \in Q$ there exists $k < 0$ such that $g$ does not fix any neighborhood of $t_k$. For example, we can choose two distinct elements $g_1, g_2 \in F_n'$ supported in a small neighborhood of $t_0$ with the same slope at $t_0$, and define $Q_i \coloneqq \langle g_i^{a^k} \mid k \leq 0 \rangle$ for each $i=1,2$. This way $\Cap(Q_1) = \Cap(Q_2)$ and $\Xi(Q_1) = \Xi(Q_2)$, but $Q_1$ and $Q_2$ are not equivalent. Thus in fact we see that $\Xi \times \Cap$ is not even injective.
\end{remark}

\section{Comparing lamplike confining subsets}\label{sec:lamplike:plural}

In this section, we study how the posets $\LL_t$ interact with each other for different $t$. The goal is to prove all the statements of \cref{intro:thm:local} regarding lamplike confining subsets. Part (i) has already been shown in \cref{cor:largest}. In the next three subsections we show parts (ii), (iii) and (iv). We retain all the previous notation for the element $a$, the point $r$, the character $\chi$, and so forth.

\subsection{First properties}

We start by clarifying how to parametrize the $\LL_t$ taking $t$ from a real interval as claimed in part (ii) of \cref{intro:thm:local}:

\begin{lemma}
\label{lem:LLt:equality}

For every $t \in (0, r)$ we have an equality of posets $\LL_t = \LL_{t.a^k}$ for all $k \in \Z$.
\end{lemma}

\begin{proof}
For a $\chi$-confining subset $Q$ we have $[Q] \in \LL_t$ if and only if $[Q^{a^k}] \in \LL_{t.a^k}$. Moreover, $Q$ and $Q^{a^k}$ are equivalent by \cref{cor:conjugate}.
\end{proof}

In particular, the $\LL_t$ are indexed by a real interval, namely, we can use $t\in [t_0,t_0.a)$ for any $t_0\in (0,r)$.

The next important property of $\LL_t$ is that it is upper. Recall that a subset of a poset is \emph{upper} if it is closed under taking upper bounds.

\begin{lemma}
\label{lem:LLt:upper}

The subposet $\LL_t$ is upper, both as a subset of $\X$ and as a subset of $\Hyp(F_n)$.
\end{lemma}

\begin{proof}
By \cref{intro:thm:global}, it suffices to show that $\LL_t$ is upper as a subset of $\X$. We need to show that if $Q_1, Q_2 \in \Conf^\ast_\chi(F_n)$ are such that $[Q_1] \preccurlyeq_\chi [Q_2]$ and $[Q_1] \in \LL_t$, then $[Q_2] \in \LL_t$. By \cref{cor:conjugate} there exists $k \geq 0$ such that $Q_2^{a^k} \subseteq Q_1$. Then $Q_2^{a^k} \subseteq Q_1 \subseteq \Fix_{F_n}(t)$, by definition, which means that $Q_2 \subseteq \Fix_{F_n}(t.a^{-k})$ and thus $[Q_2] \in \LL_{t.a^{-k}} = \LL_t$ by \cref{lem:LLt:equality}.
\end{proof}

Finally, we identify the largest and smallest elements of $\LL_t$:

\begin{lemma}
\label{lem:LLt:largestsmallest}

The poset $\LL_t$ admits a largest element $[F_n'[t, 1)]$, and a smallest element $[F_n'[t_\bullet]]$.
\end{lemma}

\begin{proof}
The equivalence class of $F_n'[t, 1)$ is the largest element in $\X$ by \cref{cor:largest}, so a posteriori also in $\LL_t$. We saw in \cref{def:Fnt} that every element in $\LL_t$ is represented by a confining subset that is contained in $F_n[t_\bullet]$, and so a posteriori in $F_n'[t_\bullet]$; this shows that $[F_n'[t_\bullet]]$ is the smallest element.
\end{proof}

We now compare the smallest elements of $\LL_t$ for two different values of $t$:

\begin{lemma}
\label{lem:LLt:smallest:compare}

The smallest elements of $\LL_t$ and $\LL_s$ are comparable if and only if $t$ and $s$ are in the same $\langle a\rangle$-orbit, in which case these smallest elements coincide.
\end{lemma}

\begin{proof}
Suppose that $F_n'[t_\bullet] \preceq_\chi F_n'[s_\bullet]$ for some $t, s \in (0, r)$. By \cref{cor:conjugate} there exists $k \geq 0$ such that $F_n'[s_\bullet]^{a^k} \subseteq F_n'[t_\bullet]$. But the global fixed points of $F_n'[s_\bullet]^{a^k}$ are precisely $(s.a^i)_{i \leq k}$, and therefore $t = s.a^i$ for some $i \leq k$. This shows that $t$ and $s$ are in the same orbit, and then \cref{lem:LLt:equality} applies. The reverse implication also follows from \cref{lem:LLt:equality}. 
\end{proof}

\subsection{Smallest elements}

In this subsection we prove a more delicate result, which completes part (iii) of \cref{intro:thm:local}. It should be understood as formalizing the idea that each $\LL_t$ is ``spread out'' through the whole poset $\X$.

\begin{proposition}
\label{prop:LLt:minimal}

For each $t\in (0,r)$, the smallest element of $\LL_t$ is minimal in $\X$.
\end{proposition}

\begin{proof}
Let $Q$ be a $\chi$-confining subset such that $Q\preceq_\chi F_n'[t_\bullet]$. We need to show that $Q$ is either equivalent to $F_n'[t_\bullet]$ or to $F_n'$. By \cref{cor:conjugate}, we may assume without loss of generality that $F_n'[t_\bullet] \subseteq Q$. Suppose that $Q$ is not equivalent to $F_n'[t_\bullet]$. We claim that for all $k < 0$ there exists $g_k \in Q$ such that $t.g_k \leq t_k$, where recall that $t_k=t.a^k$. This will imply that
\[
F_n'(t_k, 1) \subseteq F_n'(t.g_k, 1) = F_n'(t, 1)^{g_k} \subseteq Q^3.
\]
Since this is independent of $k$, this shows that $Q$ is equivalent to $F_n'$, as desired.

We are left to prove the claim. We will build our elements $g_k$ in several steps. First, we claim that for each $k \leq 0$ there exists $h_k \in Q$ such that $t_k.h_k \neq t_k$. Indeed, otherwise we would have $[Q] \in \LL_{t_k} = \LL_t$ by \cref{lem:LLt:equality}, and so $F_n'[t_\bullet] \preceq_\chi Q$ by \cref{lem:LLt:largestsmallest}, but this is impossible since we have assumed that $Q$ is not equivalent to $F_n'[t_\bullet]$. Up to replacing $h_k$ by its inverse, we may assume that $t_k.h_k < t_k$. 

Next, we claim that for all $k \leq 0$, there exists $q_k \in Q^3$ such that $t_k.q_k \leq t_{k-1}$. For each $k\le 0$, let $h_k \in Q$ be such that $t_k.h_k < t_k$, given by the previous step. We assume $t_k.h_k \in (t_{k-1}, t_k)$, since otherwise there is nothing to prove. Now choose some $f_k \in F_n'(t_{k-1}, t_k)$ such that $(t_k.h_k).f_k\le t_{k-1}.h_{k-1}^{-1}$, so we have $t_k.(h_k f_k h_{k-1}) \leq t_{k-1}$. The element $q_k \coloneqq h_k f_k h_{k-1} \in Q \cdot F_n'(t_{k-1}, t_k) Q \subseteq Q^3$, is now as desired. 

Finally, fix $k < 0$, and we will construct the desired $g_k$. Let $K \geq 0$ be such that $(Q^{3|k|})^{a^K} \subseteq Q$, which exists by \cref{lem:products}. By the previous step, there exist elements $q_0', q_{-1}', \ldots, q_{k+1}' \in Q^3$ such that $t_{i-K}.q_i' \leq t_{(i-1)-K}$ for each $0\ge i \ge k+1$. Then:
\begin{align*}
t. (q_0' q_{-1}' \cdots q_{k-1}')^{a^K} &= t_{-K}. (q_0' q_{-1}' \cdots q_{k-1}' a^K) \le t_{-1-K}. (q_{-1}' \cdots q_{k-1}' a^K) \\
&\le \cdots \le t_{k+1 - K}.(q_{k-1}' a^K) = t_{k-K}.a^K = t_{k}.
\end{align*}
Since $(q_0' q_{-1}' \cdots q_{k-1}')^{a^K} \in (Q^{3|k|})^{a^K} \subseteq Q$, we conclude by setting $g_k$ equal to $(q_0' q_{-1}' \cdots q_{k-1}')^{a^K}$.
\end{proof}

\begin{remark}
\label{rem:meet}

\cref{prop:LLt:minimal} shows that $\X = \Conf^\ast_{\chi}(F_n)$ is not a meet semilattice, since it has uncountably many incomparable minimal elements.
It is not clear to us whether $\Conf_\chi(F_n)$ is a meet semilattice, see \cref{q:meet}.
\end{remark}

\subsection{Largeness}

We have already seen that each $\LL_t$ contains a copy of $\mathcal{P}(\N)$ for all $t \in (0, r)$ (\cref{cor:P(N):LLt}). In this subsection, we give a more precise largeness statement, which shows that intersecting several ``lamplike'' subposets and complements of lamplike subposets together still yields something large. This is a more precise version of part (iv) of \cref{intro:thm:local}.

\begin{proposition}[\cref{intro:thm:local} part (iv)]
\label{prop:intersections}

Let $t \in (0, r)$. Then, for all $x \in [t.a^{-1}, t)$ there exists a copy of $\mathcal{P}(\N)$ contained in $\displaystyle \bigcap_{x < s \leq t} \LL_s \setminus \bigcup_{t.a^{-1} < s < x} \LL_s$. In particular, there exists a copy of $\mathcal{P}(\N)$ contained in $\displaystyle \bigcap_{t \in (0, r)} \LL_t$.
\end{proposition}

The ``in particular'' statement is just the special case $x=t.a^{-1}$, thanks to \cref{lem:LLt:equality}.

Note that the copy of $\mathcal{P}(\N)$ inside $\LL_t$ that we obtained in \cref{cor:P(N):LLt} via \cref{ex:P(N):subgroups} cannot be contained in $\displaystyle \bigcap_{t \in (0, r)} \LL_t$, so we need a different construction. For the proof, it will be more natural to embed not $\mathcal{P}(\N)$, but another, related poset:

\begin{definition}
For $X, Y \in \mathcal{P}(\N)$, define $X \preceq Y$ if $X \setminus Y$ is finite. This is a preorder, which induces an equivalence relation $\sim$, given by $X \sim Y$ if and only if the symmetric difference $X \Delta Y$ is finite. Denote by $\mathcal{P}(\N)/\mathrm{Fin}$ the corresponding poset. In words, elements of $\mathcal{P}(\N)/\mathrm{Fin}$ are equivalence classes $[X]_{\mathrm{fin}}$ of subsets of $\N$, considered up to finite sets, and $[X]_{\mathrm{fin}} \preccurlyeq [Y]_{\mathrm{fin}}$ if and only if $X \setminus Y$ is finite.
\end{definition}

It is enough to consider this poset because it is still large:

\begin{lemma}
\label{lem:PtildeN}

The poset $\mathcal{P}(\N)/\mathrm{Fin}$ contains a copy of $\mathcal{P}(\N)$.
\end{lemma}

\begin{proof}
A similar proof can be found in the proof of \cite[Proposition~6.1]{ABO}. Enumerate the prime numbers $ p_1 < p_2 < \cdots$. We define a map $\mathbb{P} \colon \mathcal{P}(\N) \to \mathcal{P}(\N)$ by $\mathbb{P}(X) = \{ p_i^k \mid i \in X, k \geq 1 \}$. In other words, $\mathbb{P}(X)$ is made up of all powers of primes whose index matches an element of $X$. We consider the composition of $\mathbb{P}$ with the quotient map $\mathcal{P}(\N) \to \mathcal{P}(\N)/\mathrm{Fin}$, and claim that this is an order embedding.

We need to show that $X \subseteq Y$ if and only if $\mathbb{P}(X) \setminus \mathbb{P}(Y)$ is finite. The forward direction is true in a stronger sense: if $X \subseteq Y$ then $\mathbb{P}(X) \subseteq \mathbb{P}(Y)$. Conversely, suppose that $X$ is not contained in $Y$, so let $i \in X \setminus Y$. Then $\mathbb{P}(X)$ contains all powers of $p_i$, while $\mathbb{P}(Y)$ does not contain any of them. Therefore $\mathbb{P}(X) \setminus \mathbb{P}(Y)$ is infinite.
\end{proof}

Note that $\mathcal{P}(\N)$ with $\subseteq$ is poset isomorphic to $\mathcal{P}(\N)$ with $\supseteq$, via the map $X\mapsto \N\setminus X$, and the analogous statement is also true for $\mathcal{P}(\N)/\mathrm{Fin}$, so when embedding these posets we may use either order.

We can now proceed with our construction:

\begin{proof}[Proof of \cref{prop:intersections}]
Let $x\in [t.a^{-1},t)$. We will work inside the group $F_n'(t_\bullet)$ of elements of $F_n'$ fixing neighborhoods of each $t_k$ for $k\le 0$ (\cref{def:Fnt:isolated}). Using \cref{lem:PtildeN}, it suffices to show that there exists a copy of $\mathcal{P}(\N)/\mathrm{Fin}$, with the reverse ordering, contained in $\displaystyle \bigcap_{x < s \leq t} \LL_s \setminus \bigcup_{t.a^{-1}<s<x}\LL_s$. Choose a strictly increasing sequence $p_i \in (x, t)$ for $i \in \N$ such that $p_i \to t$ as $i \to \infty$. Let $P_{0,1} \coloneqq F_n'(x,p_1)$ and $P_{0,i} \coloneqq F_n'(p_{i-1},p_i)$ for all $i>1$, so the $P_{0,i}$ are all supported on $(x,t)$. Since the sequence $p_i$ is strictly increasing, the $P_{0,i}$ are disjointly supported and hence commute. For a subset $X \in \mathcal{P}(\N)$, denote by $P_{0,X} \coloneqq \langle P_{0,i} \mid i\in X \rangle \cong  \oplus_{i \in X} P_{0,i}$. For all $k\le 0$, set $P_{k,i} \coloneqq (P_{0,i})^{a^k} \le F_n(x_k,t_k)$ and
\[
P_{k,X} \coloneqq \langle P_{k,i} \mid i \in X\text{, } i > |k| \rangle,
\]
where $x_k$ denotes the translate of $x$ under $a^k$. We define
\[
Q_X \coloneqq \{ g \in F_n'(t_\bullet) \mid g|_{(x_k,t_k)} \in P_{k,X} \text{ for all } k \le 0 \}.
\]
Now we claim that $Q_X$ is a $t$-split confining subgroup in $\LL_t$. Indeed, \StayInConf\ follows from the fact that $P_{k-1,X} \subseteq P_{k,X}$, \GetInConf\ follows from the fact that $Q_X$ contains $F_n'[t, 1)$, and \ProdConf\ follows from the fact that $Q_X$ is a subgroup. Moreover, $Q_X$ is in $\LL_t$ by construction and is clearly $t$-split. Now we claim that moreover $Q_X$ is in $\LL_s$ for all $s\in (x,t)$.
Indeed, let $s\in (x,t)$, and choose $k \leq 0$ such that $p_{|k|-1} \leq s < p_{|k|}$. Then $P_{0, i}$ fixes $s$ for all $i > |k|$ and so $P_{k, X}$ fixes $s.a^k$. It then follows from the definition of $Q_X$ that every element of $Q_X$ fixes $s.a^k$, and so $[Q_X] \in \LL_{s.a^k} = \LL_s$ by \cref{lem:LLt:equality}. 

Next we show that $X \to Q_X$ induces an order embedding of $\mathcal{P}(\N)/\mathrm{Fin}$ with the reverse ordering, into $\Conf^\ast_\chi(F_n)$ with the usual ordering. We need to show that for two subsets $X, Y \in \mathcal{P}(\N)$, the difference $X \setminus Y$ is finite if and only if $Q_Y \preceq_\chi Q_X$. The latter is equivalent to the existence of $j \geq 0$ such that $Q_X^{a^j} \subseteq Q_Y$, by \cref{cor:conjugate}. By construction, this containment holds if and only if $P_{k-j,X} \leq P_{k,Y}$, i.e., $\{ i \in X \mid i > |k-j| \} \subseteq \{ i \in Y \mid i > |k| \}$, for all $k\le 0$. Thus we see that $Q_Y \preceq_\chi Q_X$ if and only if there exists $j\ge 0$ such that for all $k\le 0$ we have $X \setminus \{1, \ldots, |k-j| \} \subseteq Y \setminus \{1, \ldots, |k| \}$. Now we can prove the desired equivalence. First suppose $X\setminus Y$ is finite, and set $j$ equal to the largest element of $X\setminus Y$. Then for any $k\le 0$ we have $|k-j|\ge j$ and $|k-j|\ge |k|$, and so if $x\in X \setminus \{1, \ldots, |k-j| \}$ then $x\in Y\setminus \{1, \ldots, |k| \}$. Conversely, if there exists $j\ge 0$ such that for all $k\le 0$ we have $X \setminus \{1, \ldots, |k-j| \} \subseteq Y \setminus \{1, \ldots, |k| \}$, then the $k=0$ case reveals that $X\setminus Y$ is contained in $\{1,\dots,j\}$ and hence is finite.

To complete the proof, it remains to prove that $Q_X$ is not equivalent to any confining subset in $\LL_s$ for $s\in (t.a^{-1},x)$. Since the definition of $Q_X$ did not involve any conditions regarding the restriction of elements to any $(t_{k-1},x_k)$, it is easy to see that no $s.a^k$, for any $k\le 0$, is a global fixed point of $Q_X$. Thus, $Q_X$ cannot be equivalent to any confining subset in any such $\LL_s$.
\end{proof}

\section{Non-lamplike confining subsets}\label{sec:moving}

In \cref{sec:lamplike} and \cref{sec:lamplike:plural}, we gathered a thorough understanding of the posets $\LL_t$ for $t \in (0, r)$. Moreover, in \cref{lem:LLt:upper} we saw that there are no confining subsets outside $\LL_t$ that dominate all of $\LL_t$, and in \cref{prop:LLt:minimal} we saw that there are no confining subsets outside $\LL_t$ that are dominated by all of $\LL_t$. This, together with the fact that all strictly confining \emph{subgroups} are lamplike (\cref{prop:contains:subgroup}), all seems to suggests that lamplike confining subsets tell the whole story.

However, in this section, we prove that this intuition is wrong. Namely we prove \cref{intro:thm:non_lamp} that there is a copy of $\mathcal{P}(\N)$ inside $\X$ that consists entirely of non-lamplike confining subsets. Note that none of these can be equivalent to a confining subgroup, by \cref{prop:contains:subgroup}. We limit ourselves to one specific construction, so we will make a number of specific choices whose aim is to facilitate the proof of \cref{intro:thm:non_lamp}. Our goal is not to give general statements, which seem difficult to come by in the setting of non-lamplike structures.

\medskip

We start by constructing a special sequence $\tau \subseteq (0, r)$, which is discrete and accumulates at $0$. Let $\tau_1$ be an arbitrary element of $(0, r)$, and set $\tau_{2^i} = \tau_1.a^{-i}$ for all $i \geq 0$. Next, choose $\tau_3$ to be an arbitrary element in $(\tau_4, \tau_2)$ and set $\tau_{2^i 3} = \tau_3.a^{-i}$. Inductively, at step $k$, we have defined $\tau_j$ for all $j$ of the form a power of $2$ times an odd number less than $(2k+1)$. We then choose $\tau_{2k+1}$ to be an arbitrary element in $(\tau_{2k+2}, \tau_{2k})$, and set $\tau_{2^i (2k+1)} = \tau_{2k+1}.a^{-i}$.
In the end we have a sequence
\[0 < \cdots < \tau_j < \tau_{j-1} < \cdots < \tau_2 < \tau_1 <  r,\]
with the property that $\tau_j.a^{-i} = \tau_{2^i j}$ for all $i, j \geq 1$.

Our non-lamplike confining subsets will be built by imposing restrictions on how far certain elements from the sequence $\tau$ are allowed to move. In order to obtain an actual confining subset, we will need to impose some technical conditions on the sets involved.

\begin{definition}
\label{def:Stilde}
    Let $S \subseteq \N$ be a subset of odd numbers . We define $\widetilde{S}$ to be the smallest subset of $\N$ that contains $S$ and such that:
    \begin{enumerate}
        \item If $s \in \widetilde{S}$ then $2s \in \widetilde{S}$;
        \item If $s \in \widetilde{S}$ then $2^5 s + 1 \in \widetilde{S}$ and $2^5 s - 1 \in \widetilde{S}$.
    \end{enumerate}
\end{definition}

It will become apparent in the course of this section that $2^5$ could just as well be replaced with any larger power of $2$.

\begin{definition}
\label{def:nonlamplike}
    Let $S \subseteq \N$ be a subset of odd numbers. Define
    \[Q_S \coloneqq \{ g \in F_n' \mid \tau_s.g \in (\tau_{s+1}, \tau_{s-1})\text{, } \tau_s.g^{-1} \in (\tau_{s+1}, \tau_{s-1}) \text{ for all } s \in \widetilde{S} \}.\]
\end{definition}

Note that, unlike our previous constructions of confining subsets, here we need to impose conditions on both $g$ and $g^{-1}$ to ensure that $Q_S$ is symmetric.

\begin{proposition}
\label{prop:nonlamplike}
    For every non-empty subset $S \subseteq \N$ of odd numbers, $Q_S$ is a $\chi$-confining subset, which does not belong to $\LL_t$ for any $t \in (0, r)$.
\end{proposition}

\begin{proof}
    The set $Q_S$ is symmetric by definition. Because $Q_S$ is only defined in terms of conditions on the sequence $\tau \subseteq (0, r)$, and $a_i$ is supported on $(r, 1)$ for $i = 1, \ldots, (n-1)$, we see that $Q_S^z = Q_S$ for all $z \in \langle a_1, \ldots, a_{n-1} \rangle$. To prove \StayInConf, it therefore suffices to check that $Q_S^a \subseteq Q_S$. For $s \in \widetilde{S}$, using that $2s \in \widetilde{S}$ we have
    \[\tau_s.a^{-1}g = \tau_{2s}.g \in (\tau_{2s+1}, \tau_{2s-1}) \subseteq (\tau_{2(s+1)}, \tau_{2(s-1)}) = (\tau_{s+1}, \tau_{s-1}).a^{-1};\]
    and so $\tau_s.g^a \in (\tau_{s+1}, \tau_{s-1})$. Similarly, $\tau_s.(g^{-1})^a \in (\tau_{s+1}, \tau_{s-1})$.

    \GetInConf follows immediately since $F_n'(\tau_1, 1) \subseteq Q_S$. Finally, for \ProdConf, we will show that $(Q_S.Q_S)^{a^{5}} \subseteq Q_S$. Let $g, h \in Q_S$. For $s \in \widetilde{S}$, using that $2^5s, 2^5s-1$ and $2^5s+1$ belong to $\widetilde{S}$, we have
    \begin{align*}
        \tau_s.a^{-{5}}gh &= \tau_{2^5s}.gh \in (\tau_{2^5s+1}, \tau_{2^5s-1}).h \subseteq (\tau_{2^5s+2}, \tau_{2^5s-2}) \\
        &\subseteq (\tau_{2^5(s+1)}, \tau_{2^5(s-1)}) = (\tau_{s+1}, \tau_{s-1}).a^{-{5}};
    \end{align*}
    and so $\tau_s.(gh)^{a^{5}} \in (\tau_{s+1}, \tau_{s-1})$. Similarly, $\tau_s.((gh)^{-1})^{a^{5}} \in (\tau_{s+1}, \tau_{s-1})$.

    Finally, let $t \in (0, r)$. Because $\tau$ is discrete, there exists a neighborhood $(x, y)$ of $t$ that intersects at most one element of $\tau$; more precisely we can ensure that $(x, y) \cap \tau$ is empty, unless $t \in \tau$ in which case we can ensure that $(x, y) \cap \tau = \{ t \}$. In either case, $(x, y)$ will not contain two consecutive elements of $\tau$, and so any element supported on $(x, y)$ will satisfy the definition of $Q_S$. This shows that $Q_S$ contains $F_n'(x, y)$, and so $t$ cannot be a fixed point of every element of $Q_S$. Because $t$ was arbitrary, we conclude that $Q_S$ is not lamplike, by definition.
\end{proof}

Our next goal is to show how to compare $Q_S$ and $Q_R$ for different sets $S$ and $R$ of well-chosen odd numbers. This will require a couple of technical lemmas.

\begin{lemma}
\label{lem:Stilde:estimate}
    Every element of $\widetilde{S}$ is of the form $2^p s + \varepsilon$ for some $p \geq 0$, where $s \in S$ and $|\varepsilon| < 2^{p-4}$.
\end{lemma}

\begin{proof}
    Applying the two rules from \cref{def:Stilde} iteratively, we can describe $\widetilde{S}$ explicitly as the set of all natural numbers that can be written as
    \[2^p s + \sum\limits_{i = 1}^n e_i 2^{k_i}  \quad \text{where} \quad p - k_1 \geq 5\text{, } k_i - k_{i+1} \geq 5\text{, } p, k_i \geq 0\text{, } e_i \in \{ \pm 1 \}\text{, and } s \in S.\]
    We estimate
    \[\left| \sum\limits_{i = 1}^n e_i 2^{k_i} \right| \leq \sum\limits_{j = 0}^{p-5} 2^j = 2^{p-4} - 1 < 2^{p-4}.\qedhere\]
\end{proof}

Now we choose a suitable infinite set of odd numbers $O$, that is very sparse. We will then show that, for distinct subsets $S, R \subseteq O$, the corresponding sets $\widetilde{S}, \widetilde{R}$ are easy to distinguish.

\begin{lemma}
\label{lem:goododdnumbers}
    There exists an infinite set $O$ consisting of odd numbers, with the following property. If $x, y \in O$ satisfy $|x - 2^p y| \leq 2^{p-3}$ for some $p \in \Z$, then $p = 0$ and $x = y$.
\end{lemma}

For the proof, let us recall that the \emph{density} of a set $A \subseteq \N$ is defined as 
\[
d(A) \coloneqq \limsup\limits_{n \to \infty} \frac{1}{n} \bigg|A \cap \{1, \ldots, n \}\bigg|.
\]
Note that $d$ is subadditive, and that finite sets have density $0$.

\begin{proof}
    Let us start with some observations. First, for all $x \in \N$ and all $p \in \Z$ we have $|x - 2^px| = x|2^p - 1| \geq |2^p - 1|$. When $p \neq 0$, this is always greater than $2^{p-3}$. Thus it suffices to consider the case of $x \neq y$. Secondly, notice that if $p \leq 0$, then $2^{|p|}|x - 2^p y| = |2^{|p|}x - y|$. Because $x, y$ are distinct odd integers, $|2^{|p|}x - y| \geq 1 > 2^{-3}$. Thus it suffices to consider the case of $p > 0$. Finally, if $x < y$, then $|x - 2^py| = 2^py - x > (2^p - 1)y \geq (2^p-1)$, which again is greater than $2^{p-3}$ since $p \neq 0$. Combining these, it suffices to find an infinite set of odd numbers $O$ such that for all $x > y \in O$ and all $p > 0$ we have $|x - 2^py| > 2^{p-3}$.
    
    We construct the set $O=\{x_1,x_2,\dots\}$ by induction. We start by choosing $x_1 \coloneqq 3$. Now suppose that $x_1, \ldots, x_k$ have been chosen so that for $i < j \in \{ 1, \ldots, k \}$
    \begin{itemize}
        \item $x_i$ and $x_j$ are odd, $x_i > x_j$;
        \item $|x_j - 2^p x_i| > 2^{p-3}$ for all $p > 0$.
    \end{itemize}
    We need to choose $x = x_{k+1}$ to be an odd number such that for all $i \in \{1, \ldots, k\}$ and all $p > 0$, the following hold:
    \begin{itemize}
        \item $x$ is odd, $x > x_i$, and $x > 2^{k+1}$;
        \item $x$ does not belong to the set $A_i^p \coloneqq \{ x \in \N \mid |x - 2^p x_i| \leq 2^{p-3} \}$.
    \end{itemize}
    
    The set of elements satisfying the first condition has density $1/2$, as does its complement. Thus it suffices to show that the union of the $A_i^p$, for $i \in \{1, \ldots, k\}$ and $p > 0$, has density strictly smaller than $1/2$. It will then follow that there exist (infinitely many) elements satisfying both conditions, which will conclude the induction step and thus the construction of the set $O$. We estimate:
    \begin{align*}
        d\left(\bigcup_{p > 0} A_i^p\right) &= d\left( \bigcup_{p > 0} \N \cap [2^p x_i - 2^{p-3}, 2^p x_i + 2^{p-3}] \right) \\
        &\leq \limsup\limits_{n \to \infty} \frac{1}{2^n x_i + 2^{n-3}} \sum\limits_{p = 1}^n (2 \cdot 2^{p-3} + 1) \\
        &\leq \limsup\limits_{n \to \infty} \frac{1}{2^n x_i + 2^{n-3}} \left( n + 2^{-2} \sum\limits_{p = 1}^n 2^p \right) \\
        &= \limsup\limits_{n \to \infty} \frac{n + 2^{-2}(2^{n+1} - 2)}{2^n x_i + 2^{n-3}} \\
        &= \limsup\limits_{n \to \infty} \frac{2^{n-1} + n + 2^{-1}}{2^n x_i + 2^{n-3}} \leq \frac{1}{2x_i} < \frac{1}{2^{i+1}}.
    \end{align*}
We conclude:
    \[d\left(\bigcup\limits_{i = 1}^k \bigcup\limits_{p > 0} A_i^p \right) < \sum\limits_{i = 1}^k 2^{-i - 1} < \frac{1}{2} \sum\limits_{i \geq 1} 2^{-i} = \frac{1}{2}.\qedhere\]
\end{proof}

Now we are ready to show how to compare different confining subsets of the form $Q_S$, for subsets $S \subseteq O$.

\begin{proposition}
\label{prop:nonlamplike:compare}
    Let $O$ be as in \cref{lem:goododdnumbers} and let $S, R \subseteq O$. Then $Q_S \preceq_\chi Q_R$ if and only if $S \subseteq R$.
\end{proposition}

\begin{proof}
    If $S \subseteq R$, then $\widetilde{S} \subseteq \widetilde{R}$, so $Q_R \subseteq Q_S$ and thus $Q_S \preceq_\chi Q_R$. Conversely, suppose that $Q_S \preceq_\chi Q_R$. By \cref{cor:conjugate}, there exists $k \geq 1$ such that $Q_R^{a^k} \subseteq Q_S$. Thus for every $s \in \widetilde{S}$, and every $g \in Q_R$, we have $\tau_s.g^{a^k} \in (\tau_{s-1}, \tau_{s+1})$. This is equivalent to:
    \[\tau_{2^ks}.g \in (\tau_{2^k(s+1)}, \tau_{2^k(s-1)}).\]
    Because elements of $Q_R$ are only restricted in terms of $\widetilde{R}$, it follows that $\widetilde{R}$ must intersect the interval $(2^k(s-1), 2^k(s+1))$ in $\N$, and this must hold for every $s \in \widetilde{S}$. Choosing $s = 2^{\ell-k} t$, for $\ell > k$ and $t \in S$, we see that the interval $(2^\ell t - 2^k, 2^\ell t + 2^k)$ contains an element $r_\ell \in \widetilde{R}$, for all $\ell > k$. By \cref{lem:Stilde:estimate}, we can write $r_\ell = 2^{p_\ell} q_\ell + \varepsilon_\ell$, where $|\varepsilon_\ell| < 2^{p_\ell - 4}$ and $q_\ell \in R$. Then
    \[|2^{p_\ell} q_\ell - 2^\ell t| \leq |\varepsilon_\ell| + 2^k < 2^{p_\ell - 4} + 2^k.\]
    Suppose first that the sequence $(p_\ell)_{\ell>k}$ is uniformly bounded. Up to taking a subsequence we may assume that $p_\ell = p$ is constant. Then
    \[|2^p q_\ell - 2^\ell t| < 2^{p - 4} + 2^k \quad\Rightarrow\quad |q_\ell - 2^{\ell - p} t| < 2^{-4} + 2^{k-p}.\]
    Choose $\ell$ large enough that the right hand side of the last inequality is bounded above by $2^{\ell - p - 3}$, and then by \cref{lem:goododdnumbers} we deduce that $t = q_\ell \in R$. Suppose otherwise that $p_\ell$ is unbounded. Then
    \[|2^\ell t - 2^{p_\ell} q_\ell| < 2^{p_\ell - 4} + 2^k \quad\Rightarrow\quad |t - 2^{p_\ell - \ell} q_\ell| < 2^{p_\ell - \ell - 4} + 2^{k-\ell}.\]
    For $p_\ell \geq k + 4$, the right hand side of the last inequality is bounded above by $2^{p_\ell - \ell - 3}$, and we conclude again by \cref{lem:goododdnumbers} that $t = q_\ell \in R$.
    
    In both cases we deduced that $t \in R$, and since $t$ was an arbitrary element of $S$, we conclude that $S \subseteq R$.
\end{proof}

\begin{theorem}[\cref{intro:thm:non_lamp}]\label{thm:non_lamp}
There exists a copy of $\mathcal{P}(\N)$ in $\X$ consisting entirely of non-lamplike confining subsets.
\end{theorem}

\begin{proof}
    Let $O$ be the subset of odd numbers from \cref{lem:goododdnumbers}. By \cref{prop:nonlamplike} and \cref{prop:nonlamplike:compare}, for every non-empty set $S \subseteq O$ there is a non-lamplike confining subset $Q_S \in \X$ such that $S \subseteq R$ if and only if $Q_S \preceq_\chi Q_R$. Fix an element $x \in O$ and let $\mathcal{P}'(O)$ denote the set of subsets of $O$ that contain $x$ (so that the empty set does not belong to $\mathcal{P}'(O))$. Then $S \to Q_S$ induces an embedding $\mathcal{P}'(O) \to \X$ whose image consists entirely of non-lamplike confining subsets, and of course $\mathcal{P}'(O) \cong \mathcal{P}(\N)$.
\end{proof}

\section{Questions}\label{sec:questions}

We conclude by collecting some natural questions that arise. We start with a general question on the structure of confining subsets:

\begin{question}
\label{q:meet}
    We saw in \cref{prop:join} that $\Conf^\ast_\rho(G)$, and thus also $\Conf_\rho(G)$, is a join semilattice, for a general $G = H \rtimes \Z^n$ and $\rho \colon \Z^n \to \R$. We also saw in \cref{rem:meet} that $\Conf^\ast_{\chi_0}(F_n)$ is not a meet semilattice. This is due to the existence of uncountably many minimal elements; however, these all have a meet once we move into $\Conf_{\chi_0}(F_n)$, namely the smallest element. Thus we ask, is $\Conf_\rho(G)$ a meet semilattice in general, and therefore a lattice?
\end{question}

A few questions about $\Hyp(F_n)$ remain open.

\begin{question}
Following \cref{rem:injectivity:rational}, is it possible to give a more explicit description of the poset $\LL_t$ for rational, non-$n$-ary $t$? The retractions $\Xi$ and $\Cap$ do not piece together directly into an injection, but a more nuanced understanding of their interaction could lead to a more complete picture.
\end{question}

\begin{question}
By \cref{intro:thm:non_lamp}, the poset of non-lamplike confining subsets of $F_n$ contains a copy of $\mathcal{P}(\N)$, so in particular is uncountable. Can we say more about its structure? For example, does it decompose similarly to the lamplike case into a family of subposets indexed by some interval of real numbers? Do there exist non-lamplike confining subsets that are substantially different than those in \cref{prop:nonlamplike}?
\end{question}

\begin{question}
We have seen in Propositions \ref{prop:subgroups_to_trees} and \ref{prop:trees_to_subgroups} that quasi-parabolic structures representable by actions on trees correspond to discrete characters and confining subgroups (see also \cref{rem:simplicial}). It follows that many of the hyperbolic structures constructed in this paper are not representable by actions on trees. These include some lamplike (\cref{rem:lamplike:subgroup}) and all non-lamplike structures (\cref{intro:thm:trees}). What spaces arise for those? In particular, are they all representable by actions on quasi-trees? Are any of them representable by actions on (a space quasi-isometric to) the hyperbolic plane $\mathbb{H}^2$?
\end{question}

Finally, there are several close relatives of the groups $F_n$ whose hyperbolic actions one could study. Let us list three examples here, and note that the corresponding $T$-like groups are all known to have trivial poset of hyperbolic structures \cite{NL}. 

\begin{example}[Irrational slope]\label{ex:irrat}
Let $\tau=\frac{\sqrt{5}-1}{2}$ be the small golden ratio. The \emph{Cleary group} $F_\tau$, introduced by Cleary in \cite{cleary00}, is the group of all piecewise linear orientation preserving homeomorphisms of $[0,1]$ with all slopes powers of $\tau$ and breakpoints in $\Z[\tau]$. This group has been studied from several viewpoints, for instance its metric properties \cite{burillo21}, BNSR-invariants \cite{MNSR}, non-embeddability into $F$ \cite{hydemoore}, and values of stable commutator length and simplicial volume \cite{commutingconjugates}. With intriguing similarities and differences with $F$, it would be interesting to understand its poset of hyperbolic actions. One could also inspect other ``irrational slope'' Thompson-like groups among the Bieri--Strebel groups from \cite{bieristrebel}. Note that not all such groups admit tree pair diagrams \cite{winstone}, however in this paper we almost exclusively used the dynamical point of view, so this may not be a substantial obstruction.
\end{example}

\begin{example}[Stein group]
The \emph{Stein group} $F_{2,3}$ is the group of all piecewise linear orientation preserving homeomorphisms of $[0,1]$ with all slopes of the form $2^a 3^b$ for $a,b\in\Z$ and breakpoints in $\Z[\frac{1}{6}]$. This group, along with the generalizations $F_{n_1,\ldots,n_k}$ defined in the analogous way, first appeared in work of Stein \cite{stein92}, and is surprisingly different from $F_n$ in several ways. For example, unlike $F_n$, every discrete character of $F_{2,3}$ lies in $\Sigma^\infty(F_{2,3})$ \cite{spahn21}, but there exist non-discrete characters not lying in $\Sigma^1(F_{2,3})$ \cite{bieri87}. Also, $F_{2,3}$ does not embed into any $F_n$ \cite{lodha20}. Understanding the poset of hyperbolic actions of $F_{2,3}$ seems more daunting than for $F_n$, partly thanks to the non-discreteness of the characters in the complement of $\Sigma^1(F_{2,3})$. Another complication is that the abelianization map $F_{2,3}\to\Z^4$ does not split, although certain maps onto $\Z^2$ do split.
\end{example}

\begin{example}[Lodha--Moore]
The \emph{Lodha--Moore group}, introduced by Lodha and Moore in \cite{lodha16}, was the first known example of a non-amenable group of type $\F_\infty$ with no non-abelian free subgroups \cite{lodha16,lodhaFinfty}. Without getting into too much detail, the group can be viewed as the result of taking Thompson's group $F$, pictured as a group of homeomorphisms of the Cantor set $\{0,1\}^\N$, and introducing a new element $\lambda$, deferred to each cone of the Cantor set, where $\lambda$ is recursively defined via
\[
\lambda \colon \left\{\begin{array}{ll}
    00\cdot\omega &\mapsto 0\cdot \lambda(\omega)\\
    01\cdot\omega &\mapsto 10\cdot \lambda^{-1}(\omega)\\
    1\cdot\omega &\mapsto 11\cdot \lambda(\omega)\text{.}
\end{array}\right.
\]
(Technically this yields the ``big'' Lodha--Moore group; in the ``small'' Lodha--Moore group, we only defer $\lambda$ to cones not containing $\overline{0}$ or $\overline{1}$.) The BNSR-invariants of the Lodha--Moore group are known \cite{lodha23}, along with some decompositions into ascending HNN-extensions \cite{zaremsky16}, and it would be interesting to try to understand its poset of hyperbolic actions.
\end{example}

The next two examples have infinite locally finite normal subgroups, which makes them quite different from the Thompson-like groups we have considered so far. Nevertheless, they have been studied with Thompson-like techniques in the past, and their hyperbolic structures might be approachable in a similar way. 

\begin{example}[Houghton groups]
The $n$th \emph{Houghton group} $H_n$, introduced by Houghton in \cite{houghton78}, is the group of permutations $\sigma$ of $\N\times\{1,\dots,n\}$ that are eventually translations, meaning there exist $m_1,\dots,m_n\in\Z$ such that for all $i\in\{1,\dots,n\}$ and all sufficiently large $x\in\N$, $\sigma$ sends $(x,i)$ to $(x+m_i,i)$. (An equivalent, somewhat more natural-seeming description is as the group of all permutations $\sigma$ of $\N$ such that for large enough $x\in\N$, we have $\sigma(x+n)=\sigma(x)+n$ and $\sigma(x)$ is congruent to $x$ mod $n$.) Sending $\sigma$ to the tuple $(m_1,\dots,m_n)$ defines a map to $\Z^n$, whose image is isomorphic to $\Z^{n-1}$ (thanks to the relation that $m_1+\cdots+m_n=0$ since $\sigma$ is a bijection). This abelianization map $H_n\to\Z^{n-1}$ does not split, but there is a split surjection $H_n \to \Z^{\lfloor n/2 \rfloor}$. The BNSR-invariants of $H_n$ are known \cite{zaremsky17F_n,zaremsky20}, and have some resemblance to those of $F_n$. The commutator subgroup of $H_n$ is the group of all finitely supported permutations of $\N\times\{1,\dots,n\}$, so $H_n$ is (locally finite)-by-abelian, hence amenable. All of this is to say, the general setup from \cref{sec:confining} could likely be used to analyze part of the poset of hyperbolic actions of $H_n$. The fact that the abelianization map does not split, however, means that we might not necessarily get a full picture of the poset just using these tools. It would be interesting to see whether a full computation of $\Hyp(H_n)$ is possible; in particular, it is not clear to us whether $H_n$ admits any quasi-parabolic actions.
\end{example}

Since the first version of this paper appeared, the posets of hyperbolic structures of the Houghton groups have been classified \cite{houghton_hyp_structures}.

\begin{example}[Quasi-automorphisms]
Let $\mathcal{T}_n$ be the infinite rooted $n$-ary tree, so the vertex set of $\mathcal{T}$ is the set $\{1,\dots,n\}^*$ of finite $n$-ary words. Order the children of a given vertex $w$ by $w\cdot 1<w\cdot 2<\cdots<w\cdot n$. Let $QV_n$ be the group of bijections of the vertex set of $\mathcal{T}_n$ that preserve adjacency and order with only finitely many exceptions, as in \cite{lehnert08}. We get a quotient map $QV_n \to V_n$ onto the Higman--Thompson group $V_n$, with kernel the group of all finitely supported permutations of the vertex set. In this way, $QV_n$ can be viewed as a sort of Houghton-like version of $V_n$. Let $QF_n$ be the preimage under this quotient map of $F_n$, so $QF_n$ is a sort of Houghton-like version of $F_n$. The BNSR-invariants of $QF=QF_2$ are known \cite{nucinkis18} and resemble those of $F$, and presumably those for $QF_n$ behave similarly. It would be interesting to study the poset of hyperbolic actions of $QF_n$, and in particular to see whether every cobounded hyperbolic action comes from the quotient $QF_n\to F_n$. Note that, on the other hand, every cobounded hyperbolic action of $QV_n$ is elliptic \cite[Theorem 1.6.7]{NL}.
\end{example}

\bibliographystyle{amsalpha}
\bibliography{F_hyp_actions}

\providecommand{\bysame}{\leavevmode\hbox to3em{\hrulefill}\thinspace}
\providecommand{\MR}{\relax\ifhmode\unskip\space\fi MR }
\providecommand{\MRhref}[2]{%
  \href{http://www.ams.org/mathscinet-getitem?mr=#1}{#2}
}
\providecommand{\href}[2]{#2}
\begin{thebibliography}{GMSW01}

\bibitem[ABO19]{ABO}
C.~R. Abbott, S.~H. Balasubramanya, and D.~Osin, \emph{Hyperbolic structures on
  groups}, Algebr. Geom. Topol. \textbf{19} (2019), no.~4, 1747--1835.
  \MR{3995018}

\bibitem[ABR23]{confining}
C.~Abbott, S.~Balasubramanya, and A.~J. Rasmussen, \emph{Higher rank confining
  subsets and hyperbolic actions of solvable groups}, Adv. Math. \textbf{424}
  (2023), Paper No. 109045, 47. \MR{4585997}

\bibitem[ABR24]{ABR2}
C.~R. Abbott, S.~Balasubramanya, and A.~J. Rasmussen, \emph{Valuations,
  completions, and hyperbolic actions of metabelian groups}, J. Lond. Math.
  Soc. (2) \textbf{109} (2024), no.~6, Paper No. e12916, 72, With an appendix
  by Abbot, Balasubramanya, Rasmussen and S. Payne. \MR{4751864}

\bibitem[AR23]{AR:BS}
C.~R. Abbott and A.~J. Rasmussen, \emph{Actions of solvable
  {B}aumslag-{S}olitar groups on hyperbolic metric spaces}, Algebr. Geom.
  Topol. \textbf{23} (2023), no.~4, 1641--1692. \MR{4602410}

\bibitem[Bal20]{sahana:wreath}
S.~H. Balasubramanya, \emph{Hyperbolic structures on wreath products}, J. Group
  Theory \textbf{23} (2020), no.~2, 357--383. \MR{4069979}

\bibitem[Bav91]{bavard91}
C.~Bavard, \emph{Longueur stable des commutateurs}, Enseign. Math. (2)
  \textbf{37} (1991), no.~1-2, 109--150. \MR{1115747}

\bibitem[BBCS08]{brady08}
T.~Brady, J.~Burillo, S.~Cleary, and M.~Stein, \emph{Pure braid subgroups of
  braided {T}hompson's groups}, Publ. Mat. \textbf{52} (2008), no.~1, 57--89.
  \MR{2384840}

\bibitem[BCS01]{burillo01}
J.~Burillo, S.~Cleary, and M.~I. Stein, \emph{Metrics and embeddings of
  generalizations of {T}hompson's group {$F$}}, Trans. Amer. Math. Soc.
  \textbf{353} (2001), no.~4, 1677--1689. \MR{1806724}

\bibitem[Bel04]{belk04}
J.~M. Belk, \emph{Thompson's group {F}}, ProQuest LLC, Ann Arbor, MI, 2004,
  Thesis (Ph.D.)--Cornell University. \MR{2706280}

\bibitem[BFFG]{NL}
S.~H. Balasubramanya, F.~Fournier-Facio, and A.~Genevois, \emph{Property ({NL})
  for group actions on hyperbolic spaces}, With an appendix by A. Sisto.
  arXiv:2212.14292. To appear in Groups Geom. Dyn.

\bibitem[BG84]{brown84}
K.~S. Brown and R.~Geoghegan, \emph{An infinite-dimensional torsion-free
  {$FP_\infty$} group}, Invent. Math. \textbf{77} (1984), no.~2, 367--381.
  \MR{752825}

\bibitem[BNR21]{burillo21}
J.~Burillo, B.~Nucinkis, and L.~Reeves, \emph{An irrational-slope {T}hompson's
  group}, Publ. Mat. \textbf{65} (2021), no.~2, 809--839. \MR{4278765}

\bibitem[BNS87]{bieri87}
R.~Bieri, W.~D. Neumann, and R.~Strebel, \emph{A geometric invariant of
  discrete groups}, Invent. Math. \textbf{90} (1987), no.~3, 451--477.
  \MR{914846}

\bibitem[Bri07]{brin07}
M.~G. Brin, \emph{The algebra of strand splitting. {I}. {A} braided version of
  {T}hompson's group {$V$}}, J. Group Theory \textbf{10} (2007), no.~6,
  757--788. \MR{2364825}

\bibitem[Bro87a]{brown87}
K.~S. Brown, \emph{Finiteness properties of groups}, Proceedings of the
  {N}orthwestern conference on cohomology of groups ({E}vanston, {I}ll., 1985),
  vol.~44, 1987, pp.~45--75. \MR{885095}

\bibitem[Bro87b]{brown87trees}
\bysame, \emph{Trees, valuations, and the {B}ieri-{N}eumann-{S}trebel
  invariant}, Invent. Math. \textbf{90} (1987), no.~3, 479--504. \MR{914847}

\bibitem[BS85]{brin85}
M.~G. Brin and C.~C. Squier, \emph{Groups of piecewise linear homeomorphisms of
  the real line}, Invent. Math. \textbf{79} (1985), no.~3, 485--498.
  \MR{782231}

\bibitem[BS16]{bieristrebel}
R.~Bieri and R.~Strebel, \emph{On groups of {PL}-homeomorphisms of the real
  line}, Mathematical Surveys and Monographs, vol. 215, American Mathematical
  Society, Providence, RI, 2016. \MR{3560537}

\bibitem[Cal07]{calegari:PL}
D.~Calegari, \emph{Stable commutator length in subgroups of {${\rm PL}^+(I)$}},
  Pacific J. Math. \textbf{232} (2007), no.~2, 257--262. \MR{2366352}

\bibitem[Cal09]{scl}
\bysame, \emph{scl}, MSJ Memoirs, vol.~20, Mathematical Society of Japan,
  Tokyo, 2009. \MR{2527432}

\bibitem[CCMT15]{caprace15}
P.-E. Caprace, Y.~Cornulier, N.~Monod, and R.~Tessera, \emph{Amenable
  hyperbolic groups}, J. Eur. Math. Soc. (JEMS) \textbf{17} (2015), no.~11,
  2903--2947. \MR{3420526}

\bibitem[CFP96]{cannon96}
J.~W. Cannon, W.~J. Floyd, and W.~R. Parry, \emph{Introductory notes on
  {R}ichard {T}hompson's groups}, Enseign. Math. (2) \textbf{42} (1996),
  no.~3-4, 215--256. \MR{1426438}

\bibitem[Cle00]{cleary00}
S.~Cleary, \emph{Regular subdivision in {$\bold Z[\frac{1+\sqrt 5}{2}]$}},
  Illinois J. Math. \textbf{44} (2000), no.~3, 453--464. \MR{1772420}

\bibitem[Deh06]{dehornoy06}
P.~Dehornoy, \emph{The group of parenthesized braids}, Adv. Math. \textbf{205}
  (2006), no.~2, 354--409. \MR{2258261}

\bibitem[FFL23]{commutingconjugates}
F.~Fournier-Facio and Y.~Lodha, \emph{Second bounded cohomology of groups
  acting on 1-manifolds and applications to spectrum problems}, Adv. Math.
  \textbf{428} (2023), Paper No. 109162, 42. \MR{4604795}

\bibitem[FFLZ24]{fff_lodha_zar}
F.~Fournier-Facio, Y.~Lodha, and M.~C.~B. Zaremsky, \emph{Braided {T}hompson
  groups with and without quasimorphisms}, Algebr. Geom. Topol. \textbf{24}
  (2024), no.~3, 1601--1622. \MR{4767881}

\bibitem[Gen19]{anthony}
A.~Genevois, \emph{Hyperbolic and cubical rigidities of {T}hompson's group
  {$V$}}, J. Group Theory \textbf{22} (2019), no.~2, 313--345. \MR{3918481}

\bibitem[GMSW01]{ascending}
R.~Geoghegan, M.~L. Mihalik, M.~Sapir, and D.~T. Wise, \emph{Ascending {HNN}
  extensions of finitely generated free groups are {H}opfian}, Bull. London
  Math. Soc. \textbf{33} (2001), no.~3, 292--298. \MR{1817768}

\bibitem[Gro87]{gromov87}
M.~Gromov, \emph{Hyperbolic groups}, Essays in group theory, Math. Sci. Res.
  Inst. Publ., vol.~8, Springer, New York, 1987, pp.~75--263. \MR{919829}

\bibitem[GS17]{golansapir}
G.~Golan and M.~Sapir, \emph{On the stabilizers of finite sets of numbers in
  the {R}. {T}hompson group {$F$}}, Algebra i Analiz \textbf{29} (2017), no.~1,
  70--110. \MR{3660685}

\bibitem[GS19]{divergence}
\bysame, \emph{Divergence functions of {T}hompson groups}, Geom. Dedicata
  \textbf{201} (2019), 227--242. \MR{3978542}

\bibitem[GT]{houghton_hyp_structures}
A.~Genevois and G.~Tournier, \emph{Hyperbolic structures on {H}oughton groups},
  arXiv preprint arXiv:2502.12590.

\bibitem[Hae20]{haettel}
T.~Haettel, \emph{Hyperbolic rigidity of higher rank lattices}, Ann. Sci.
  \'{E}c. Norm. Sup\'{e}r. (4) \textbf{53} (2020), no.~2, 439--468, With an
  appendix by V. Guirardel and C. Horbez. \MR{4094562}

\bibitem[Ham17]{hamann17}
M.~Hamann, \emph{Group actions on metric spaces: fixed points and free
  subgroups}, Abh. Math. Semin. Univ. Hambg. \textbf{87} (2017), no.~2,
  245--263. \MR{3696149}

\bibitem[Hig74]{higman74}
G.~Higman, \emph{Finitely presented infinite simple groups}, Notes on Pure
  Mathematics, No. 8, Australian National University, Department of Pure
  Mathematics, Department of Mathematics, I.A.S., Canberra, 1974. \MR{0376874}

\bibitem[HL]{hydelodha}
J.~Hyde and Y.~Lodha, \emph{Finitely presented simple torsion-free groups in
  the landscape of {R}ichard {T}hompson}, arXiv:2302.04805. To appear in Ann.
  Sci. {\'E}c. Norm. Sup{\'e}r.

\bibitem[HM23]{hydemoore}
J.~Hyde and J.~T. Moore, \emph{Subgroups of {${\rm PL}_+I$} which do not embed
  into {T}hompson's group {$F$}}, Groups Geom. Dyn. \textbf{17} (2023), no.~2,
  533--554. \MR{4584675}

\bibitem[Hou79]{houghton78}
C.~H. Houghton, \emph{The first cohomology of a group with permutation module
  coefficients}, Arch. Math. (Basel) \textbf{31} (1978/79), no.~3, 254--258.
  \MR{521478}

\bibitem[HZ97]{confined:subgroup}
B.~Hartley and A.~E. Zalesski\u{\i}, \emph{Confined subgroups of simple locally
  finite groups and ideals of their group rings}, J. London Math. Soc. (2)
  \textbf{55} (1997), no.~2, 210--230. \MR{1438625}

\bibitem[Leh08]{lehnert08}
J.~Lehnert, \emph{Gruppen von quasi-automorphismen}, Ph.D. thesis,
  Goethe-Universität, Frankfurt am Main, 2008.

\bibitem[LM16]{lodha16}
Y.~Lodha and J.~T. Moore, \emph{A nonamenable finitely presented group of
  piecewise projective homeomorphisms}, Groups Geom. Dyn. \textbf{10} (2016),
  no.~1, 177--200. \MR{3460335}

\bibitem[Lod20a]{lodha20}
Y.~Lodha, \emph{Coherent actions by homeomorphisms on the real line or an
  interval}, Israel J. Math. \textbf{235} (2020), no.~1, 183--212. \MR{4068782}

\bibitem[Lod20b]{lodhaFinfty}
\bysame, \emph{A nonamenable type {$\rm F_\infty$} group of piecewise
  projective homeomorphisms}, J. Topol. \textbf{13} (2020), no.~4, 1767--1838.
  \MR{4186144}

\bibitem[LZ23]{lodha23}
Y.~Lodha and M.~C.~B. Zaremsky, \emph{The {BNSR}-invariants of the
  {L}odha-{M}oore groups, and an exotic simple group of type {${\rm
  F}_\infty$}}, Math. Proc. Cambridge Philos. Soc. \textbf{174} (2023), no.~1,
  25--48. \MR{4523108}

\bibitem[Man08]{manning}
J.~F. Manning, \emph{Actions of certain arithmetic groups on {G}romov
  hyperbolic spaces}, Algebr. Geom. Topol. \textbf{8} (2008), no.~3,
  1371--1402. \MR{2443247}

\bibitem[MNSR24]{MNSR}
L.~Molyneux, B.~Nucinkis, and Y.~Santos~Rego, \emph{The sigma invariants for
  the golden mean {T}hompson group}, New York J. Math. \textbf{30} (2024),
  532--549. \MR{4729606}

\bibitem[MO19]{erratum}
A.~Minasyan and D.~Osin, \emph{Correction to: {A}cylindrical hyperbolicity of
  groups acting on trees}, Math. Ann. \textbf{373} (2019), no.~1-2, 895--900.
  \MR{3968890}

\bibitem[NSJG18]{nucinkis18}
B.~E.~A. Nucinkis and S.~St. John-Green, \emph{Quasi-automorphisms of the
  infinite rooted 2-edge-coloured binary tree}, Groups Geom. Dyn. \textbf{12}
  (2018), no.~2, 529--570. \MR{3813202}

\bibitem[Ser77]{serre}
J.-P. Serre, \emph{Arbres, amalgames, {${\rm SL}_{2}$}}, Soci\'{e}t\'{e}
  Math\'{e}matique de France, Paris, 1977, Avec un sommaire anglais,
  R\'{e}dig\'{e} avec la collaboration de Hyman Bass, Ast\'{e}risque, No. 46.
  \MR{0476875}

\bibitem[She]{divergencen}
X.~Sheng, \emph{Divergence properties of the generalised {T}hompson groups and
  the braided-{T}hompson groups}, arXiv:2106.15571. To appear in Transform.
  Groups.

\bibitem[Ste92]{stein92}
M.~Stein, \emph{Groups of piecewise linear homeomorphisms}, Trans. Amer. Math.
  Soc. \textbf{332} (1992), no.~2, 477--514. \MR{1094555}

\bibitem[SZ21]{spahn21}
R.~Spahn and M.~C.~B. Zaremsky, \emph{The {BNSR}-invariants of the {S}tein
  group {$F_{2,3}$}}, J. Group Theory \textbf{24} (2021), no.~6, 1149--1162.
  \MR{4334008}

\bibitem[Win22]{winstone}
N.~Winstone, \emph{{B}ieri-{S}trebel groups with irrational slopes}, Ph.D.
  thesis, Royal Holloway, University of London, 2022.

\bibitem[Zar16]{zaremsky16}
M.~C.~B. Zaremsky, \emph{H{NN} decompositions of the {L}odha-{M}oore groups,
  and topological applications}, J. Topol. Anal. \textbf{8} (2016), no.~4,
  627--653. \MR{3545015}

\bibitem[Zar17]{zaremsky17F_n}
\bysame, \emph{On the {$\Sigma$}-invariants of generalized {T}hompson groups
  and {H}oughton groups}, Int. Math. Res. Not. IMRN (2017), no.~19, 5861--5896.
  \MR{3741884}

\bibitem[Zar20]{zaremsky20}
\bysame, \emph{The {BNSR}-invariants of the {H}oughton groups, concluded},
  Proc. Edinb. Math. Soc. (2) \textbf{63} (2020), no.~1, 1--11. \MR{4054771}

\end{thebibliography}

\end{document}